\documentclass[11pt,twoside]{article}

\usepackage{jmlr2e_modified}

\usepackage{amsfonts,amsmath,amssymb}
\usepackage{graphics,graphicx}
\usepackage{cases}

\newcommand{\GamHat}{\ensuremath{\widehat{\Gamma}}}
\newcommand{\gamhat}{\ensuremath{\widehat{\gamma}}}
\newcommand{\SigHat}{\ensuremath{\widehat{\Sigma}}}
\newcommand{\ThetaHat}{\ensuremath{\widehat{\Theta}}}
\newcommand{\betahat}{\ensuremath{\widehat{\beta}}}
\newcommand{\betastar}{\ensuremath{\beta^*}}

\newcommand{\Loss}{\ensuremath{\mathcal{L}}}

\newcommand{\nutil}{\ensuremath{\widetilde{\nu}}}
\newcommand{\ball}{\ensuremath{\mathbb{B}}}

\newcommand{\betatil}{\ensuremath{\widetilde{\beta}}}

\newcommand{\Lossbar}{\ensuremath{\widebar{\Loss}}}
\newcommand{\xhat}{\ensuremath{\widehat{x}}}

\newcommand{\epsstat}{\ensuremath{\epsilon_{\mbox{\tiny{stat}}}}}
\newcommand{\epsbar}{\ensuremath{\widebar{\epsilon}}}

\newcommand{\rhotil}{\ensuremath{\widetilde{\rho}}}
\newcommand{\scriptE}{\ensuremath{\mathcal{E}_\numobs}}

\newcommand{\scriptF}{\ensuremath{\mathcal{F}}}
\newcommand{\Deltatil}{\ensuremath{\widetilde{\Delta}}}
\newcommand{\Thetastar}{\ensuremath{\Theta^*}}
\newcommand{\Thetatil}{\ensuremath{\widetilde{\Theta}}}
\newcommand{\scriptT}{\ensuremath{{\mathcal{T}}}}

\newcommand{\gtil}{\ensuremath{\widetilde{g}}}
\newcommand{\ghat}{\ensuremath{\widehat{g}}}

\newcommand{\opnorm}[1]{\left|\!\left|\!\left|{#1}\right|\!\right|\!\right|}
\newcommand{\Cov}{\ensuremath{\operatorname{Cov}}}

\newcommand{\E}{\ensuremath{\mathbb{E}}}
\newcommand{\mprob}{\ensuremath{\mathbb{P}}}

\newcommand{\numobs}{\ensuremath{n}}

\newcommand{\pdim}{\ensuremath{p}}

\newcommand{\SIDE}{\ensuremath{g}}
\newcommand{\Zdata}{\ensuremath{Z_1^\numobs}}
\newcommand{\PopLoss}{\ensuremath{\Loss}}
\newcommand{\EmpLoss}{\ensuremath{\Loss_\numobs}}
\newcommand{\myrho}{\ensuremath{\rho_\lambda}}
\newcommand{\mupar}{\ensuremath{\mu}}
\newcommand{\doubrho}{\ensuremath{\rho_{\lambda, \mupar}}}

\newcommand{\scriptTBar}{\ensuremath{\widebar{\scriptT}}}
\newcommand{\Zspace}{\ensuremath{\mathcal{Z}}}
\newcommand{\betait}[1]{\ensuremath{\beta^{#1}}}
\newcommand{\iter}{\ensuremath{t}}
\newcommand{\kdim}{\ensuremath{k}}
\newcommand{\Aevent}{\ensuremath{\mathcal{A}}}
\newcommand{\frobnorm}[1]{\ensuremath{\matsnorm{#1}{F}}}
\newcommand{\myvec}{\ensuremath{\mbox{vec}}}
\newcommand{\Aregion}{\ensuremath{\mathbb{A}}}
\newcommand{\HACK}{\ensuremath{\varphi(\numobs, \pdim, \kdim)}}
\newcommand{\SIDESPEC}{\ensuremath{\SIDE_{\lambda, \mupar}}}
\newcommand{\MartinPingPong}{\ensuremath{\mathbb{V}}}
\newcommand{\Annoying}{\mathcal{E}}
\newcommand{\EmpLossBar}{\ensuremath{\widebar{\mathcal{L}}_\numobs}}
\newcommand{\EmpLossBarSub}{\ensuremath{\bar{\mathcal{L}}_\numobs}}
\newcommand{\CovX}{\ensuremath{\Sigma}}
\newcommand{\SQUIRREL}{\psi(\numobs, \pdim, \epsilon)}
\newcommand{\CHICKA}{\ensuremath{\upsilon(k, p, n)}} 

 \newtheorem{mytheorem}{Theorem}
 \newtheorem{mycorollary}{Corollary}
\newtheorem{assumption}{Assumption}

\newtheorem{mylemma}{Lemma}
 \newtheorem{myproposition}{Proposition}
\newlength{\widebarargwidth}
\newlength{\widebarargheight}
\newlength{\widebarargdepth}
\DeclareRobustCommand{\widebar}[1]{%
  \settowidth{\widebarargwidth}{\ensuremath{#1}}%
  \settoheight{\widebarargheight}{\ensuremath{#1}}%
  \settodepth{\widebarargdepth}{\ensuremath{#1}}%
  \addtolength{\widebarargwidth}{-0.3\widebarargheight}%
  \addtolength{\widebarargwidth}{-0.3\widebarargdepth}%
  \makebox[0pt][l]{\hspace{0.3\widebarargheight}%
    \hspace{0.3\widebarargdepth}%
    \addtolength{\widebarargheight}{0.3ex}%
    \rule[\widebarargheight]{0.95\widebarargwidth}{0.1ex}}%
  {#1}}

\newcommand{\matsnorm}[2]{|\!|\!| #1 | \! | \!|_{{#2}}}

\newenvironment{carlist}
 {\begin{list}{$\bullet$}
 {\setlength{\topsep}{0in} \setlength{\partopsep}{0in}
  \setlength{\parsep}{0in} \setlength{\itemsep}{\parskip}
  \setlength{\leftmargin}{0.07in} \setlength{\rightmargin}{0.08in}
  \setlength{\listparindent}{0in} \setlength{\labelwidth}{0.08in}
  \setlength{\labelsep}{0.1in} \setlength{\itemindent}{0in}}}
 {\end{list}}

\newcommand{\bcar}{\begin{carlist}}
\newcommand{\ecar}{\end{carlist}}

\long\def\comment#1{}

\makeatletter
\def\@cite#1#2{[\if@tempswa #2 \fi #1]}
\makeatother

\makeatletter
\long\def\barenote#1{
    \insert\footins{\footnotesize
    \interlinepenalty\interfootnotelinepenalty 
    \splittopskip\footnotesep
    \splitmaxdepth \dp\strutbox \floatingpenalty \@MM
    \hsize\columnwidth \@parboxrestore
    {\rule{\z@}{\footnotesep}\ignorespaces
      #1\strut}}}
\makeatother
\newcommand{\Ind}{\ensuremath{\mathbb{I\,}}}

\newcommand{\bit}{\begin{itemize}}
\newcommand{\eit}{\end{itemize}}
\newcommand{\ben}{\begin{enumerate}}
\newcommand{\een}{\end{enumerate}}

\newcommand{\bear}{\begin{eqnarray}}
\newcommand{\eear}{\end{eqnarray}}

\newcommand{\etabar}{\ensuremath{\bar{\eta}}}

\newcommand{\supp}{\ensuremath{\operatorname{supp}}}

\newcommand{\order}{{\mathcal{O}}}

\newcommand{\inprod}[2]{\ensuremath{\langle #1 , \, #2 \rangle}}

\newcommand{\Exs}{\ensuremath{{\mathbb{E}}}}

\newcommand{\beq}{\begin{quotation}}
\newcommand{\enq}{\end{quotation}}

\newcommand{\estart}{\begin{equation}}
\newcommand{\eend}{\end{equation}}

\newcommand{\widgraph}[2]{\includegraphics[keepaspectratio,width=#1]{#2}}

\newcommand{\defn}{\ensuremath{:  =}}

\newcommand{\bec}{\begin{center}}
\newcommand{\enc}{\end{center}}

\newcommand{\beit}{\begin{itemize}}
\newcommand{\enit}{\end{itemize}}

\newcommand{\been}{\begin{enumerate}}
\newcommand{\enen}{\end{enumerate}}

\newcommand{\comsl}{\begin{slide}}
\newcommand{\comspor}{\begin{slide*}}

\newcommand{\comsld}[2]{\begin{slide}[#1,#2]}
\newcommand{\comspord}[2]{\begin{slide*}[#1,#2]}

\newcommand{\mendsl}{\end{slide}}
\newcommand{\mendspo}{\end{slide*}}

\newcommand{\real}{\ensuremath{{\mathbb{R}}}}

\DeclareMathOperator{\var}{var}

\DeclareMathOperator{\cov}{cov}

\DeclareMathOperator{\trace}{trace}

\DeclareMathOperator{\sign}{sign}

\jmlrheading{1}{2014}{1-56}{4/14; Revised 10/14}{}{Po-Ling Loh and Martin J.\ Wainwright}

\ShortHeadings{Local Optima of Nonconvex $M$-estimators}{Loh and Wainwright}
\firstpageno{1}

\begin{document}

\title{Regularized $M$-estimators with Nonconvexity: Statistical and Algorithmic Theory for Local Optima}

\author{%
\name Po-Ling Loh \email loh@wharton.upenn.edu \\
\addr Department of Statistics \\
The Wharton School\\
466 Jon M.\ Huntsman Hall \\
3730 Walnut Street \\
Philadelphia, PA 19104, USA
\AND
\name Martin J.\ Wainwright \email wainwrig@berkeley.edu\\
\addr Departments of EECS and Statistics\\
263 Cory Hall \\
University of California\\
Berkeley, CA 94720, USA
}

\editor{Nicolai Meinshausen}

\maketitle

\begin{abstract}%
We provide novel theoretical results regarding local optima of
regularized $M$-estimators, allowing for nonconvexity in both loss and
penalty functions. Under restricted strong convexity on the loss and
suitable regularity conditions on the penalty, we prove that \emph{any
  stationary point} of the composite objective function will lie
within statistical precision of the underlying parameter vector. Our
theory covers many nonconvex objective functions of interest,
including the corrected Lasso for errors-in-variables linear models;
regression for generalized linear models with nonconvex penalties such
as SCAD, MCP, and capped-$\ell_1$; and high-dimensional graphical
model estimation. We quantify statistical accuracy
by providing bounds on the $\ell_1$-, $\ell_2$-, and prediction
error between stationary points and the population-level optimum. We also propose a simple modification of composite gradient
descent that may be used to obtain a near-global optimum within
statistical precision $\epsstat$ in $\log(1/\epsstat)$ steps, which is
the fastest possible rate of any first-order method.  We provide
simulation studies illustrating the sharpness of our theoretical
results.

\end{abstract}


%%%%%%%%%%%

\section{Introduction}

Although recent years have brought about a flurry of work on
optimization of convex functions, optimizing nonconvex functions is in
general computationally intractable~\citep{NesNem87,Vav95}. Nonconvex
functions may possess local optima that are not global optima, and
iterative methods such as gradient or coordinate descent may terminate
undesirably in local optima. Unfortunately, standard statistical
results for nonconvex $M$-estimators often only provide guarantees for
\emph{global} optima. This leads to a significant gap between theory
and practice, since computing global optima---or even near-global
optima---in an efficient manner may be extremely difficult in
practice. Nonetheless, empirical studies have shown that local optima
of various nonconvex $M$-estimators arising in statistical problems
appear to be well-behaved (e.g.,~\citealp{BreHua11}). This type of
observation is the starting point of our work.

A key insight is that nonconvex functions occurring in statistics
are not constructed adversarially, so that ``good behavior'' might be
expected in practice. Our recent work~\citep{LohWai11a} confirmed this
intuition for one specific case: a modified version of the Lasso
applicable to errors-in-variables regression.  Although the Hessian of
the modified objective has many negative eigenvalues in the
high-dimensional setting, the objective function resembles a strongly
convex function when restricted to a cone set that includes the stationary points of the objective. This allows us to establish bounds on
the statistical and optimization error.

Our current paper is framed in a more general setting, and we focus on
various $M$-estimators coupled with (nonconvex) regularizers of
interest. On the statistical side, we establish bounds on the distance
between \emph{any local optimum} of the empirical objective and the
unique minimizer of the population risk. Although the nonconvex
functions may possess multiple local optima (as demonstrated in
simulations), our theoretical results show that all local optima are
essentially as good as a global optima from a statistical
perspective. The results presented here subsume our previous
work~\citep{LohWai11a}, and our present proof techniques are much more
direct.

Our theory also sheds new light on a recent line of work involving the
nonconvex SCAD and MCP regularizers~\citep{FanLi01, BreHua11, Zha12,
  ZhaZha12}. Various methods previously proposed for nonconvex
optimization include local quadratic approximation
(LQA)~\citep{FanLi01}, minorization-maximization (MM)~\citep{HunLi05},
local linear approximation (LLA)~\citep{ZouLi08}, and coordinate
descent~\citep{BreHua11, MazEtal11}.  However, these methods may
terminate in local optima, which were not previously known to be
well-behaved. In a recent paper, \cite{ZhaZha12} provided statistical
guarantees for global optima of least-squares linear regression with
nonconvex penalties and showed that gradient descent starting from a
Lasso solution would terminate in specific local
minima. \cite{FanEtal13} also showed that if the LLA algorithm is
initialized at a Lasso optimum satisfying certain properties, the
two-stage procedure produces an oracle solution for various nonconvex
penalties. Finally, \cite{CheGu13} showed that specific local optima
of nonconvex regularized least-squares problems are stable, so
optimization algorithms initialized sufficiently close by will converge
to the same optima. See the survey paper~\citep{ZhaZha12} for a more
complete overview of related work.

In contrast, our paper is the first to establish
appropriate regularity conditions under which \emph{all stationary points} (including both local and global optima) lie within a small ball of the population-level
minimum. Thus, standard first-order methods such as projected and
composite gradient descent~\citep{Nes07} will converge to stationary points that
lie within statistical error of the truth, eliminating the need for
specially designed optimization algorithms that converge to specific
local optima. Our work provides an important contribution to the growing literature on the tradeoff between statistical accuracy and optimization efficiency in high-dimensional problems, establishing that certain types of nonconvex $M$-estimators arising in statistical problems possess stationary points that both enjoy strong statistical guarantees and may be located efficiently. For a higher-level description of contemporary problems involving statistical and optimization tradeoffs, see~\cite{Wai14} and the references cited therein.

Figure~\ref{FigTeaser} provides an illustration of the
type of behavior explained by the theory in this paper. Panel (a)
shows the behavior of composite gradient descent for a form of logistic
regression with the nonconvex SCAD~\citep{FanLi01} as a regularizer:
the red curve shows the \emph{statistical error}, namely the
$\ell_2$-norm of the difference between a stationary point and the
underlying true regression vector, and the blue curve shows the
\emph{optimization error}, meaning the difference between the iterates
and a given global optimum.  As shown by the blue
curves, this problem possesses multiple local optima, since the algorithm converges
to different final points depending on the initialization.  However,
as shown by the red curves, the statistical error of each
local optimum is very low, so they are all essentially comparable
from a statistical point of view.
\begin{figure}[h]
\begin{center}
\begin{tabular}{ccc}
      \widgraph{.45\textwidth}{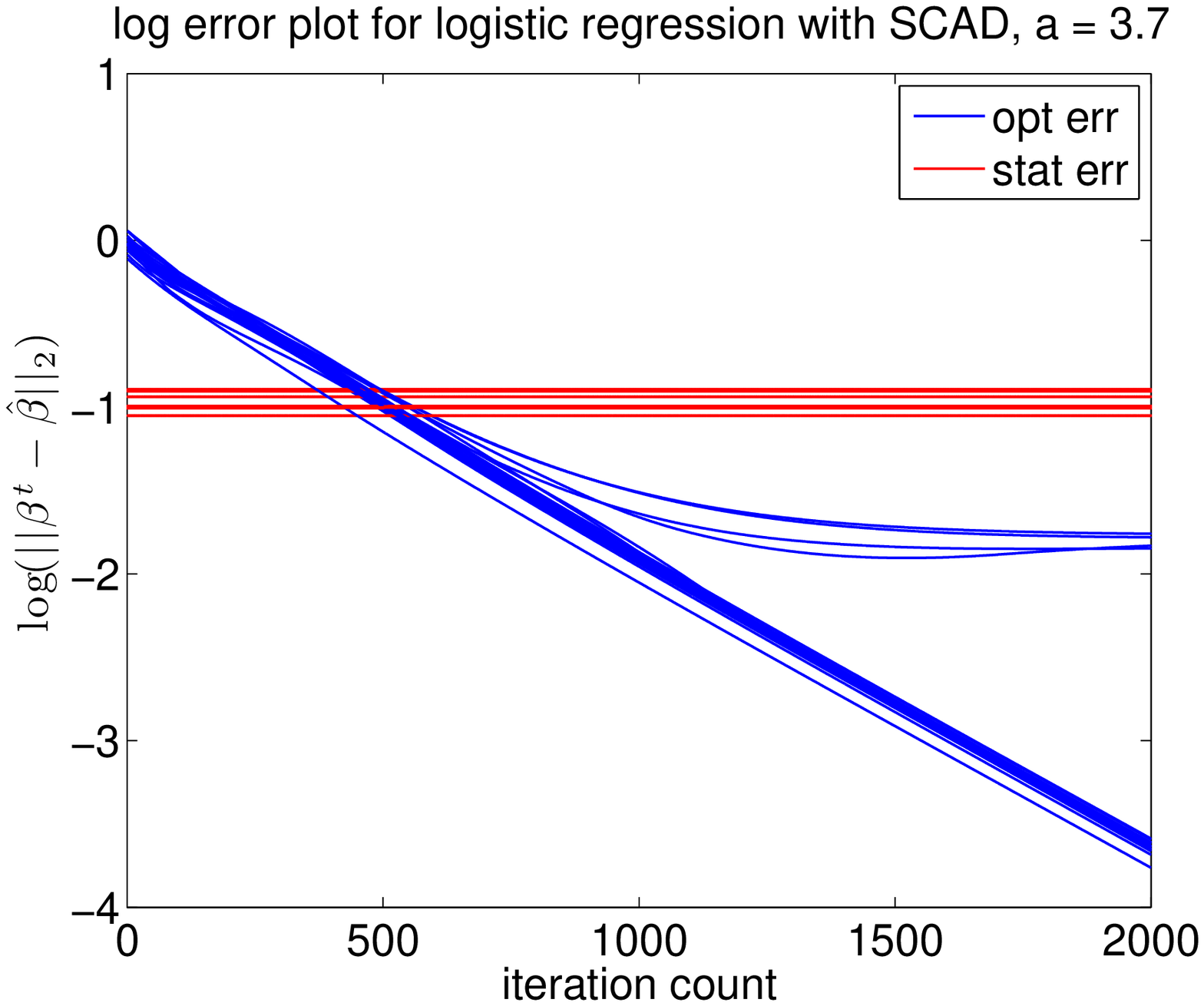} & &
      \widgraph{.45\textwidth}{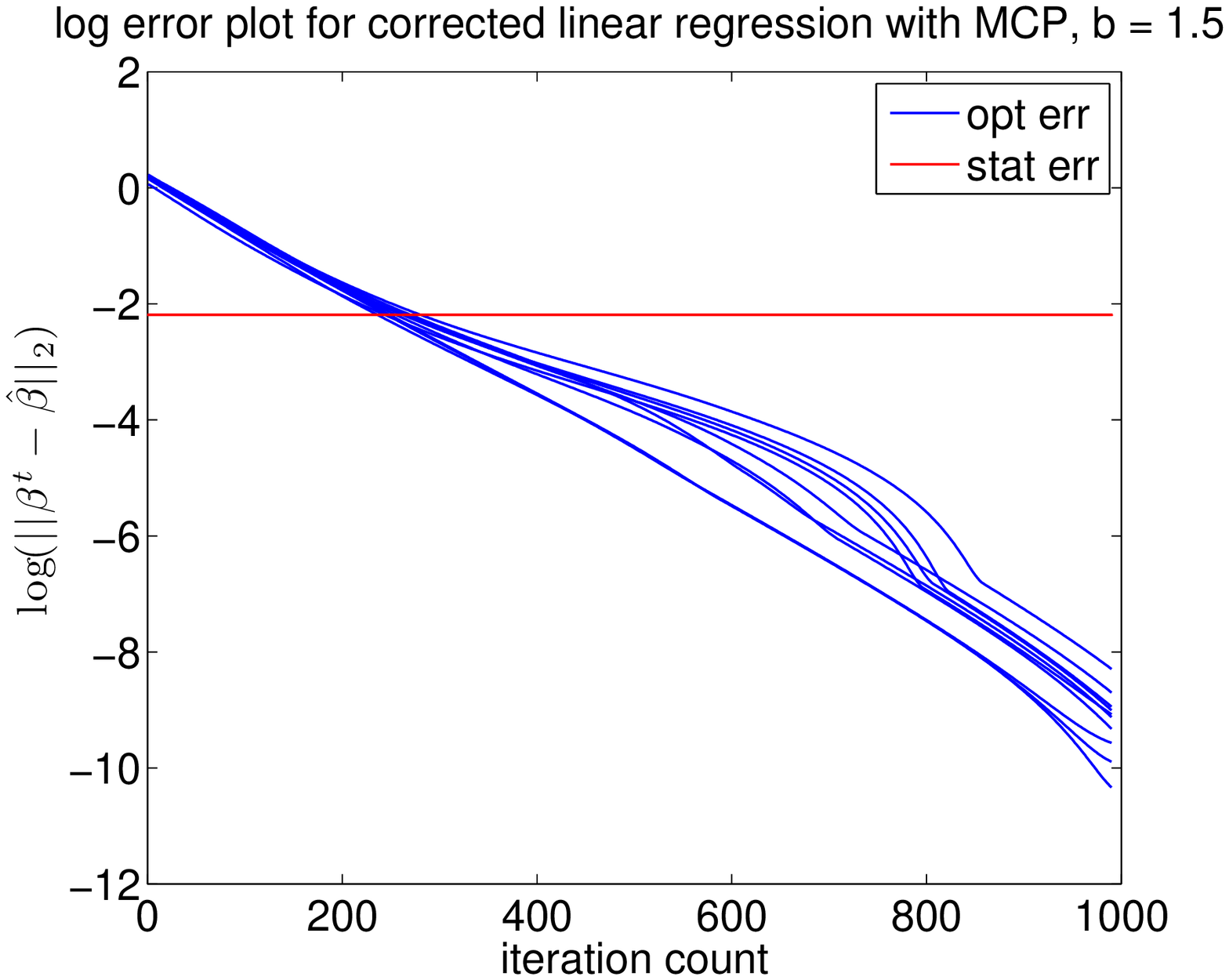} \\
(a) & & (b)
\end{tabular}
\caption{Plots of the optimization error (blue curves) and statistical
  error (red curves) for a modified form of composite gradient
  descent, applicable to problems that may involve nonconvex cost
  functions and regularizers. (a) Plots for logistic regression with
  the nonconvex SCAD regularizer. (b) Plots for a corrected form of
  least squares (a nonconvex quadratic program) with the nonconvex MCP
  regularizer.}
\label{FigTeaser}
\end{center}
\end{figure}
Panel (b) exhibits the same behavior for a problem in which both the
cost function (a corrected form of least-squares suitable for missing
data, described in~\citealp{LohWai12}) and the regularizer (the MCP
function, described in~\citealp{Zha12}) are nonconvex.  Nonetheless, as guaranteed by
our theory, we still see the same qualitative behavior of the
statistical and optimization error.  Moreover, our theory also
predicts the geometric convergence rates that are apparent in these
plots. More precisely, under the same sufficient conditions for
statistical consistency, we show that a modified form of composite
gradient descent only requires $\log(1/\epsstat)$ steps to achieve a
solution that is accurate up to the statistical precision $\epsstat$,
which is the rate expected for \emph{strongly convex}
functions. Furthermore, our techniques are more generally applicable
than the methods proposed by previous authors and are not restricted
to least-squares or even convex loss functions.

While our paper was under review after its initial arXiv
posting~\citep{LohWai13}, we became aware of an independent line of
related work by~\cite{WanEtal13}. Our contributions are substantially
different, in that we provide sufficient conditions guaranteeing
statistical consistency for \emph{all} local optima, whereas their
work is only concerned with establishing good behavior of successive
iterates along a certain path-following algorithm. In addition, our
techniques are applicable even to regularizers that do not satisfy
smoothness constraints on the entire positive axis (such as
capped-$\ell_1$). Finally, we provide rigorous proofs showing the
applicability of our sufficient condition on the loss function to a
broad class of generalized linear models, whereas the applicability of
their sparse eigenvalue condition to such objectives was not
established.

The remainder of the paper is organized as follows. In Section 2, we
set up basic notation and provide background on nonconvex regularizers
and loss functions of interest. In Section 3, we provide our main
theoretical results, including bounds on $\ell_1$-, $\ell_2$-, and
prediction error, and also state corollaries for special
cases. Section 4 contains a modification of composite gradient descent
that may be used to obtain near-global optima and includes
theoretical results establishing the linear convergence of our
optimization algorithm. Section 5 supplies the results of various
simulations. Proofs are contained in the Appendix. We note that a
preliminary form of the results given here, without any proofs or
algorithmic details, was presented at the NIPS
conference~\citep{LohWai13NIPS}. \\

\textbf{Notation:} For functions $f(n)$ and $g(n)$, we write $f(n) \precsim g(n)$ to mean
that \mbox{$f(n) \le c g(n)$} for some universal constant $c \in (0,
\infty)$, and similarly, $f(n) \succsim g(n)$ when \mbox{$f(n) \ge c'
  g(n)$} for some universal constant $c' \in (0, \infty)$. We write
$f(n) \asymp g(n)$ when $f(n) \precsim g(n)$ and $f(n) \succsim g(n)$
hold simultaneously. For a vector $v \in \real^p$ and a subset $S
\subseteq \{1, \dots, p\}$, we write $v_S \in \real^S$ to denote the
vector $v$ restricted to $S$. For a matrix $M$, we write
$\opnorm{M}_2$ and $\opnorm{M}_F$ to denote the spectral and Frobenius
norms, respectively, and write $\opnorm{M}_{\max} \defn \max_{i,j}
|m_{ij}|$ to denote the elementwise $\ell_\infty$-norm of $M$. For a
function $h: \real^p \rightarrow \real$, we write $\nabla h$ to denote
a gradient or subgradient, if it exists. Finally, for $q, r > 0$, let
$\ball_q(r)$ denote the $\ell_q$-ball of radius $r$ centered around
0. We use the term ``with high probability" (w.h.p.) to refer to
events that occur with probability tending to 1 as $n, p, k
\rightarrow \infty$.  This is a loose requirement, but we will always
take care to write out the expression for the probability explicitly
(up to constant factors) in the formal statements of our theorems and
corollaries below.

%%%%%%%%%%%

\section{Problem Formulation}

In this section, we develop some general theory for regularized
$M$-estimators. We begin by establishing our notation and basic
assumptions, before turning to the class of nonconvex regularizers and
nonconvex loss functions to be covered in this paper.

\subsection{Background}
Given a collection of $\numobs$ samples $\Zdata = \{Z_1, \ldots,
Z_\numobs\}$, drawn from a marginal distribution $\mprob$ over a space
$\Zspace$, consider a loss function $\EmpLoss: \real^\pdim \times
(\Zspace)^\numobs \rightarrow \real$.  The value $\EmpLoss(\beta;
\Zdata)$ serves as a measure of the ``fit'' between a parameter vector
$\beta \in \real^\pdim$ and the observed data.  This empirical loss
function should be viewed as a surrogate to the \emph{population risk
  function} $\Loss:\real^\pdim \rightarrow \real$, given by
\begin{align*}
\PopLoss(\beta) & \defn \Exs_{Z} \big[ \EmpLoss(\beta; \Zdata) \big].
\end{align*}
Our goal is to estimate the parameter vector $\betastar \defn \arg
\min \limits_{\beta \in \real^\pdim} \PopLoss(\beta)$ that minimizes
the population risk, assumed to be unique.

To this end, we consider a regularized $M$-estimator of the form
\begin{align}
\label{EqnNonconvexReg}
\betahat & \in \arg \min_{\SIDE(\beta) \leq R, \; \beta \in \Omega} \left \{
\EmpLoss(\beta; \Zdata) + \myrho(\beta) \right \},
\end{align}
where $\myrho: \real^\pdim \rightarrow \real$ is a \emph{regularizer},
depending on a tuning parameter $\lambda > 0$, which serves to enforce
a certain type of structure on the solution. Here, $R > 0$ is another
tuning parameter that much be chosen carefully to make $\betastar$ a
feasible point. In all cases, we consider regularizers that are
separable across coordinates, and with a slight abuse of notation, we
write
\begin{equation*}
\myrho(\beta) = \sum_{j=1}^\pdim \myrho(\beta_j).
\end{equation*}
Our theory allows for possible nonconvexity in \emph{both} the loss
function $\EmpLoss$ and the regularizer $\myrho$.  Due to this
potential nonconvexity, our $M$-estimator also includes a side
constraint \mbox{$\SIDE: \real^\pdim \rightarrow \real_+$,} which we
require to be a convex function satisfying the lower bound
\mbox{$\SIDE(\beta) \geq \|\beta\|_1$} for all $\beta \in
\real^\pdim$.  Consequently, any feasible point for the optimization
problem~\eqref{EqnNonconvexReg} satisfies the constraint $\|\beta\|_1
\leq R$, and as long as the empirical loss and regularizer are
continuous, the Weierstrass extreme value theorem guarantees that a
global minimum $\betahat$ exists. Finally, our theory also allows for
an additional side constraint of the form $\beta \in \Omega$, where
$\Omega$ is some convex set containing $\betastar$. For the graphical
Lasso considered in Section~\ref{SecGLasso}, we take $\Omega =
\mathcal{S}_+$ to be the set of positive semidefinite matrices; in
settings where such an additional condition is extraneous, we simply
set $\Omega = \real^p$.

\subsection{Nonconvex Regularizers}
\label{SecNonconvexRegExas}

We now state and discuss conditions on the regularizer,
defined in terms of a univariate function $\myrho: \real \rightarrow
\real$.
\begin{assumption}
\label{AsRho}
\mbox{}
\begin{enumerate}
\item[(i)] The function $\myrho$ satisfies $\myrho(0) = 0$ and is
  symmetric around zero (i.e., $\myrho(t) = \myrho(-t)$ for all $t \in
  \real$).
\item[(ii)] On the nonnegative real line, the function $\myrho$ is
  nondecreasing.
\item[(iii)] For $t > 0$, the function $t \mapsto \frac{\myrho(t)}{t}$
  is nonincreasing in $t$.
\item[(iv)] The function $\myrho$ is differentiable for all $t \neq
  0$ and subdifferentiable at $t = 0$, with $\lim_{t \rightarrow 0^+} \rho_\lambda'(t) = \lambda L$.
\item[(v)] There exists $\mupar > 0$ such that
  $\doubrho(t) \defn \myrho(t) + \frac{\mu}{2} t^2$ is convex.
\end{enumerate}
\end{assumption}

It is instructive to compare the conditions of Assumption~\ref{AsRho}
to similar conditions previously proposed in literature. Conditions
(i)--(iii) are the same as those proposed in~\cite{ZhaZha12}, except we omit the extraneous condition of
subadditivity (cf.\ Lemma 1 of~\citealp{CheGu13}). Such
conditions are relatively mild and are satisfied for a wide variety of
regularizers. Condition (iv) restricts the class of penalties by
excluding regularizers such as the bridge ($\ell_q$-) penalty, which
has infinite derivative at 0; and the capped-$\ell_1$ penalty, which
has points of non-differentiability on the positive real
line. However, one may check that if $\rho_\lambda$ has an unbounded
derivative at zero, then $\betatil = 0$ is \emph{always} a local
optimum of the composite objective~\eqref{EqnNonconvexReg}, so there
is no hope for $\|\betatil - \betastar\|_2$ to be vanishingly
small. Condition (v), known as \emph{weak convexity}~\citep{Via82},
also appears in~\cite{CheGu13} and is a type of curvature
constraint that controls the level of nonconvexity of
$\myrho$. Although this condition is satisfied by many regularizers of
interest, it is again not satisfied by capped-$\ell_1$ for any $\mu >
0$. For details on how our arguments may be modified to handle the
more tricky capped-$\ell_1$ penalty, see Appendix~\ref{AppCapped}.

Nonetheless, many regularizers that are commonly used in practice
satisfy all the conditions in Assumption~\ref{AsRho}.  It is easy
to see that the standard $\ell_1$-norm \mbox{$\myrho(\beta) = \lambda
  \|\beta\|_1$} satisfies these conditions.  More exotic functions
have been studied in a line of past work on nonconvex regularization,
and we provide a few examples here: \\

\textbf{SCAD penalty:}  This penalty,
due to~\cite{FanLi01}, takes the form
\begin{align}
 \label{EqnSCADdefn}
\myrho(t) & \defn
\begin{cases}
\lambda |t|, & \mbox{for $|t| \le \lambda$,} \\
-(t^2 - 2a\lambda |t| + \lambda^2)/(2(a-1)), & \mbox{for $\lambda <
  |t| \le a \lambda$,} \\
(a+1)\lambda^2/2, & \mbox{ for $|t| > a\lambda$},
\end{cases}
\end{align}
where $a > 2$ is a fixed parameter.  As verified in
Lemma~\ref{LemSCAD} of Appendix~\ref{AppVerify}, the SCAD penalty
satisfies the conditions of Assumption~\ref{AsRho} with $L = 1$ and
$\mu = \frac{1}{a-1}$. \\

\textbf{MCP regularizer:} This penalty, due to~\cite{Zha12},
takes the form
\begin{align}
 \label{EqnMCPdefn}
\myrho(t) & \defn \sign(t) \, \lambda \cdot \int_0^{|t|} \left(1 -
\frac{z}{\lambda b}\right)_+ dz,
\end{align}
where $b > 0$ is a fixed parameter.  As verified in Lemma~\ref{LemMCP}
in Appendix~\ref{AppVerify}, the MCP regularizer satisfies the
conditions of Assumption~\ref{AsRho} with $L = 1$ and $\mu =
\frac{1}{b}$.

%%%%%%%%%%%%%%%%%%%%%%%%%%%%%%%%%%%%%%%%%%%%%%%%%%%%%%%%%%%%%%%%%%%%%%%%%%

\subsection{Nonconvex Loss Functions and Restricted Strong Convexity}

Throughout this paper, we require the loss function $\EmpLoss$ to be
differentiable, but we do not require it to be convex. Instead, we
impose a weaker condition known as restricted strong convexity (RSC).
Such conditions have been discussed in previous
literature~\citep{NegRavWaiYu12, AgaEtal12}, and involve a lower bound
on the remainder in the first-order Taylor expansion of $\EmpLoss$. In
particular, our main statistical result is based on the following RSC
condition:
\begin{subnumcases}{
\label{EqnRSC}
 \inprod{\nabla \EmpLoss(\betastar + \Delta) - \nabla
   \EmpLoss(\betastar)}{\Delta} \geq}
\label{EqnLocalRSC}
\alpha_1 \|\Delta\|_2^2 - \tau_1 \frac{\log p}{n} \|\Delta\|_1^2, \quad \quad
\forall \|\Delta\|_2 \leq 1, \\
\label{EqnL2Bd}
\alpha_2 \|\Delta\|_2 - \tau_2 \sqrt{\frac{\log p}{n}} \|\Delta\|_1, \quad
\forall \|\Delta\|_2 \geq 1,
\end{subnumcases}
where the $\alpha_j$'s are strictly positive constants and the
$\tau_j$'s are nonnegative constants.

To understand this condition, note that if $\EmpLoss$ were actually
strongly convex, then both these RSC inequalities would hold with
$\alpha_1 = \alpha_2 > 0$ and $\tau_1 = \tau_2 = 0$.  However, in the
high-dimensional setting ($\pdim \gg \numobs)$, the empirical loss
$\EmpLoss$ will not in general be strongly convex or even convex, but
the RSC condition may still hold with strictly positive $(\alpha_j,
\tau_j)$.  In fact, if $\EmpLoss$ is convex (but not strongly convex),
the left-hand expression in~\eqref{EqnRSC} is always
nonnegative, so~\eqref{EqnLocalRSC} and~\eqref{EqnL2Bd}
hold trivially for $\frac{\|\Delta\|_1}{\|\Delta\|_2} \ge
\sqrt{\frac{\alpha_1 n}{\tau_1 \log p}}$ and
$\frac{\|\Delta\|_1}{\|\Delta\|_2} \ge \frac{\alpha_2}{\tau_2}
\sqrt{\frac{n}{\log p}}$, respectively. Hence, the RSC inequalities
only enforce a type of strong convexity condition over a cone of the
form $\left\{\frac{\|\Delta\|_1}{\|\Delta\|_2} \le c
\sqrt{\frac{n}{\log p}}\right\}$.

It is important to note that the class of functions satisfying RSC
conditions of this type is much larger than the class of convex
functions; for instance, our own past work~\citep{LohWai11a} exhibits
a large family of nonconvex quadratic functions that satisfy the
condition (see Section~\ref{SecCorrLinear} below for further
discussion). Furthermore, note that we have stated two separate RSC
inequalities~\eqref{EqnRSC} for different ranges of $\|\Delta\|_2$,
unlike in past work~\citep{NegRavWaiYu12, AgaEtal12, LohWai11a}. As
illustrated in the corollaries of Sections~\ref{SecGLM}
and~\ref{SecGLasso} below, an equality of the first
type~\eqref{EqnLocalRSC} will only hold locally over $\Delta$ when we
have more complicated types of loss functions that are only quadratic
around a neighborhood of the origin. As proved in
Appendix~\ref{AppGenRSC}, however, \eqref{EqnL2Bd} is
implied by~\eqref{EqnLocalRSC} in cases when $\EmpLoss$ is
convex, which sustains our theoretical conclusions even under the
weaker RSC conditions~\eqref{EqnRSC}. Further note that by the
inequality
\begin{equation*}
	\Loss_n(\betastar + \Delta) - \Loss_n(\betastar) \le \inprod{\nabla \Loss_n(\betastar + \Delta)}{\Delta},
\end{equation*}
which holds whenever $\Loss_n$ is convex, the RSC condition appearing in past work (e.g., \citealp{AgaEtal12}) implies that~\eqref{EqnLocalRSC} holds, so~\eqref{EqnL2Bd} also holds by Lemma~\ref{LemLocalRSC} in Appendix~\ref{AppGenRSC}. In cases where $\Loss_n$ is quadratic but not necessarily convex (cf.\ Section~\ref{SecCorrLinear}), our RSC condition~\eqref{EqnRSC} is again no stronger than the conditions appearing in past work, since those RSC conditions enforce~\eqref{EqnLocalRSC} globally over $\Delta \in \real^p$, which by Lemma~\ref{LemGlobalRSC} in Appendix~\ref{AppGenRSC} implies that~\eqref{EqnL2Bd} holds, as well. To allow for more general situations where $\Loss_n$ may be non-quadratic and/or nonconvex, we prefer to use the RSC formulation~\eqref{EqnRSC} in this paper.

Finally, we clarify that whereas~\cite{NegRavWaiYu12} define an RSC condition with respect to a fixed subset $S \subseteq \{1, \dots, p\}$, we follow the setup of~\cite{AgaEtal12} and~\cite{LohWai11a} and essentially require an RSC condition of the type defined in~\cite{NegRavWaiYu12} to hold uniformly over \emph{all} subsets $S$ of size $k$. Although the results on statistical consistency may be established under the weaker RSC assumption with $S \defn\supp(\betastar)$, a uniform RSC condition is preferred because the true support set is not known a priori. The uniform RSC condition may be shown to hold w.h.p.\ in the sub-Gaussian settings we consider here (cf.\ Sections~\ref{SecCorrLinear}---\ref{SecGLasso} below); in fact, the proofs contained in~\cite{NegRavWaiYu12} establish a uniform RSC condition, as well.

%%%%%%%%%%%%%%%%%%%%%%%%%%%%%%%%%%%%%%%%%%%%%%%%%%%%%%%%%%%%%%%%%%%%%%%%

\section{Statistical Guarantees and Consequences}
\label{SecMain}

With this setup, we now turn to the statements and proofs of our main
statistical guarantees, as well as some consequences for various
statistical models.  Our theory applies to any vector $\betatil \in
\real^\pdim$ that satisfies the \emph{first-order necessary
  conditions} to be a local minimum of the
program~\eqref{EqnNonconvexReg}:
\begin{align}
\label{EqnFirstOrder}
\inprod{ \nabla \EmpLoss(\betatil) + \nabla \myrho(\betatil)}{\beta -
  \betatil} & \geq 0, \qquad \mbox{for all feasible $\beta \in
  \real^\pdim$.}
\end{align}
When $\betatil$ lies in the interior of the constraint set, this
condition reduces to the usual zero-subgradient condition:
\begin{equation*}
	\nabla \EmpLoss(\betatil) + \nabla \myrho(\betatil) = 0.
\end{equation*}
Such vectors $\betatil$ satisfying the condition~\eqref{EqnFirstOrder} are also known as \emph{stationary points}~\citep{Ber99}; note that the set of stationary points also includes interior local maxima. Hence, although some of the discussion below is stated in terms of ``local minima," the results hold for interior local maxima, as well.

%%%%%%%

\subsection{Main Statistical Results}

Our main theorems are deterministic in nature and specify
conditions on the regularizer, loss function, and parameters that
guarantee that any local optimum $\betatil$ lies close to the target
vector \mbox{$\betastar = \arg \min \limits_{\beta \in \real^\pdim}
  \PopLoss(\beta)$.} Corresponding probabilistic results will be
derived in subsequent sections, where we establish that for
appropriate choices of parameters $(\lambda, R)$, the required
conditions hold with high probability. Applying the theorems to
particular models requires bounding the random quantity
$\|\nabla \EmpLoss(\betastar)\|_\infty$ and verifying the RSC
conditions~\eqref{EqnRSC}.  We begin with a theorem that provides
guarantees on the error $\betatil - \betastar$ as measured in the
$\ell_1$- and $\ell_2$-norms:

\begin{mytheorem}
\label{TheoEll12Meta}
Suppose the regularizer $\myrho$ satisfies Assumption~\ref{AsRho}, the
empirical loss $\EmpLoss$ satisfies the RSC conditions~\eqref{EqnRSC}
with $\frac{3}{4} \mu < \alpha_1$, and $\betastar$ is feasible for the objective.
Consider any choice of $\lambda$ such that
\begin{align}
\label{EqnLambdaChoice}
\frac{4}{L} \cdot \max \left \{\|\nabla \EmpLoss(\betastar)\|_\infty,
\; \alpha_2 \sqrt{\frac{\log \pdim}{\numobs}}\right\} \; \leq \;
\lambda \; \leq \; \frac{\alpha_2}{6RL},
\end{align}
and suppose $\numobs \geq \frac{16R^2 \max(\tau_1^2, \tau_2^2)}{\alpha_2^2} \log
\pdim$.  Then any vector $\betatil$ satisfying the first-order
necessary conditions~\eqref{EqnFirstOrder} satisfies the error bounds
\begin{align}
\label{EqnStatBounds}
\|\betatil - \betastar\|_2 \le \frac{6 \lambda L \sqrt{k}}{4\alpha_1 -
  3\mu}, \qquad \mbox{and} \qquad \|\betatil - \betastar\|_1 \le
\frac{24 \lambda L k}{4\alpha_1 - 3\mu},
\end{align}
where $k = \|\betastar\|_0$.
\end{mytheorem}

From the bound~\eqref{EqnStatBounds}, note that the squared
$\ell_2$-error grows proportionally with $k$, the number of nonzeros
in the target parameter, and with $\lambda^2$.  As will be clarified
in the following sections, choosing $\lambda$ proportional to
$\sqrt{\frac{\log \pdim}{\numobs}}$ and $R$ proportional to
$\frac{1}{\lambda}$ will satisfy the requirements of
Theorem~\ref{TheoEll12Meta} w.h.p.\ for many statistical models, in
which case we have a squared-$\ell_2$ error that scales as
$\frac{\kdim \log \pdim}{\numobs}$, as expected.

Our next theorem provides a bound on a measure of the prediction
error, as defined by the quantity
\begin{align}
\label{EqnPredError}
D \big( \betatil; \betastar) & \defn \inprod{\nabla \Loss_n(\betatil)
  - \nabla \Loss_n(\betastar)}{\betatil - \betastar}.
\end{align}
When the empirical loss $\Loss_n$ is a convex function, this
measure is always nonnegative, and in various special cases, it has a
form that is readily interpretable.  For instance, in the case of the
least-squares objective function $\Loss_n(\beta) = \frac{1}{2 \numobs}
\|y - X \beta\|_2^2$, we have 
\begin{align*}
D \big( \betatil; \betastar) & = \frac{1}{\numobs} \|X ( \betatil -
\betastar)\|_2^2 = \frac{1}{\numobs} \sum_{i=1}^\numobs \big(
\inprod{x_i}{\betatil - \betastar} \big)^2,
\end{align*}
corresponding to the usual measure of (fixed design) prediction error
for a linear regression problem (cf.\ Corollary~\ref{CorLinear}
below).  More generally, when the loss function is the negative log
likelihood for a generalized linear model with cumulant function
$\psi$, the error measure~\eqref{EqnPredError} is equivalent to
the symmetrized Bregman divergence defined by $\psi$. (See
Section~\ref{SecGLM} for further details.)

\begin{mytheorem}
\label{TheoPredMeta}
Under the same conditions as Theorem~\ref{TheoEll12Meta}, the
error measure~\eqref{EqnPredError} is bounded as
\begin{equation}
\label{EqnPredict}
\inprod{\nabla \Loss_n(\betatil) - \nabla \Loss_n(\betastar)}{\betatil
  - \betastar} \le \lambda^2 L^2 k \left(\frac{9}{4\alpha_1 - 3\mu}
+ \frac{27\mu}{(4\alpha_1 - 3\mu)^2}\right).
\end{equation}
\end{mytheorem}
This result shows that the prediction error~\eqref{EqnPredError}
behaves similarly to the squared Euclidean norm between $\betatil$ and
$\betastar$. \\

\textbf{Remark on $(\alpha_1, \mu)$:}
It is worthwhile to discuss the quantity $4\alpha_1 - 3\mu$ appearing
in the denominator of the bounds in Theorems~\ref{TheoEll12Meta}
and~\ref{TheoPredMeta}. Recall that $\alpha_1$ measures the level of
curvature of the loss function $\Loss_n$, while $\mu$ measures the
level of nonconvexity of the penalty $\rho_\lambda$. Intuitively, the
two quantities should play opposing roles in our result: larger values
of $\mu$ correspond to more severe nonconvexity of the penalty,
resulting in worse behavior of the overall
objective~\eqref{EqnNonconvexReg}, whereas larger values of $\alpha_1$
correspond to more (restricted) curvature of the loss, leading to
better behavior. However, while the condition $\frac{3}{4}\mu < \alpha_1$ is
needed for the proof technique employed in
Theorem~\ref{TheoEll12Meta}, it does not seem to be strictly necessary
in order to guarantee good behavior of local optima. As a careful examination of the proof reveals, the condition may be replaced by the alternate condition $c \mu < \alpha_1$, for any constant $c > \frac{1}{2}$. However, note that
the capped-$\ell_1$ penalty may be viewed as a limiting version of
SCAD when $a \rightarrow 1$, or equivalently, $\mu \rightarrow
\infty$. Viewed in this light, Theorem~\ref{TheoEll12Meta2}, to be
stated and proved in Appendix~\ref{AppCapped}, reveals that a
condition of the form $c\mu < \alpha_1$ is not necessary, at least in general, for
good behavior of local optima.  Moreover, Section~\ref{SecSims}
contains empirical studies using linear regression and the SCAD
penalty showing that local optima may be well-behaved when
$\alpha_1 < \frac{3}{4} \mu$.  Nonetheless, our simulations (see
Figure~\ref{FigBreakdown}) also convey a cautionary message: In
extreme cases, where $\alpha_1$ is significantly smaller than $\mu$, the good
behavior of local optima (and the optimization algorithms used to find
them) appear to degenerate.

Finally, we note that~\cite{NegRavWaiYu12} have shown that for convex
$M$-estimators, the arguments used to analyze $\ell_1$-regularizers
may be generalized to other types of ``decomposable'' regularizers,
such as norms for group sparsity or the nuclear norm for low-rank
matrices. In our present setting, where we allow for nonconvexity in
the loss and regularizer, our theorems have straightforward and
analogous generalizations.

%%%%%%%%%%%%%%%%%%%%%%%%%%%%%%%%%%%%%%%%%%%%%%%%%%%%%%%%%%%%%%%%%%%%%%%%%%%%%%%%

We return to the proofs of Theorems~\ref{TheoEll12Meta}
and~\ref{TheoPredMeta} in Section~\ref{SecTheoProof}. First, we
develop various consequences of these theorems for various nonconvex
loss functions and regularizers of interest. The main technical
challenge is to establish that the RSC conditions~\eqref{EqnRSC} hold
with high probability for appropriate choices of positive constants
$\{(\alpha_j, \tau_j)\}_{j=1}^2$.

%%%%%%%%%%%%%%%%%%%%%%%%%%%%%%%%%%%%%%%%%%%%%%%%%%%%%%%%%%%%%%%%%%%%%%%%%%%

\subsection{Corrected Linear Regression}
\label{SecCorrLinear}

We begin by considering the case of high-dimensional linear regression
with systematically corrupted observations. Recall that in the
framework of ordinary linear regression, we have the linear model
\begin{equation}
\label{EqnLinModel}
y_i = \underbrace{\inprod{\betastar}{x_i}}_{\sum_{j=1}^\pdim \beta^*_j
  x_{ij}} + \; \epsilon_i, \qquad \mbox{for $i=1, \dots, \numobs$,}
\end{equation}
where $\betastar \in \real^\pdim$ is the unknown parameter vector and
$\{(x_i, y_i)\}_{i=1}^n$ are observations. Following a line of past
work (e.g.,~\citealp{RosTsy10, LohWai11a}), assume we instead observe
pairs $\{(z_i, y_i)\}_{i=1}^n$, where the $z_i$'s are systematically
corrupted versions of the corresponding $x_i$'s. Some examples of
corruption mechanisms include the following:

\begin{enumerate}
\item[(a)] \emph{Additive noise:} We observe $z_i = x_i + w_i$, where
  \mbox{$w_i \in \real^{\pdim}$} is a random vector independent of
  $x_i$, say zero-mean with known covariance matrix $\Sigma_w$.
\item[(b)] \emph{Missing data:} For some fraction $\vartheta \in
  [0,1)$, we observe a random vector $z_i \in \real^{\pdim}$ such that
    for each component $j$, we independently observe $z_{ij} = x_{ij}$
    with probability $1-\vartheta$, and $z_{ij} = \ast$ with
    probability $\vartheta$.
\end{enumerate}

We use the population and empirical loss functions
\begin{align}
\label{EqnCorLinearLoss}
\PopLoss(\beta) = \frac{1}{2} \beta^T \Sigma_x \beta - \beta^{*T}
\Sigma_x \beta, \qquad \text{and} \qquad \EmpLoss(\beta) = \frac{1}{2}
\beta^T \GamHat \beta - \gamhat^T \beta,
\end{align}
where $(\GamHat, \gamhat)$ are estimators for $(\Sigma_x, \Sigma_x
\betastar)$ that depend only on $\{(z_i, y_i)\}_{i=1}^n$. It is easy to
see that $\betastar = \arg\min_\beta \Loss(\beta)$. From the
formulation~\eqref{EqnNonconvexReg}, the corrected linear regression
estimator is given by
\begin{equation}
\label{EqnLinearEst}
\betahat \in \arg \min_{\SIDE(\beta) \le R} \left\{\frac{1}{2} \beta^T
\GamHat \beta - \gamhat^T \beta + \rho_\lambda(\beta) \right \}.
\end{equation}

We now state a concrete corollary in the case of additive noise (model
(a) above).  In this case, as discussed in~\cite{LohWai11a}, an
appropriate choice of the pair $(\GamHat, \gamhat)$ is given by
\begin{equation}
\label{EqnDefnGamHat}
\GamHat = \frac{Z^TZ}{n} - \Sigma_w, \qquad \text{and} \qquad \gamhat
= \frac{Z^T y}{n}.
\end{equation}
Here, we assume the noise covariance $\Sigma_w$ is known or may be
estimated from replicates of the data. Such an assumption also appears
in canonical errors-in-variables literature~\citep{CarEtal95}, but it
is an open question how to devise a corrected estimator when an
estimate of $\Sigma_w$ is not readily available. If we assume a
sub-Gaussian model on the covariates and errors (i.e., $x_i$, $w_i$,
and $\epsilon_i$ are sub-Gaussian with parameters $\sigma_x^2$,
$\sigma_w^2$, and $\sigma_\epsilon^2$, respectively), the contribution
of the error covariances may be summarized in the error term
\begin{equation}
\label{EqnPhi}
\varphi = (\sigma_x + \sigma_w)\big(\sigma_\epsilon + \sigma_w
\|\betastar\|_2\big),
\end{equation}
which appears as a prefactor in the deviation bounds and
estimation/prediction error bounds for the subsequent estimators
(cf.\ Lemma 2 in~\citealp{LohWai11a}). We make this dependence
explicit in the statement of the corollary for high-dimensional
errors-in-variables regression below. Note in particular that
$\varphi$ scales up with both $\sigma_\epsilon$ and $\sigma_w$. Hence,
even when $\sigma_\epsilon = 0$, corresponding to no additive error,
we will have $\varphi \neq 0$ due to errors in the covariates; whereas
when $\sigma_w = 0$, corresponding to cleanly observed covariates, we
will still have $\varphi \neq 0$ due to the additional additive error
introduced by the $\epsilon_i$'s, agreeing with canonical results for
the Lasso~\citep{BicEtal08}.

In the high-dimensional setting ($\pdim \gg \numobs$), the matrix
$\GamHat$ in~\eqref{EqnDefnGamHat} is always negative
definite: the matrix $\frac{Z^T Z}{\numobs}$ has rank at most
$\numobs$, and the positive definite matrix $\Sigma_w$ is then
subtracted to obtain $\GamHat$.  Consequently, the empirical loss
function $\EmpLoss$ previously defined~\eqref{EqnCorLinearLoss} is
nonconvex.  Other choices of $\GamHat$ are applicable to missing data
(model (b)), and also lead to nonconvex programs
(see~\citealp{LohWai11a} for further details).

\begin{mycorollary}
\label{CorLinear}
Suppose we have i.i.d.\ observations $\{(z_i, y_i)\}_{i=1}^n$ from a
corrupted linear model with additive noise, where the covariates and
error terms are sub-Gaussian. Let $\varphi$ be defined as in~\eqref{EqnPhi} with respect to the sub-Gaussian
parameters. Suppose $(\lambda, R)$ are chosen such that $\betastar$ is
feasible and
\begin{equation*}
c \varphi \sqrt{\frac{\log p}{n}} \le \lambda \le \frac{c'}{R}.
\end{equation*}
Also suppose $\frac{3}{4}\mu < \frac{1}{2} \lambda_{\min}(\Sigma_x)$. Then given a sample
size $\numobs \ge C \, \max\{R^2, k\} \log \pdim$, any stationary point
$\betatil$ of the nonconvex program~\eqref{EqnLinearEst} satisfies the
estimation error bounds
\begin{align*}
\|\betatil - \betastar\|_2 \leq \frac{c_0 \lambda
  \sqrt{k}}{2\lambda_{\min}(\Sigma_x) - 3\mu}, \qquad \mbox{and} \qquad
\|\betatil - \betastar\|_1 \leq \frac{c_0' \lambda
  k}{2\lambda_{\min}(\Sigma_x) - 3\mu},
\end{align*}
and the prediction error bound
\begin{equation*}
	\nutil^T \GamHat \nutil \le \lambda^2 k
        \left(\frac{\widetilde{c_0}}{2\lambda_{\min}(\Sigma_x) - 3\mu} +
        \frac{\widetilde{c_0}' \mu}{(2\lambda_{\min}(\Sigma_x) -
          3\mu)^2}\right),
\end{equation*}
with probability at least $1 - c_1 \exp(-c_2 \log \pdim)$, where
$\|\betastar\|_0 = k$.
\end{mycorollary}

When $\myrho(\beta) = \lambda \|\beta\|_1$ and $\SIDE(\beta) =
\|\beta\|_1$, taking \mbox{$\lambda \asymp \varphi
  \sqrt{\frac{\log p}{n}}$} and \mbox{$R = b_0 \sqrt{k}$} for some
constant $b_0 \ge \|\betastar\|_2$ yields the required scaling $n
\succsim k \log p$. Hence, the bounds of Corollary~\ref{CorLinear}
agree with bounds previously established in Theorem 1
of~\cite{LohWai11a}. Note, however, that those results are stated only
for a \emph{global minimum} $\betahat$ of the
program~\eqref{EqnLinearEst}, whereas Corollary~\ref{CorLinear} is a
much stronger result holding for \emph{any stationary point}
$\betatil$. Theorem 2 of our earlier paper~\citep{LohWai11a} provides
a rather indirect (algorithmic) route for establishing similar bounds on
$\|\betatil - \betastar\|_1$ and $\|\betatil - \betastar\|_2$, since
the proposed projected gradient descent algorithm may become stuck at
a stationary point. In contrast, our argument here is much more direct
and does not rely on an algorithmic proof. Furthermore, our result is
applicable to a more general class of (possibly nonconvex) penalties
beyond the usual $\ell_1$-norm.

Corollary~\ref{CorLinear} also has important consequences in the case
where pairs $\{(x_i, y_i)\}_{i=1}^\numobs$ from the linear
model~\eqref{EqnLinModel} are observed cleanly without corruption and
$\rho_\lambda$ is a nonconvex penalty.  In that case, the empirical
loss $\EmpLoss$ previously defined~\eqref{EqnCorLinearLoss} is
equivalent to the least-squares loss, modulo a constant factor. Much
existing work, including that of~\cite{FanLi01} and~\cite{ZhaZha12},
first establishes statistical consistency results concerning
\emph{global} minima of the program~\eqref{EqnLinearEst}, then
provides specialized algorithms such as a local linear approximation
(LLA) for obtaining specific local optima that are provably close to the
global optima. However, our results show that \emph{any} optimization
algorithm guaranteed to converge to a stationary point of the program
suffices. See Section~\ref{SecOpt} for a more detailed discussion of
optimization procedures and fast convergence guarantees for obtaining
stationary points. In the fully-observed case, we also have $\GamHat =
\frac{X^TX}{n}$, so the prediction error bound in
Corollary~\ref{CorLinear} agrees with the familiar scaling
$\frac{1}{n} \|X(\betatil - \betastar)\|_2^2 \precsim \frac{k \log
  p}{n}$ appearing in $\ell_1$-theory.

Furthermore, our theory provides a theoretical motivation for why the
usual choice of $a = 3.7$ for linear regression with the SCAD
penalty~\citep{FanLi01} is reasonable. Indeed, as discussed in
Section~\ref{SecNonconvexRegExas}, we have
\begin{equation*}
\mu = \frac{1}{a-1} \approx 0.37
\end{equation*}
in that case. Since $x_i \sim N(0, I)$ in the SCAD simulations, we
have $\frac{3}{4} \mu < \frac{1}{2} \lambda_{\min}(\Sigma_x)$ for the choice $a = 3.7$. For
further comments regarding the parameter $a$ in the SCAD penalty, see
the discussion concerning Figure~\ref{FigLinLogerrs} in
Section~\ref{SecSims}.

%%%%%%%%%%%%%%%%%%%%%%%%%%%%%%%%%%%%%%%%%%%%%%%%%%%%%%%%%%%%%%%%%%%%%%%%%%%%%%%%

\subsection{Generalized Linear Models}
\label{SecGLM}

Moving beyond linear regression, we now consider the case where
observations are drawn from a generalized linear model (GLM). Recall
that a GLM is characterized by the conditional distribution
\begin{align*}
\mprob(y_i \mid x_i, \beta, \sigma) & = \exp \left \{\frac{y_i
  \inprod{\beta}{x_i} - \psi (x_i^T \beta)}{c(\sigma)} \right \},
\end{align*}
where $\sigma > 0$ is a scale parameter and $\psi$ is the cumulant
function, By standard properties of exponential
families~\citep{McCNel89, LehCas98}, we have
\begin{align*}
\psi' (x_i^T \beta) & = \E[y_i \mid x_i, \beta, \sigma].
\end{align*}
In our analysis, we assume that there exists $\alpha_u > 0$ such that
$\psi''(t) \leq \alpha_u$, for all $t \in \real$.  Note that this
boundedness assumption holds in various settings, including linear
regression, logistic regression, and multinomial regression, but does
not hold for Poisson regression. The bound will be necessary to
establish both statistical consistency results in the present section
and fast global convergence guarantees for our optimization algorithms
in Section~\ref{SecOpt}.

The population loss corresponding to the negative log likelihood is
then given by
\begin{align*}
\Loss(\beta) & = -\E[\log \mprob(x_i, y_i)] = -\E[\log \mprob(x_i)] -
\frac{1}{c(\sigma)} \cdot \E[y_i \inprod{\beta}{x_i} - \psi (x_i^T
  \beta)],
\end{align*}
giving rise to the population-level and empirical gradients
\begin{align*}
\nabla \Loss(\beta) & = \frac{1}{c(\sigma)} \cdot \E[(\psi' (x_i^T
  \beta) - y_i) x_i], \quad \text{and} \\
\nabla \EmpLoss(\beta) & = \frac{1}{c(\sigma)} \cdot \frac{1}{n}
\sum_{i=1}^n \big (\psi'(x_i^T \beta) - y_i \big) x_i.
\end{align*}
Since we are optimizing over $\beta$, we will rescale the loss
functions and assume $c(\sigma) = 1$. We may check that if $\betastar$
is the true parameter of the GLM, then $\nabla \Loss(\betastar) = 0$;
furthermore,
\begin{equation*}
\nabla^2 \Loss_n(\beta) = \frac{1}{n} \sum_{i=1}^n \psi''(x_i^T \beta)
x_i x_i^T \succeq 0,
\end{equation*}
so $\Loss_n$ is convex.
	
We will assume that $\betastar$ is sparse and optimize the penalized
maximum likelihood program
\begin{align}
\label{EqnGLMEst}
\betahat & \in \arg \min_{ \SIDE(\beta) \leq R}\left\{\frac{1}{n}
\sum_{i=1}^n \big(\psi(x_i^T \beta) - y_i x_i^T \beta\big) +
\rho_\lambda(\beta)\right\}.
\end{align}
We then have the following corollary, proved in
Appendix~\ref{AppCorGLM}:

\begin{mycorollary}
\label{CorGLM}
Suppose we have i.i.d.\ observations $\{(x_i, y_i)\}_{i=1}^n$ from a
GLM, where the $x_i$'s are sub-Gaussian. Suppose $(\lambda, R)$ are
chosen such that $\betastar$ is feasible and
\begin{equation*}
c\sqrt{\frac{\log p}{n}} \le \lambda \le \frac{c'}{R}.
	\end{equation*}
Then given a sample size $n \ge C R^2 \log p$, any stationary point
$\betatil$ of the nonconvex program~\eqref{EqnGLMEst} satisfies
\begin{equation*}
\|\betatil - \betastar\|_2 \le \frac{c_0\lambda \sqrt{k}}{4\alpha_1-
  3\mu}, \qquad \mbox{and} \qquad \|\betatil - \betastar\|_1 \le
\frac{c_0'\lambda k}{4\alpha_1 - 3\mu},
\end{equation*}
with probability at least $1 - c_1 \exp(-c_2 \log p)$, where
$\|\betastar\|_0 = k$. Here, $\alpha_1$ is a constant depending on
$\|\betastar\|_2$, $\psi$, $\lambda_{\min}(\Sigma_x)$, and the
sub-Gaussian parameter of the $x_i$'s, and we assume $\mu <
2\alpha_1$.
\end{mycorollary}

Although $\EmpLoss$ is convex in this case, the overall program may
\emph{not} be convex if the regularizer $\myrho$ is nonconvex, giving
rise to multiple local optima.  For instance, see the simulations of
Figure~\ref{FigLogisticLogerrs} in Section~\ref{SecSims} for a
demonstration of such local optima.  In past work, \cite{BreHua11}
studied logistic regression with SCAD and MCP regularizers, but did
not provide any theoretical results on the quality of the local
optima.  In this context, Corollary~\ref{CorGLM} shows that their
coordinate descent algorithms are guaranteed to converge to a stationary point $\betatil$ within close proximity of the true parameter
$\betastar$.

In the statement of Corollary~\ref{CorGLM}, we choose not to write out
the form of $\alpha_1$ explicitly as in Corollary~\ref{CorLinear},
since it is rather complicated. As explained in the proof of
Corollary~\ref{CorGLM} in Appendix~\ref{AppCorGLM}, the precise form
of $\alpha_1$ may be traced back to Proposition 2
of~\cite{NegRavWaiYu12}.

\subsection{Graphical Lasso}
\label{SecGLasso}
Finally, we specialize our results to the case of the graphical
Lasso. Given $p$-dimensional observations $\{x_i\}_{i=1}^n$, the goal
is to estimate the structure of the underlying (sparse) graphical
model. Recall that the population and empirical losses for the
graphical Lasso are given by
\begin{align*}
\Loss(\Theta) = \trace(\Sigma \Theta) - \log\det(\Theta), \quad
\text{and} \quad \EmpLoss(\Theta) = \trace(\SigHat \Theta) -
\log\det(\Theta),
\end{align*}
where $\SigHat$ is an empirical estimate for the covariance matrix
$\Sigma = \Cov(x_i)$. The objective function for the graphical Lasso
is then given by
\begin{equation}
\label{EqnGLasso}
\ThetaHat \in \arg\min_{\SIDE(\Theta) \le R, \; \Theta \succeq 0}
\left\{\trace(\SigHat \Theta) - \log \det(\Theta) + \sum_{j,k = 1}^p
\myrho(\Theta_{jk})\right\},
\end{equation}
where we apply the (possibly nonconvex) penalty function $\myrho$ to
all entries of $\Theta$, and define $\Omega \defn \big\{ \Theta \in
\real^{\pdim \times \pdim} \mid \Theta = \Theta^T, \; \Theta \succeq 0
\big \}$.

A host of statistical and algorithmic results have been established
for the graphical Lasso in the case of Gaussian observations with an
$\ell_1$-penalty~\citep{BanEtal08, FriEtal08, Rot08, YuaLin07}, and
more recently, for discrete-valued observations, as
well~\citep{LohWai12}. In addition, a version of the graphical Lasso
incorporating a nonconvex SCAD penalty has been
proposed~\citep{FanEtal09}. Our results subsume previous Frobenius
error bounds for the graphical Lasso and again imply that even in the
presence of a nonconvex regularizer, all stationary points of the nonconvex
program~\eqref{EqnGLasso} remain close to the true inverse covariance
matrix $\Thetastar$.

As suggested by~\cite{LohWai12}, the graphical Lasso easily
accommodates systematically corrupted observations, with the only
modification being the form of the sample covariance matrix
$\SigHat$. Just as in Corollary~\ref{CorLinear}, the magnitude and
form of corruption would occur as a prefactor in the deviation
condition captured in~\eqref{EqnSigConc} below; for
instance, in the case of $\SigHat = \frac{Z^TZ}{n} - \Sigma_w$, corresponding to
additive noise in the $x_i$'s, the bound~\eqref{EqnSigConc} would
involve a prefactor of $\sigma_z^2$ rather than $\sigma_x^2$, where
$\sigma_z^2$ and $\sigma_x^2$ are the sub-Gaussian parameters of $z_i$
and $x_i$, respectively.

Further note that the program~\eqref{EqnGLasso} is always useful for
obtaining a consistent estimate of a sparse inverse covariance matrix,
regardless of whether the $x_i$'s are drawn from a distribution for
which $\Thetastar$ is relevant in estimating the edges of the
underlying graph. Note that other variants of the graphical Lasso
exist in which only off-diagonal entries of $\Theta$ are penalized,
and similar results for statistical consistency hold in that
case. Here, we assume that all entries are penalized equally in order to
simplify our arguments. The same framework is considered
by~\cite{FanEtal09}.

We have the following result, proved in
Appendix~\ref{AppCorGLasso}. The statement of the corollary is purely
deterministic, but in cases of interest (say, sub-Gaussian
observations), the deviation condition~\eqref{EqnSigConc} holds with
probability at least $1 - c_1 \exp(-c_2 \log p)$, translating into the
Frobenius norm bound~\eqref{EqnFrobBound} holding with the same
probability.

\begin{mycorollary}
\label{CorGLasso}
Suppose we have an estimate $\SigHat$ of the covariance matrix
$\Sigma$ based on (possibly corrupted) observations $\{x_i\}_{i=1}^n$,
such that
\begin{equation}
\label{EqnSigConc}
\opnorm{\SigHat - \Sigma}_{\max} \le c_0 \sqrt{\frac{\log p}{n}}.
\end{equation}
Also suppose $\Thetastar$ has at most $s$ nonzero entries. Suppose
$(\lambda, R)$ are chosen such that $\Thetastar$ is feasible and
\begin{equation*}
	c \sqrt{\frac{\log p}{n}} \le \lambda \le \frac{c'}{R}.
\end{equation*}
Suppose $\frac{3}{4} \mu < \left(\opnorm{\Thetastar}_2 + 1\right)^{-2}$. Then
with a sample size $\numobs > C s \log \pdim$, for a sufficiently
large constant $C > 0$, any stationary point $\Thetatil$ of the nonconvex
program~\eqref{EqnGLasso} satisfies
\begin{equation}
\label{EqnFrobBound}
\opnorm{\Thetatil - \Thetastar}_F \le \frac{c_0'\lambda
  \sqrt{s}}{4\left(\opnorm{\Thetastar}_2 + 1\right)^{-2} - 3\mu}.
\end{equation}
\end{mycorollary}

When $\rho$ is simply the $\ell_1$-penalty, the
bound~\eqref{EqnFrobBound} from Corollary~\ref{CorGLasso} matches the
minimax rates for Frobenius norm estimation of an $s$-sparse inverse
covariance matrix~\citep{Rot08, RavEtal11}.

\subsection{Proof of Theorems~\ref{TheoEll12Meta} and~\ref{TheoPredMeta}}
\label{SecTheoProof}

We now turn to the proofs of our two main theorems. \\

\textbf{Proof of Theorem~\ref{TheoEll12Meta}:} Introducing the shorthand $\nutil \defn \betatil - \betastar$, we
 begin by proving that $\|\nutil\|_2 \leq 1$.  If not, then~\eqref{EqnL2Bd} gives the lower bound
\begin{align}
\label{EqnL2BdTwo}
\inprod{\nabla \Loss_n(\betatil) - \nabla \Loss_n(\betastar)}{\nutil}
& \geq \alpha_2 \|\nutil\|_2 - \tau_2 \sqrt{\frac{\log p}{n}}
\|\nutil\|_1.
\end{align}
Since $\betastar$ is feasible, we may take $\beta = \betastar$ in~\eqref{EqnFirstOrder}, and combining with~\eqref{EqnL2BdTwo} yields
\begin{align}
\label{EqnPlug}
\inprod{-\nabla \rho_\lambda(\betatil) - \nabla
  \EmpLoss(\betastar)}{\nutil} & \geq \alpha_2 \|\nutil\|_2 - \tau_2
\sqrt{\frac{\log \pdim}{\numobs}} \|\nutil\|_1 .
\end{align}
By H\"{o}lder's inequality, followed by the triangle inequality, we
also have
\begin{align*}
\inprod{-\nabla \myrho(\betatil) - \nabla \EmpLoss(\betastar)}{\nutil}
& \leq \Big \{ \|\nabla \myrho(\betatil)\|_\infty + \|\nabla
\EmpLoss(\betastar)\|_\infty \Big \} \, \|\nutil\|_1 \\
& \stackrel{(i)}{\leq} \left \{ \lambda L + \frac{\lambda L}{2} \right
\} \, \|\nutil\|_1,
\end{align*}
where inequality (i) follows since $\| \nabla
\EmpLoss(\betastar)\|_\infty \leq \frac{\lambda L}{2}$ by the
bound~\eqref{EqnLambdaChoice}, and \mbox{$\|\nabla
  \myrho(\betatil)\|_\infty \le \lambda L$} by Lemma~\ref{LemRhoCond}
in Appendix~\ref{AppGeneralProps}.  Combining this upper bound with~\eqref{EqnPlug} and rearranging then yields
\begin{align*}
\|\nutil\|_2 & \le \frac{\|\nutil\|_1}{\alpha_2} \left(\frac{3 \lambda
  L}{2} + \tau_2 \sqrt{\frac{\log p}{n}}\right) \le
\frac{2R}{\alpha_2} \left(\frac{3 \lambda L}{2} + \tau_2
\sqrt{\frac{\log \pdim}{\numobs}}\right).
\end{align*}
By our choice of $\lambda$ from~\eqref{EqnLambdaChoice} and
the assumed lower bound on the sample size $\numobs$, the right hand
side is at most $1$, so $\|\nutil\|_2 \leq 1$, as claimed.  \\

\vspace*{.1in}

Consequently, we may apply~\eqref{EqnLocalRSC}, yielding
the lower bound
\begin{align}
\label{EqnDeltaRSC}
\inprod{\nabla \Loss_n(\betatil) - \nabla \Loss_n(\betastar)}{\nutil}
& \geq \alpha_1 \|\nutil\|_2^2 - \tau_1 \frac{\log p}{n}
\|\nutil\|_1^2.
\end{align}
Since the function $\doubrho(\beta) \defn \myrho(\beta) +
\frac{\mu}{2} \|\beta\|_2^2$ is convex by assumption, we have
\begin{align*}
\doubrho(\betastar) - \doubrho(\betatil) & \geq \inprod{\nabla
  \doubrho(\betatil)}{\betastar - \betatil} \; = \; \inprod{\nabla
  \myrho(\betatil) + \mu \betatil}{\betastar - \betatil},
\end{align*}
implying that
\begin{align}
\label{EqnRhotilConvex} 
\inprod{\nabla \myrho(\betatil)}{\betastar - \betatil} \le
\myrho(\betastar) - \myrho(\betatil) + \frac{\mu}{2} \|\betatil -
\betastar\|_2^2.
\end{align}
Combining~\eqref{EqnDeltaRSC} with~\eqref{EqnFirstOrder} and~\eqref{EqnRhotilConvex}, we
obtain
\begin{equation*}
\alpha_1 \|\nutil\|_2^2 - \tau_1 \frac{\log p}{n} \|\nutil\|_1^2 \le - \inprod{\nabla \Loss_n(\betastar)}{\nutil} + \rho_\lambda(\betastar) - \rho_\lambda(\betatil) + \frac{\mu}{2} \|\betatil - \betastar\|_2^2.
\end{equation*}
Rearranging and using H\"{o}lder's inequality, we then have
\begin{align}
\label{EqnSanta}
\left(\alpha_1 - \frac{\mu}{2}\right) \|\nutil\|_2^2 & \le \rho_\lambda(\betastar) - \rho_\lambda(\betatil) + \|\nabla \Loss_n(\betastar)\|_\infty \cdot \|\nutil\|_1 + \tau_1 \frac{\log p}{n} \|\nutil\|_1^2 \notag \\
& \le \rho_\lambda(\betastar) - \rho_\lambda(\betatil) + \left(\|\nabla \Loss_n(\betastar)\|_\infty + 4R\tau_1 \frac{\log p}{n}\right) \|\nutil\|_1.
\end{align}
Note that by our assumptions, we have
\begin{equation*}
\|\nabla \Loss_n(\betastar)\|_\infty + 4R\tau_1 \frac{\log p}{n} \le \frac{\lambda L}{4} + \alpha_2 \sqrt{\frac{\log p}{n}} \le \frac{\lambda L}{2}.
\end{equation*}
Combining this with~\eqref{EqnSanta} and~\eqref{EqnXmas} in Lemma~\ref{LemRhoCond} in Appendix~\ref{AppGeneralProps}, as well as the subadditivity of $\rho_\lambda$, we then have
\begin{align*}
\left(\alpha_1 - \frac{\mu}{2}\right) \|\nutil\|_2^2 & \le \rho_\lambda(\betastar) - \rho_\lambda(\betatil) + \frac{\lambda L}{2} \cdot \left(\frac{\rho_\lambda(\nutil)}{\lambda L} + \frac{\mu}{2\lambda L} \|\nutil\|_2^2\right) \\
& \le \rho_\lambda(\betastar) - \rho_\lambda(\betatil) + \frac{\rho_\lambda(\betastar) + \rho_\lambda(\betatil)}{2} + \frac{\mu}{4} \|\nutil\|_2^2,
\end{align*}
implying that
\begin{equation}
\label{EqnReindeer}
0 \le \left(\alpha_1 - \frac{3\mu}{4}\right) \|\nutil\|_2^2 \le \frac{3}{2} \rho_\lambda(\betastar) - \frac{1}{2} \rho_\lambda(\betatil).
\end{equation}
In particular, we have $3 \rho_\lambda(\betastar) - \rho_\lambda(\betatil) \ge 0$, so we may apply Lemma~\ref{LemEll1Reg} in Appendix~\ref{AppGeneralProps} to conclude that
\begin{equation}
\label{EqnFrosty}
3 \rho_\lambda(\betastar) - \rho_\lambda(\betatil) \le 3\lambda L \|\nutil_A\|_1 - \lambda L \|\nutil_{A^c}\|_1,
\end{equation}
where $A$ denotes the index set of the $k$ largest elements of $\betatil - \betastar$ in magnitude. In particular, we have the cone condition
\begin{equation}
\label{EqnBlueBunny}
\|\nutil_{A^c}\|_1 \le 3 \|\nutil_A\|_1.
\end{equation}
Substituting~\eqref{EqnFrosty} into~\eqref{EqnReindeer}, we then have
\begin{align*}
\left(2\alpha_1 - \frac{3\mu}{2}\right) \|\nutil\|_2^2 \; \le \; 3 \lambda L \|\nutil_A\|_1 - \lambda L \|\nutil_{A^c}\|_1 \; \le \; 3 \lambda L \|\nutil_A\|_1 \; \le \; 3 \lambda L \sqrt{k} \|\nutil\|_2,
\end{align*}
from which we conclude that
\begin{equation*}
\|\nutil\|_2 \le \frac{6 \lambda L \sqrt{k}}{4 \alpha_1 - 3\mu},
\end{equation*}
as wanted. The $\ell_1$-bound
follows from the $\ell_2$-bound and the observation that
\begin{equation*}
\|\nutil\|_1 \le \|\nutil_A\|_1 + \|\nutil_{A^c}\|_1 \le 4 \|\nutil_A\|_1 \le 4 \sqrt{k} \|\nutil\|_2,
\end{equation*}
using the cone inequality~\eqref{EqnBlueBunny}. \\

\textbf{Proof of Theorem~\ref{TheoPredMeta}:} In order to establish~\eqref{EqnPredict}, note that
combining the first-order condition~\eqref{EqnFirstOrder} with the
upper bound~\eqref{EqnRhotilConvex}, we have
\begin{align}
\label{EqnElf}
	\inprod{\nabla \Loss_n(\betatil) - \nabla
          \Loss_n(\betastar)}{\nutil} & \le \inprod{- \nabla
          \rho_\lambda(\betatil) - \nabla \Loss_n(\betastar)}{\nutil}
        \notag \\ 
& \le \rho_\lambda(\betastar) - \rho_\lambda(\betatil) +
        \frac{\mu}{2}\|\nutil\|_2^2 + \|\nabla
        \Loss_n(\betastar)\|_\infty \cdot \|\nutil\|_1.
\end{align}
Furthermore, as noted earlier, Lemma~\ref{LemRhoCond} in Appendix~\ref{AppGeneralProps} implies that
\begin{equation*}
\|\nabla \Loss_n(\betastar)\|_\infty \cdot \|\nutil\|_1 \le \frac{\lambda L}{2} \cdot \left(\frac{\rho_\lambda(\betastar) + \rho_\lambda(\betatil)}{\lambda L} + \frac{\mu}{2\lambda L} \|\nutil\|_2^2\right) \le \frac{\rho_\lambda(\betastar) + \rho_\lambda(\betatil)}{2} + \frac{\mu}{4} \|\nutil\|_2^2.
\end{equation*}
Substituting this into~\eqref{EqnElf} then gives
\begin{align*}
\inprod{\nabla \Loss_n(\betatil) - \nabla \Loss_n(\betastar)}{\nutil} & \le \frac{3}{2} \rho_\lambda(\betastar) - \frac{1}{2} \rho_\lambda(\betatil) + \frac{3\mu}{4} \|\nutil\|_2^2 \\
& \le \frac{3 \lambda L}{2} \|\nutil_{A}\|_1 - \frac{\lambda L}{2} \|\nutil_{A^c}\|_1 + \frac{3\mu}{4} \|\nutil\|_2^2 \\
& \le \frac{3\lambda L \sqrt{k}}{2} \|\nutil\|_2 + \frac{3\mu}{4} \|\nutil\|_2^2,
\end{align*}
so substituting in the $\ell_2$-bound~\eqref{EqnStatBounds} yields the
desired result.

%%%%%%%%%%%%%%%%%%%%%%%%

\section{Optimization Algorithms}
\label{SecOpt}

We now describe how a version of composite gradient
descent~\citep{Nes07} may be applied to efficiently optimize the
nonconvex program~\eqref{EqnNonconvexReg}, and show that it enjoys a
linear rate of convergence under suitable conditions.  In this
section, we focus exclusively on a version of the optimization problem
with the side function
\begin{equation}
\label{EqnDefnSIDESPEC}
\SIDESPEC(\beta) \defn \frac{1}{\lambda} \Big \{ \rho_\lambda(\beta) +
\frac{\mu}{2} \|\beta\|_2^2 \Big \}.
\end{equation}
Note that this choice of $\SIDESPEC$ is convex by
Assumption~\ref{AsRho}. We may then write the
program~\eqref{EqnNonconvexReg} as
\begin{align}
\label{EqnAlternative}
\betahat & \in \arg \min_{\SIDESPEC(\beta) \leq R, \; \beta \in
  \Omega} \Big\{ \underbrace{\big(\Loss_n(\beta) - \frac{\mu}{2}
  \|\beta\|_2^2 \big)}_{\EmpLossBarSub} + \lambda \SIDESPEC(\beta)
\Big\}.
\end{align}
In this way, the objective function decomposes nicely into a sum of a
differentiable but nonconvex function and a possibly nonsmooth but
convex penalty. Applied to the representation~\eqref{EqnAlternative}
of the objective function, the composite gradient descent procedure
of~\cite{Nes07} produces a sequence of iterates
$\{\betait{\iter}\}_{\iter=0}^\infty$ via the updates
\begin{align}
\label{EqnCompGradUpdate}
\betait{\iter+1} \in \arg\min_{\SIDESPEC(\beta) \leq R, \; \beta \in
  \Omega} \left\{\frac{1}{2} \left\|\beta - \left(\betait{\iter} -
\frac{\nabla \EmpLossBar(\betait{\iter})}{\eta}\right)\right\|_2^2 +
\frac{\lambda}{\eta} \SIDESPEC(\beta) \right\},
\end{align}
where $\frac{1}{\eta}$ is the stepsize.  As discussed in
Section~\ref{SecUpdates}, these updates may be computed in a
relatively straightforward manner.

%%%%%%%%%%%%%%%%%%%%%%%%%%%%%%%%%%%%%%%%%%%%%%%%%%%%%%%%%%%%%%%

\subsection{Fast Global Convergence}
\label{SecFastGlobal}

The main result of this section is to establish that the algorithm
defined by the iterates~\eqref{EqnCompGradUpdate} converges very
quickly to a $\delta$-neighborhood of any global optimum, for all
tolerances $\delta$ that are of the same order (or larger) than the
statistical error.

We begin by setting up the notation and assumptions underlying our
result.  The \emph{Taylor error} around the vector $\beta_2$ in the
direction $\beta_1 - \beta_2$ is given by
\begin{align}
\label{EqnTaylorDefn}
\scriptT(\beta_1, \beta_2) & \defn \EmpLoss(\beta_1) -
\EmpLoss(\beta_2) - \inprod{\nabla \EmpLoss(\beta_2)}{\beta_1 -
  \beta_2}.
\end{align}
We analogously define the Taylor error $\scriptTBar$ for the modified
loss function $\Lossbar_n$, and note that
\begin{equation}
\label{EqnTCompare}
\scriptTBar(\beta_1, \beta_2) = \scriptT(\beta_1, \beta_2) -
\frac{\mu}{2} \|\beta_1 - \beta_2\|_2^2.
\end{equation}
For all vectors $\beta_2 \in \ball_2(3) \cap \ball_1(R)$, we require
the following form of restricted strong convexity:
\begin{subnumcases}{
\label{EqnRSC2}
\scriptT(\beta_1, \beta_2) \ge}
\label{EqnLocalRSCFirst}
\alpha_1 \|\beta_1 - \beta_2\|_2^2 - \tau_1 \frac{\log p}{n} \|\beta_1
- \beta_2\|_1^2, & \text{$\forall \|\beta_1 - \beta_2\|_2 \le 3$}, \\
\label{EqnLocalRSCSecond}
\alpha_2 \|\beta_1 - \beta_2\|_2 - \tau_2 \sqrt{\frac{\log p}{n}}
\|\beta_1 - \beta_2\|_1, & \text{$\forall \|\beta_1 - \beta_2\|_2 \ge
  3$}.
\end{subnumcases}
The conditions~\eqref{EqnRSC2} are similar but not identical to the
earlier RSC conditions~\eqref{EqnRSC}.  The main difference is that we
now require the Taylor difference to be bounded below uniformly over
$\beta_2 \in \ball_2(3) \cap \ball_1(R)$, as opposed to for a fixed
$\beta_2 = \betastar$.  In addition, we assume an analogous upper
bound on the Taylor series error:
\begin{align}
\label{EqnRSM}
\scriptT(\beta_1, \beta_2) & \leq \alpha_3 \|\beta_1 - \beta_2\|_2^2 +
\tau_3 \frac{\log \pdim}{\numobs} \|\beta_1 - \beta_2\|_1^2, \qquad
\mbox{for all $\beta_1, \beta_2 \in \Omega$,}
\end{align}
a condition referred to as \emph{restricted smoothness} in past
work~\citep{AgaEtal12}. Throughout this section, we assume $2 \alpha_i
> \mupar$ for all $i$, where $\mupar$ is the coefficient ensuring the
convexity of the function $\SIDESPEC$ from~\eqref{EqnDefnSIDESPEC}.  Furthermore, we define $\alpha =
\min \{ \alpha_1, \alpha_2 \}$ and $\tau = \max \{\tau_1, \tau_2,
\tau_3 \}$. \\

The following theorem applies to any population loss function
$\PopLoss$ for which the population minimizer $\betastar$ is
$\kdim$-sparse and $\|\betastar\|_2 \leq 1$. Similar results could be
derived for general $\|\betastar\|_2$, with the radius of the RSC
condition~\eqref{EqnLocalRSCFirst} replaced by $3 \|\betastar\|_2$ and
Lemma~\ref{LemL2Assump} in Section~\ref{SecProofFG} adjusted
appropriately, but we only include the analysis for $\|\betastar\|_2
\le 1$ in order to simplify our exposition. We also assume the scaling
$\numobs > C \kdim \log \pdim$, for a constant $C$ depending on the
$\alpha_i$'s and $\tau_i$'s.  Note that this scaling is reasonable,
since no estimator of a $\kdim$-sparse vector in $\pdim$ dimensions
can have low $\ell_2$-error unless the condition holds
(see~\citealp{RasEtal11} for minimax rates). We show that the
composite gradient updates~\eqref{EqnCompGradUpdate} exhibit a type of
\emph{globally geometric convergence} in terms of the quantity
\begin{align}
\label{EqnKappa}
\kappa \defn \frac{1 - \frac{2\alpha - \mupar}{8\eta} + \HACK}{1 -
  \HACK}, \qquad \mbox{where} \quad \HACK \defn \frac{c \tau k
  \frac{\log \pdim}{\numobs}} {2\alpha - \mupar}.
\end{align}
Under the stated scaling on the sample size, we are guaranteed that
$\kappa \in (0,1)$, so it is a \emph{contraction factor}.  Roughly
speaking, we show that the squared optimization error will fall below
$\delta^2$ within $T \asymp \frac{\log(1/\delta^2)}{\log(1/\kappa)}$
iterations. More precisely, our theorem guarantees $\delta$-accuracy
for all iterations larger than
\begin{align}
\label{EqnTBound}
T^*(\delta) & \defn \frac{2\log\left(\frac{\phi(\beta^0) -
    \phi(\betahat)}{\delta^2}\right)}{\log(1/\kappa)} + \left ( 1 +
\frac{\log 2}{\log(1/\kappa)}\right) \, \log \log\left(\frac{\lambda
  RL}{\delta^2}\right),
\end{align}
where $\phi(\beta) \defn \EmpLoss(\beta) + \myrho(\beta)$ denotes the
composite objective function.  As clarified in the theorem statement,
the squared tolerance $\delta^2$ is not allowed to be arbitrarily
small, which would contradict the fact that the composite gradient
method may converge to a stationary point. However, our theory allows
$\delta^2$ to be of the same order as the squared \emph{statistical
  error} $\epsstat^2 = \|\betahat - \betastar\|_2^2$, the distance
between a fixed global optimum and the target parameter $\betastar$.
From a statistical perspective, there is no point in optimizing beyond
this tolerance.\\

\noindent With this setup, we now turn to a precise statement of our
main optimization-theoretic result. As with
Theorems~\ref{TheoEll12Meta} and~\ref{TheoPredMeta}, the statement of
Theorem~\ref{TheoFastGlobal} is entirely deterministic.

\begin{mytheorem}
\label{TheoFastGlobal}
Suppose the empirical loss $\EmpLoss$ satisfies the RSC/RSM
conditions~\eqref{EqnRSC2} and~\eqref{EqnRSM}, and suppose the
regularizer $\myrho$ satisfies Assumption~\ref{AsRho}. Suppose
$\betahat$ is any global minimum of the
program~\eqref{EqnAlternative}, with regularization parameters chosen
such that
\begin{equation*}
\frac{8}{L} \cdot \max\left\{\|\nabla \EmpLoss(\betastar)\|_\infty, \;
c'\tau \sqrt{\frac{\log \pdim}{\numobs} } \right\} \le \lambda \le \frac{c'' \; \alpha}{RL}.
\end{equation*}
Suppose $\mu < 2\alpha$. Then for any stepsize parameter $\eta \ge
\max\{2 \alpha_3 - \mu, \, \mu\}$ and tolerance $\delta^2 \geq \frac{c
  \epsstat^2}{1- \kappa} \cdot \frac{k \log p}{n}$, we have
\begin{align}
\label{EqnEllTwoBound}
\|\beta^t - \betahat\|^2_2 & \; \leq \; \frac{4}{2\alpha - \mupar} \;
\left( \delta^2 + \frac{\delta^4}{\tau} + c \tau \frac{k \log
  \pdim}{\numobs} \epsstat^2 \right), \qquad \mbox{$\forall t \geq
  T^*(\delta)$.}
\end{align}
\end{mytheorem}

\textbf{Remark:}  
Note that for the optimal choice of tolerance parameter $\delta \asymp
\frac{k \log p}{n} \epsstat$, the error bound appearing in~\eqref{EqnEllTwoBound} takes the form $\frac{c
  \epsstat^2}{2\alpha - \mu} \cdot \frac{k \log p}{n}$, meaning that
successive iterates of the composite gradient descent algorithm are
guaranteed to converge to a region within statistical accuracy of the
true global optimum $\betahat$. Concretely, if the sample size
satisfies $n \succsim C k \log p$ and the regularization parameters
are chosen appropriately, Theorem~\ref{TheoEll12Meta} guarantees that
$\epsstat = \order\left(\sqrt{\frac{k \log p}{n}}\right)$ with high
probability.  Combined with Theorem~\ref{TheoFastGlobal}, we then
conclude that
\begin{align*}
\max \Big \{ \|\beta^t - \betahat\|_2, \; \|\beta^t - \betastar\|_2
\Big \} & = \order\left(\sqrt{\frac{k \log p}{n}}\right),
\end{align*}
for all iterations $t \geq T(\epsstat)$.

As would be expected, the (restricted) curvature $\alpha$ of the loss
function and nonconvexity parameter $\mu$ of the penalty function
enter into the bound via the denominator $2\alpha - \mu$. Indeed, the
bound is tighter when the loss function possesses more curvature or
the penalty function is closer to being convex, agreeing with
intuition. Similar to our discussion in the remark following
Theorem~\ref{TheoPredMeta}, the requirement $\mu < 2\alpha$ is
certainly necessary for our proof technique, but it is possible that
composite gradient descent still produces good results when this
condition is violated. See Section~\ref{SecSims} for simulations in
scenarios involving mild and severe violations of this condition.

Finally, note that the parameter $\eta$ must be sufficiently large (or
equivalently, the stepsize must be sufficiently small) in order for
the composite gradient descent algorithm to be
well-behaved. See~\cite{Nes07} for a discussion of how the stepsize
may be chosen via an iterative search when the problem parameters are
unknown.

In the case of corrected linear regression
(Corollary~\ref{CorLinear}), Lemma 13 of~\cite{LohWai11a} establishes
the RSC/RSM conditions for various statistical models. The following
proposition shows that the conditions~\eqref{EqnRSC2}
and~\eqref{EqnRSM} hold in GLMs when the $x_i$'s are drawn
i.i.d.\ from a zero-mean sub-Gaussian distribution with parameter
$\sigma_x^2$ and covariance matrix $\CovX = \cov(x_i)$. As usual, we
assume a sample size $\numobs \ge c \, \kdim \log \pdim$, for a
sufficiently large constant $c > 0$. Recall the definition of the
Taylor error $\scriptT(\beta_1, \beta_2)$ from~\eqref{EqnTaylorDefn}.

\begin{myproposition}
	\label{PropGLM} [RSC/RSM conditions for generalized linear models]
There exists a constant $\alpha_\ell > 0$, depending only on the GLM
and the parameters $(\sigma_x^2, \CovX)$, such that for all vectors $\beta_2 \in
\ball_2(3) \cap \ball_1(R)$, we have
\begin{subnumcases}{
	\label{EqnRSCglm}
	\scriptT(\beta_1, \beta_2) \geq} %
	\label{EqnGLMFirst}
	 \frac{\alpha_\ell}{2} \|\Delta\|_2^2 - \frac{c^2
           \sigma_x^2}{2\alpha_\ell} \frac{\log p}{n} \|\Delta\|_1^2,
         & \text{for all $\|\beta_1 - \beta_2\|_2 \le 3$,} \\ %
	\label{EqnGLMSecond}
	\frac{3\alpha_\ell}{2} \|\Delta\|_2 - 3 c \sigma_x
        \sqrt{\frac{\log p}{n}} \|\Delta\|_1, & \text{for all $\:
          \|\beta_1 - \beta_2\|_2 \ge 3$},
\end{subnumcases}
with probability at least $1 - c_1 \exp (-c_2 n)$. With the bound
$\|\psi''\|_\infty \leq \alpha_u$, we also have
\begin{align}
\label{EqnRSMglm}
\scriptT(\beta_1, \beta_2) & \leq \alpha_u \lambda_{max}(\CovX) \;
\left( \frac{3}{2} \|\Delta\|_2^2 + \frac{\log p}{n}
\|\Delta\|_1^2\right), \qquad \mbox{for all $\beta_1, \beta_2 \in
  \real^\pdim$},
\end{align}
with probability at least $1 - c_1 \exp(-c_2 \numobs)$.
\end{myproposition}

For the proof of Proposition~\ref{PropGLM}, see Appendix~\ref{AppGLM}.

%%%%%%%%%%%%%%%%%%%%%%%%%%%%%%%%%%%%%%%%%%%%%%%%%%%%%%%%%%%%%%%%%%%%%%%%%%%%

\subsection{Form of Updates}
\label{SecUpdates}

In this section, we discuss how the updates~\eqref{EqnCompGradUpdate}
are readily computable in many cases.  We begin with the case $\Omega
= \real^\pdim$, so we have no additional constraints apart from
$\SIDESPEC(\beta) \leq R$.  In this case, given iterate $\betait{t}$,
the next iterate $\betait{t+1}$ may be obtained via the following
three-step procedure:

\begin{enumerate}
\item[(1)] First optimize the unconstrained program
  \begin{equation}
    \label{EqnCompUnconstr}
    \betahat \in \arg \min_{\beta \in \real^\pdim} \left\{ \frac{1}{2}
    \left\|\beta - \left(\beta^t - \frac{\nabla
      \Lossbar_n(\beta^t)}{\eta}\right)\right\|_2^2 +
    \frac{\lambda}{\eta} \cdot \SIDESPEC(\beta)\right\}.
  \end{equation}
\item[(2)] If $\SIDESPEC(\betahat) \leq R$, define $\beta^{t+1} =
  \betahat$.
\item[(3)] Otherwise, if $\SIDESPEC(\betahat) > R$, optimize the
  constrained program
\begin{align}
\label{EqnRhobarProj}
\beta^{t+1} & \in \arg \min_{\SIDESPEC(\beta) \leq R} \left \{
\frac{1}{2} \left\| \beta - \left(\beta^t - \frac{\nabla
  \Lossbar_n(\beta^t)} {\eta} \right ) \right\|_2^2 \right\}.
\end{align}
\end{enumerate}

We derive the correctness of this procedure in
Appendix~\ref{AppCompIterate}. For many nonconvex regularizers
$\myrho$ of interest, the unconstrained
program~\eqref{EqnCompUnconstr} has a convenient closed-form solution:
For the SCAD penalty~\eqref{EqnSCADdefn}, the
program~\eqref{EqnCompUnconstr} has simple closed-form solution given
by
\begin{equation}
\label{EqnXSCAD}
\betahat_{\mbox{\tiny{SCAD}}} =
\begin{cases}
0 & \mbox{if } 0 \le |z| \le \nu\lambda, \\ z - \sign(z) \cdot
\nu\lambda & \mbox{if } \nu\lambda \le |z| \le (\nu + 1) \lambda, \\
\frac{z - \sign(z) \cdot \frac{a\nu\lambda}{a-1}}{1 - \frac{\nu}{a-1}}
& \mbox{if } (\nu + 1) \lambda \le |z| \le a\lambda, \\ z & \mbox{if }
|z| \ge a\lambda.
\end{cases}
\end{equation}
For the MCP~\eqref{EqnMCPdefn}, the optimum of the
program~\eqref{EqnCompUnconstr} takes the form
\begin{equation}
\label{EqnXMCP}
\betahat_{\mbox{\tiny{MCP}}} = \begin{cases} 0 & \mbox{if $0 \le |z|
    \le \nu\lambda$,} \\
\frac{z - \sign(z) \cdot \nu\lambda}{1 - \nu/b} & \mbox{if $\nu\lambda
  \le |z| \le b\lambda$,} \\
z & \mbox{if $|z| \ge b\lambda$.}
\end{cases}
\end{equation}		
In both~\eqref{EqnXSCAD} and~\eqref{EqnXMCP}, we have
\begin{equation*}
	 z \defn \frac{1}{1 + \mu/\eta} \left(\beta^t - \frac{\nabla
           \Lossbar_n(\beta^t)}{\eta}\right), \qquad \text{and} \qquad
         \nu \defn \frac{1/\eta} {1 + \mu/\eta},
\end{equation*}
and the operations are taken componentwise. See
Appendix~\ref{AppUpdates} for the derivation of these closed-form
updates.

More generally, when $\Omega \subsetneq \real^\pdim$ (such as in the
case of the graphical Lasso), the minimum in the
program~\eqref{EqnCompGradUpdate} must be taken over $\Omega$, as
well. Although the updates are not as simply stated, they still
involve solving a convex optimization problem. Despite this more
complicated form, however, our results from
Section~\ref{SecFastGlobal} on fast global convergence under
restricted strong convexity and restricted smoothness assumptions
carry over without modification, since they only require RSC/RSM
conditions holding over a sufficiently small radius together with
feasibility of $\betastar$.

%%%%%%%%%%%%%%%%%%%%%%%%%%%%%%%%%%%%%%%%%%%%%%%%%%%%%%%%%%%%%%%%%%%%%%%%%%%%

\subsection{Proof of Theorem~\ref{TheoFastGlobal}}
\label{SecProofFG}
	
We provide the outline of the proof here, with more technical results
deferred to Appendix~\ref{AppFastGlobal}.  In broad terms, our proof
is inspired by a result of~\cite{AgaEtal12}, but requires various
modifications in order to be applied to the much larger family of
nonconvex regularizers considered here.

Our first lemma shows that the optimization error $\beta^t - \betahat$
lies in an approximate cone set:

\begin{mylemma}
\label{LemICB}
Under the conditions of Theorem~\ref{TheoFastGlobal}, suppose there
exists a pair $(\etabar, T)$ such that
\begin{equation}
\label{EqnObjTol}
\phi(\beta^t) - \phi(\betahat) \le \etabar, \qquad \forall t \ge T.
\end{equation}
Then for any iteration $t \ge T$, we have
\begin{equation*}
\|\beta^t - \betahat\|_1 \le 8\sqrt{k} \|\beta^t - \betahat\|_2 +
16\sqrt{k} \|\betahat - \betastar\|_2 + 2 \cdot
\min\left(\frac{2\etabar}{\lambda L}, R\right).
\end{equation*}
\end{mylemma}

Our second lemma shows that as long as the composite gradient descent
algorithm is initialized with a solution $\beta^0$ within a constant
radius of a global optimum $\betahat$, all successive iterates also
lie within the same ball:

\begin{mylemma}
\label{LemL2Assump}
Under the conditions of Theorem~\ref{TheoFastGlobal}, and with an
initial vector $\beta^0$ such that $\|\beta^0 - \betahat\|_2 \le 3$,
we have
\begin{equation}
\label{EqnL2Assump}
\|\beta^t - \betahat\|_2 \le 3, \qquad \mbox{for all $t \ge 0$.}
\end{equation}
\end{mylemma}

In particular, suppose we initialize the composite gradient procedure
with a vector $\beta^0$ such that \mbox{$\|\beta^0\|_2 \le
  \frac{3}{2}$}. Then by the triangle inequality,
\begin{equation*}
\|\beta^0 - \betahat\|_2 \le \|\beta^0\|_2 + \|\betahat -
\betastar\|_2 + \|\betastar\|_2 \le 3,
\end{equation*}
where we have assumed our scaling of $\numobs$ guarantees $\|\betahat
- \betastar\|_2 \leq 1/2$.

Finally, recalling our earlier definition~\eqref{EqnKappa} of
$\kappa$, the third lemma combines the results of Lemmas~\ref{LemICB}
and~\ref{LemL2Assump} to establish a bound on the value of the
objective function that decays exponentially with $t$:

\begin{mylemma}
\label{LemIterate}
Under the same conditions of Lemma~\ref{LemL2Assump}, suppose in
addition that~\eqref{EqnObjTol} holds and $\frac{32k \tau
  \log p}{n} \le \frac{2\alpha - \mu}{4}$. Then for any $t \ge T$, we
have
\begin{equation*}
\phi(\beta^t) - \phi(\betahat) \le \kappa^{t - T} (\phi(\beta^T) -
\phi(\betahat)) + \frac{\xi}{1-\kappa} (\epsilon^2 + \epsbar^2),
\end{equation*}
where $\epsbar \defn 8\sqrt{k} \epsstat$, $\epsilon \defn 2 \cdot \min
\left(\frac{2\etabar}{\lambda L}, R\right)$, the quantities $\kappa$
and $\varphi$ are defined according to~\eqref{EqnKappa}, and
\begin{align}
 \label{EqnXi}
\xi & \defn \frac{1}{1 - \varphi(n,p,k)} \cdot \frac{\tau \log p}{n}
\cdot \left(\frac{2\alpha - \mu}{4\eta} + 2\varphi(n,p,k) + 5\right).
\end{align}
\end{mylemma}

\vspace{1cm}

The remainder of the proof follows an argument used
in~\cite{AgaEtal12}, so we only provide a high-level sketch. We first
prove the following inequality:
\begin{equation}
\label{EqnYoda}
\phi(\beta^t) - \phi(\betahat) \le \delta^2, \qquad \mbox{for all $t
  \geq T^*(\delta)$,}
\end{equation}
as follows. We divide the iterations $t \ge 0$ into a series of epochs
$[T_\ell, T_{\ell + 1})$ and define tolerances $\etabar_0 > \etabar_1
  > \cdots$ such that
\begin{equation*}
 \phi(\beta^t) - \phi(\betahat) \le \etabar_\ell, \qquad \forall t \ge
 T_\ell.
\end{equation*}
In the first iteration, we apply Lemma~\ref{LemIterate} with
$\etabar_0 = \phi(\beta^0) - \phi(\betahat)$ to obtain
\begin{equation*}
\phi(\beta^t) - \phi(\betahat) \le \kappa^t \left(\phi(\beta^0) -
\phi(\betahat)\right) + \frac{\xi}{1 - \kappa}(4R^2 + \epsbar^2),
\qquad \forall t \ge 0.
\end{equation*}
Let $\etabar_1 \defn \frac{2\xi}{1-\kappa} (4R^2 + \epsbar^2)$, and
note that for $T_1 \defn \Bigg \lceil
\frac{\log(2\etabar_0/\etabar_1)}{\log(1/\kappa)} \Bigg \rceil$, we
have
\begin{equation*}
\phi(\beta^t) - \phi(\betahat) \le \etabar_1 \le \frac{4\xi}{1-\kappa}
\max\{4R^2, \epsbar^2\}, \qquad \mbox{for all $t \ge T_1$.}
\end{equation*}
For $\ell \ge 1$, we now define
\begin{align*}
\etabar_{\ell+1} \defn \frac{2\xi}{1-\kappa} (\epsilon_\ell^2 +
\epsbar^2), \quad \mbox{and} \quad T_{\ell+1} \defn \Bigg \lceil
\frac{\log(2\etabar_\ell/\etabar_{\ell+1})}{\log(1/\kappa)} \Bigg
\rceil + T_\ell,
\end{align*}
where $\epsilon_\ell \defn 2\min\left\{\frac{\etabar_\ell}{\lambda L},
R\right\}$.  From Lemma~\ref{LemIterate}, we have
\begin{align*}
  \phi(\beta^t) - \phi(\betahat) & \leq \kappa^{t - T_\ell}
  \left(\phi(\beta^{T_\ell}) - \phi(\betahat)\right) +
  \frac{\xi}{1-\kappa} (\epsilon_\ell^2 + \epsbar^2), \qquad \mbox{for
    all $t \ge T_\ell$,}
\end{align*}
implying by our choice of $\{(\eta_\ell, T_\ell)\}_{\ell \ge 1}$ that
\begin{equation*}
\phi(\beta^t) - \phi(\betahat) \le \etabar_{\ell+1} \le
\frac{4\xi}{1-\kappa} \max\{\epsilon_\ell^2, \epsbar^2\}, \qquad
\forall t \ge T_{\ell+1}.
\end{equation*}
Finally, we use the recursion
\begin{equation}
\label{EqnRec1}
\etabar_{\ell+1} \le \frac{4\xi}{1-\kappa} \max\{\epsilon_\ell^2,
\epsbar^2\}, \qquad T_\ell \le \ell +
\frac{\log(2^\ell\etabar_0/\etabar_\ell)}{\log(1/\kappa)},
\end{equation}
to establish the recursion
\begin{equation}
\label{EqnRec2}
\etabar_{\ell+1} \le \frac{\etabar_\ell}{4^{2^{\ell-1}}}, \qquad
\frac{\etabar_{\ell+1}}{\lambda L} \le \frac{R}{4^{2^\ell}}.
\end{equation}
Inequality~\eqref{EqnYoda} then follows from computing the number of
epochs and timesteps necessary to obtain $\frac{\lambda R
  L}{4^{2^{\ell-1}}} \le \delta^2$. For the remaining steps used to
obtain~\eqref{EqnRec2} from~\eqref{EqnRec1},
we refer the reader to~\cite{AgaEtal12}.

Finally, by~\eqref{EqnIterateOne} in the proof of
Lemma~\ref{LemIterate} in Appendix~\ref{AppLemIterate} and the
relative scaling of $(n,p,k)$, we have
\begin{align*}
\frac{2\alpha - \mu}{4} \|\beta^t - \betahat\|_2^2 & \le \phi(\beta^t)
- \phi(\betahat) + 2\tau \frac{\log p}{n}
\left(\frac{2\delta^2}{\lambda L} + \epsbar\right)^2 \\
& \le \delta^2 + 2\tau \frac{\log p}{n} \left(\frac{2\delta^2}{\lambda
  L}+ \epsbar\right)^2,
\end{align*}
where we have set $\epsilon = \frac{2\delta^2}{\lambda
  L}$. Rearranging and performing some algebra with our choice of
$\lambda$ gives the $\ell_2$-bound.

%%%%%%%%%%%%%%%%%%%%%%%%%%%%%%%%%%%%%%%%%%%%%%%%%%%%%%%%%%%%%%%%%%%%%%%%%

\section{Simulations}
\label{SecSims}

In this section, we report the results of simulations we performed to
validate our theoretical results. In particular, we present results
for two versions of the loss function $\EmpLoss$, corresponding to
linear and logistic regression, and three penalty functions, namely
the $\ell_1$-norm (Lasso), the SCAD penalty, and the MCP, as detailed
in Section~\ref{SecNonconvexRegExas}. In all cases, we chose
regularization parameters $R = \frac{1.1}{\lambda} \cdot
\rho_\lambda(\betastar)$, to ensure feasibility of $\betastar$, and
$\lambda = \sqrt{\frac{\log p}{n}}$; in practical applications where
$\betastar$ is unknown, we would need to tune $\lambda$ and $R$ using
a method such as cross-validation. \\

\textbf{Linear regression:} In the case of linear 
regression, we simulated covariates corrupted by additive noise
according to the mechanism described in Section~\ref{SecCorrLinear},
giving the estimator
\begin{equation}
\label{EqnCorrLinearEst}
\betahat \in \arg\min_{g_{\lambda, \mu}(\beta) \le R}
\left\{\frac{1}{2} \beta^T \left(\frac{Z^T Z}{n} - \Sigma_w\right)
\beta - \frac{y^T Z}{n} \beta + \rho_\lambda(\beta)\right\}.
\end{equation}
We generated i.i.d.\ samples $x_i \sim N(0, I)$ and set $\Sigma_w =
(0.2)^2 I$, and generated additive noise $\epsilon_i \sim N(0,
(0.1)^2)$. \\

\textbf{Logistic regression:}  In the case of logistic 
regression, we also generated i.i.d.\ samples $x_i \sim N(0,
I)$. Since $\psi(t) = \log(1+\exp(t))$, the program~\eqref{EqnGLMEst}
becomes
\begin{equation}
\label{EqnLogisticEst}
\betahat \in \arg \min_{g_{\lambda, \mu} (\beta) \le R }
\left\{\frac{1}{\numobs} \sum_{i=1}^\numobs \left \{\log(1 +
\exp(\inprod{\beta}{x_i}) - y_i \inprod{\beta}{x_i} \right \} + \myrho
(\beta) \right\}.
\end{equation}  \\

We optimized the programs~\eqref{EqnCorrLinearEst}
and~\eqref{EqnLogisticEst} using the composite gradient
updates~\eqref{EqnCompGradUpdate}.  In order to compute the updates,
we used the three-step procedure described in
Section~\ref{SecUpdates}, together with the updates for SCAD and MCP
given by~\eqref{EqnXSCAD} and~\eqref{EqnXMCP}. Note that the
updates for the Lasso penalty may be generated more simply and
efficiently as discussed in~\cite{AgaEtal12}.

Figure~\ref{FigScaling} shows the results of corrected linear
regression with Lasso, SCAD, and MCP regularizers for three different
problem sizes $\pdim$. In each case, $\betastar$ is a $k$-sparse
vector with \mbox{$k = \lfloor \sqrt{\pdim} \rfloor$,} where the
nonzero entries were generated from a normal distribution and the
vector was then rescaled so that $\|\betastar\|_2 = 1$. As predicted by
Theorem~\ref{TheoEll12Meta}, the three curves corresponding to the
same penalty function stack up when the estimation error
$\|\betahat - \betastar\|_2$ is plotted against the rescaled sample
size $\frac{n}{k \log p}$, and the $\ell_2$-error decreases to zero as
the number of samples increases, showing that the
estimators~\eqref{EqnCorrLinearEst} and~\eqref{EqnLogisticEst} are
statistically consistent. The Lasso, SCAD, and MCP regularizers are
depicted by solid, dotted, and dashed lines, respectively. We chose
the parameter $a = 3.7$ for the SCAD penalty, suggested
by~\cite{FanLi01} to be ``optimal" based on cross-validated empirical
studies, and chose $b = 3.5$ for the MCP. Each point represents an
average over 20 trials.

\begin{figure}[h!]
\begin{center}
 \begin{tabular}{cc}
   \widgraph{.45\textwidth}{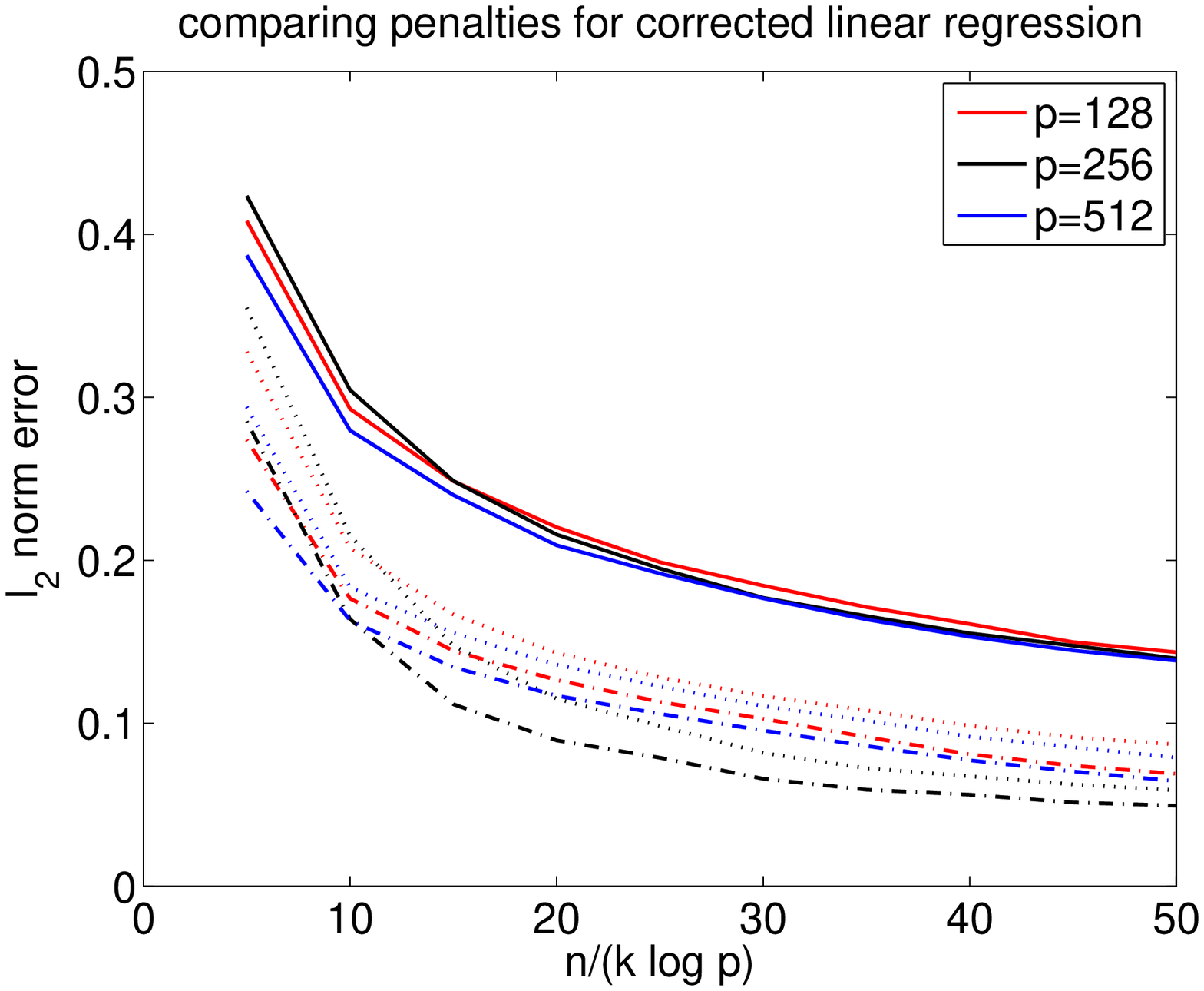} &
   \widgraph{.45\textwidth}{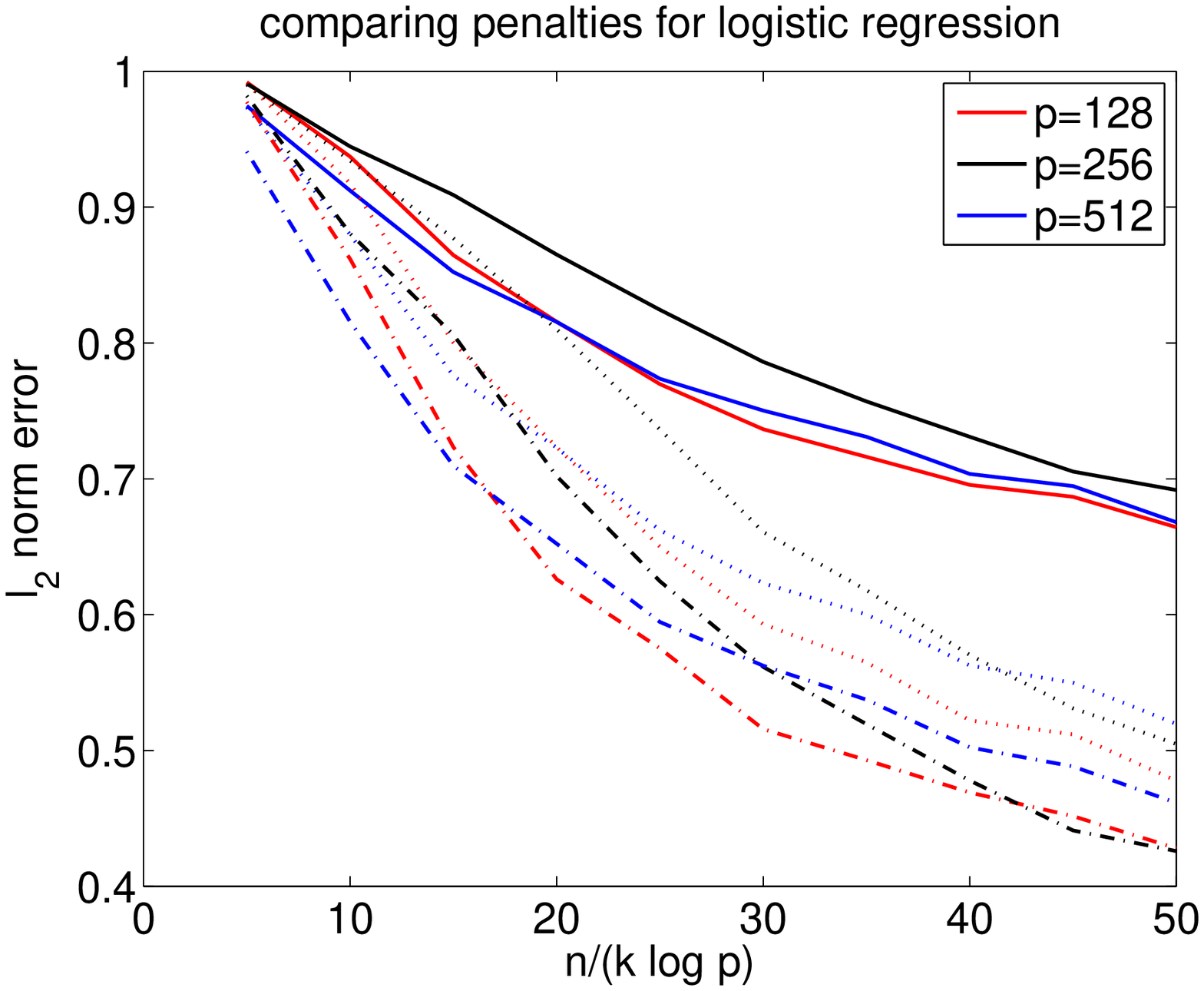} \\
   (a) & (b) \\
 \end{tabular}
\end{center}
\caption{Plots showing statistical consistency of linear and logistic
  regression with Lasso, SCAD, and MCP regularizers, and with sparsity
  level \mbox{$k = \lfloor \sqrt{p} \rfloor$.} Panel (a) shows results
  for corrected linear regression, where covariates are subject to
  additive noise with $SNR = 5$. Panel (b) shows similar results for
  logistic regression. Each point represents an average over 20
  trials. In both cases, the estimation error $\|\betahat -
  \betastar\|_2$ is plotted against the rescaled sample size
  $\frac{n}{k \log p}$. Lasso, SCAD, and MCP results are represented
  by solid, dotted, and dashed lines, respectively. As predicted by
  Theorem~\ref{TheoEll12Meta} and Corollaries~\ref{CorLinear}
  and~\ref{CorGLM}, the curves for each of the three types stack up
  for different problem sizes $p$, and the error decreases to zero as
  the number of samples increases, showing that our methods are
  statistically consistent.}
	\label{FigScaling}
\end{figure}

The simulations in Figure~\ref{FigLinLogerrs} depict the
optimization-theoretic conclusions of
Theorem~\ref{TheoFastGlobal}. Each panel shows two different families
of curves, depicting the statistical error $\log(\|\betahat -
\betastar\|_2)$ in red and the optimization error $\log(\|\beta^t -
\betahat\|_2)$ in blue. Here, the vertical axis measures the
$\ell_2$-error on a logarithmic scale, while the horizontal axis
tracks the iteration number.  Within each panel, the blue curves were
obtained by running the composite gradient descent algorithm from $10$
different initial starting points chosen at random, and the
optimization error is measured with respect to a stationary point
obtained from an earlier run of the composite gradient descent
algorithm in place of $\betahat$, since a global optimum is
unknown. The statistical error is similarly displayed as the distance
between $\betastar$ and the stationary points computed from successive
runs of composite gradient descent. In all cases, we used the
parameter settings $\pdim = 128$, $k = \lfloor \sqrt{\pdim} \rfloor$,
and $\numobs = \lfloor 20 k \log p \rfloor$. As predicted by our
theory, the optimization error decreases at a linear rate (on the log
scale) until it falls to the level of statistical error. Furthermore,
it is interesting to compare the plots in panels (c) and (d), which
provide simulation results for two different values of the SCAD
parameter $a$. We see that the choice $a = 3.7$ leads to a tighter
cluster of optimization trajectories, providing further evidence that
this setting suggested by~\cite{FanLi01} is in some sense optimal.
\begin{figure}[h!]
  \begin{center}
    \begin{tabular}{cc}
      \widgraph{.45\textwidth}{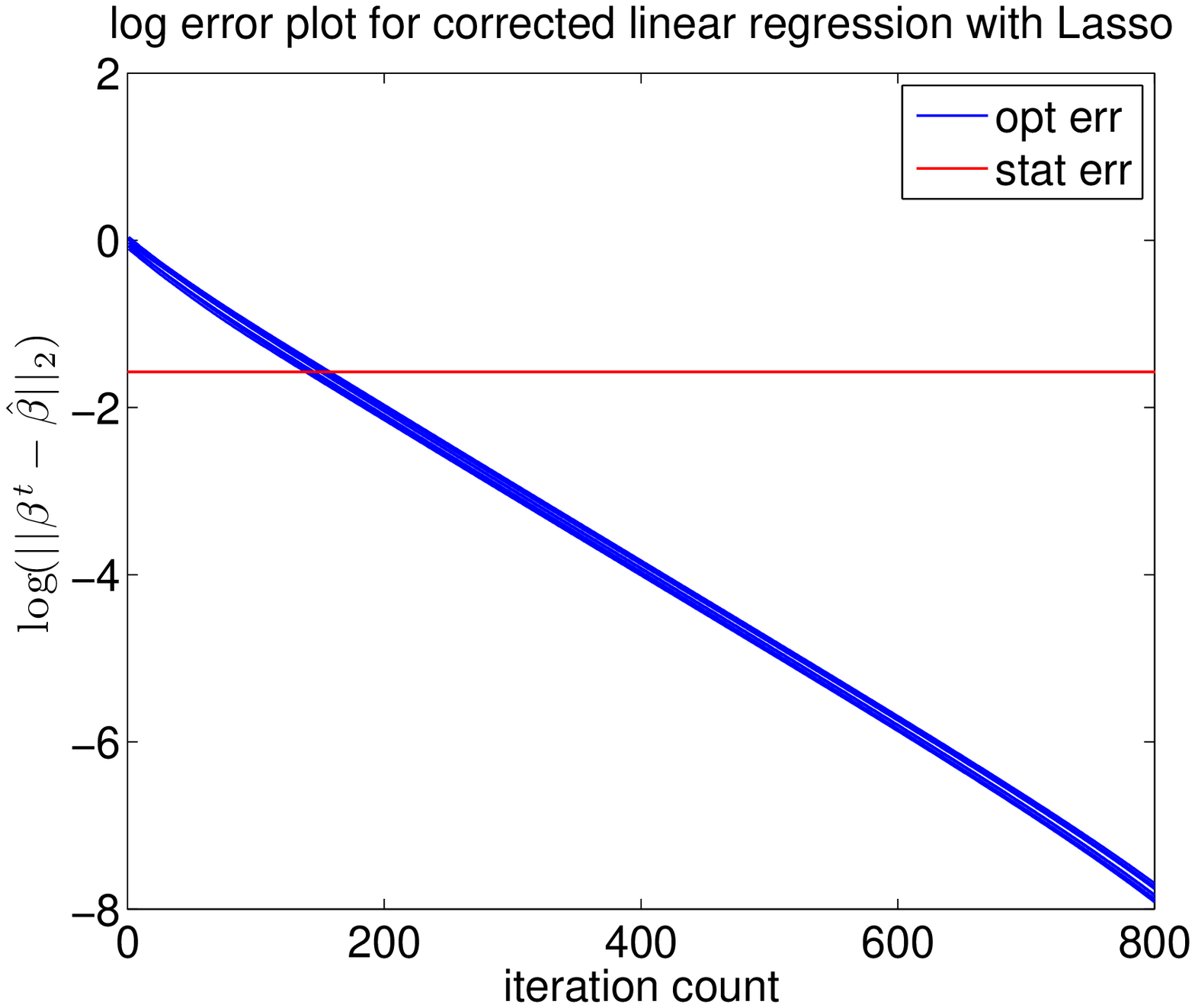} &
      \widgraph{.45\textwidth}{logerrs_lin_mcp_3} \\ % (a) & (b) \\ %
      \widgraph{.45\textwidth}{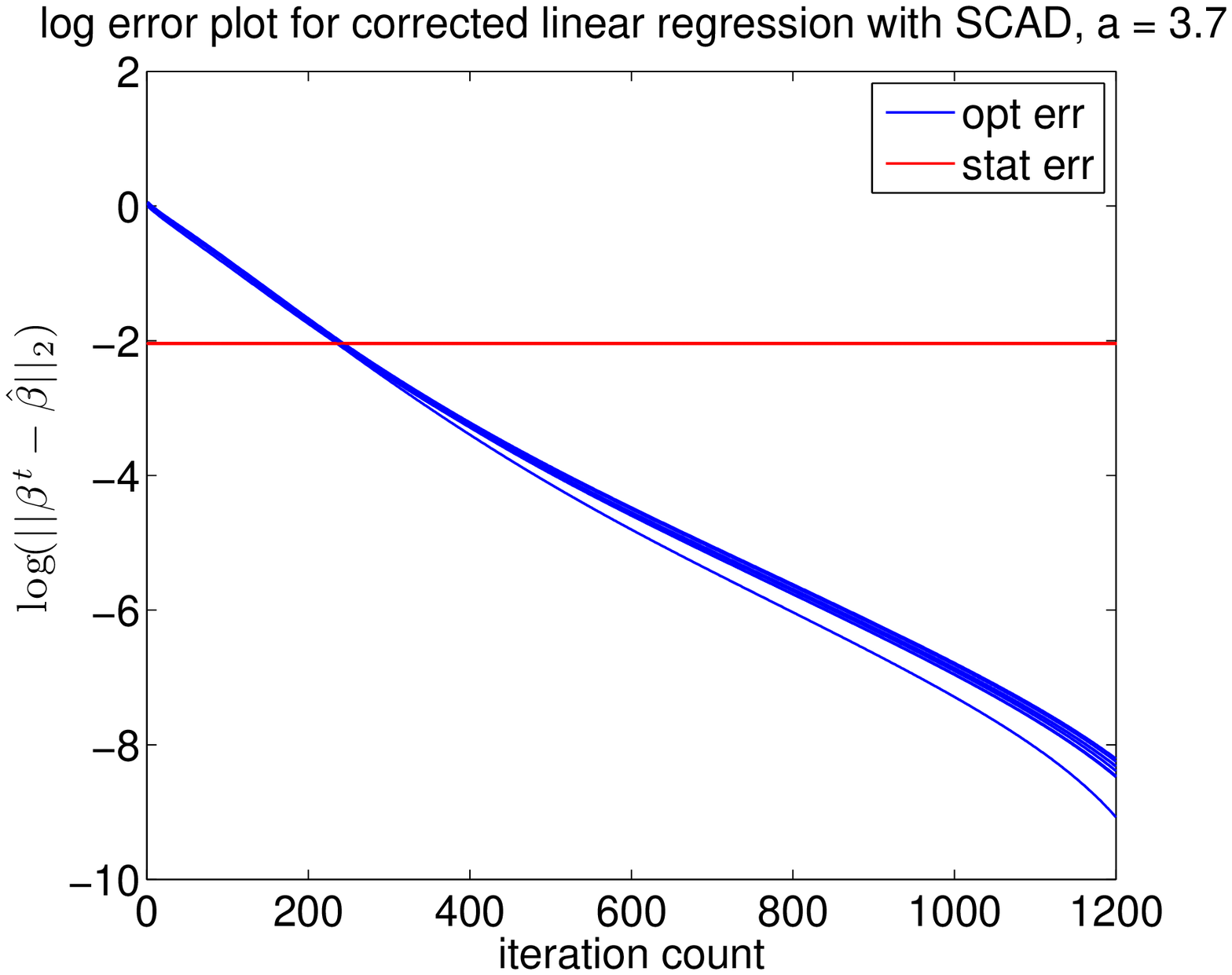} &
      \widgraph{.45\textwidth}{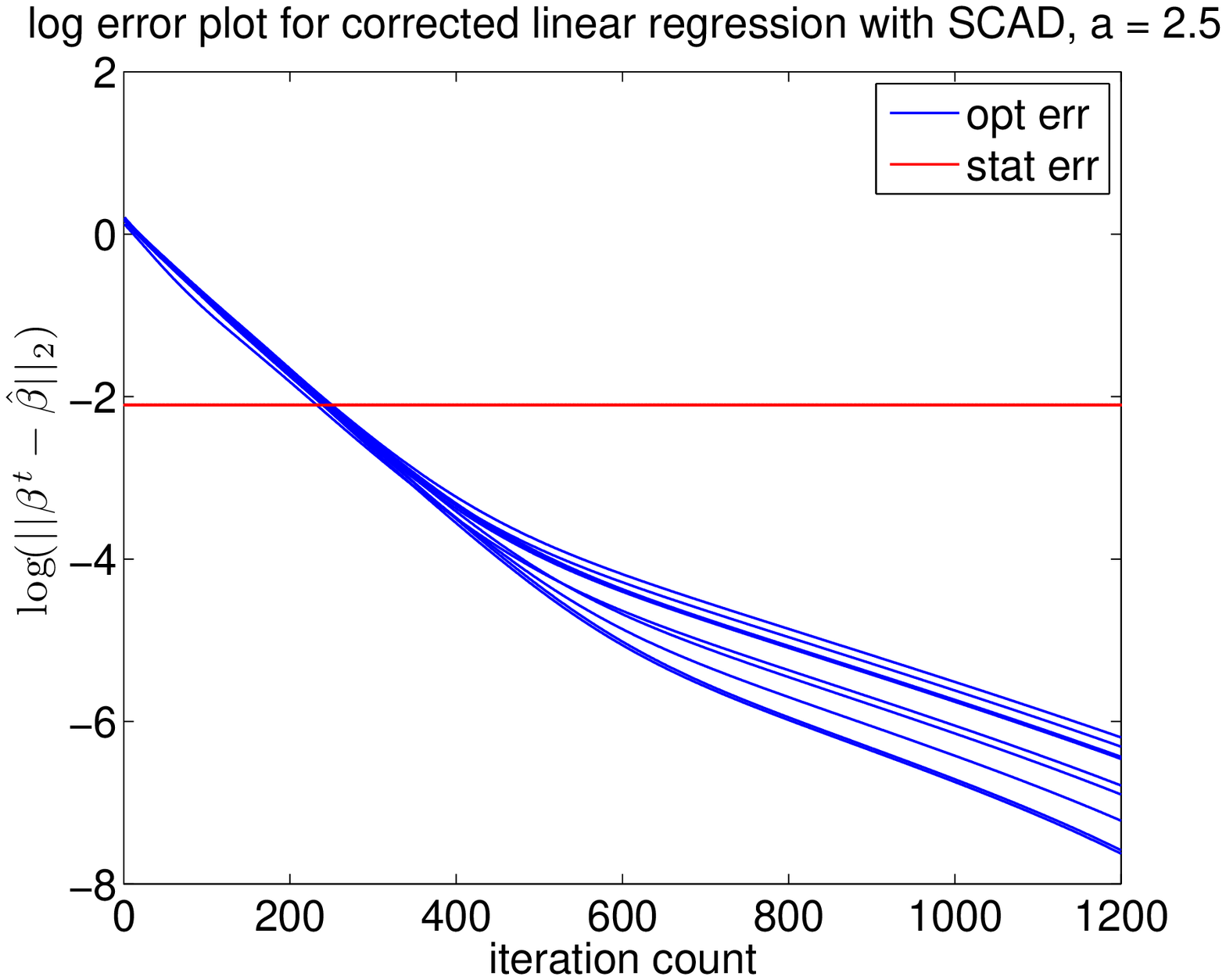} \\ % (c) & (d)
    \end{tabular}
  \end{center}
  \caption{Plots illustrating linear rates of convergence on a log
    scale for corrected linear regression with Lasso, MCP, and SCAD
    regularizers, with $\pdim = 128$, $k = \lfloor \sqrt{\pdim}
    \rfloor$, and $n = \lfloor 20 k \log p \rfloor$, where covariates
    are corrupted by additive noise with $SNR = 5$. Red lines depict
    statistical error $\log\big(\|\betahat - \betastar\|_2 \big)$ and
    blue lines depict optimization error $\log \big(\|\beta^t -
    \betahat\|_2 \big)$. As predicted by Theorem~\ref{TheoFastGlobal},
    the optimization error decreases linearly when plotted against the
    iteration number on a log scale, up to statistical accuracy. Each
    plot shows the solution trajectory for 10 different
    initializations of the composite gradient descent
    algorithm. Panels (a) and (b) show the results for Lasso and MCP
    regularizers, respectively; panels (c) and (d) show results for
    the SCAD penalty with two different parameter values. Note that
    the empirically optimal choice $a = 3.7$ proposed
    by~\cite{FanLi01} generates solution paths that exhibit a smaller
    spread than the solution paths generated for a smaller setting of
    the parameter $a$.}
\label{FigLinLogerrs}
\end{figure}

Figure~\ref{FigLogisticLogerrs} provides analogous results to
Figure~\ref{FigLinLogerrs} in the case of logistic regression, using
$p = 64, k = \lfloor \sqrt{p} \rfloor$, and $n = \lfloor 20 k \log p
\rfloor$. The plot shows solution trajectories for 20 different
initializations of composite gradient descent. Again, we see that the
log optimization error decreases at a linear rate up to the level of
statistical error, as predicted by
Theorem~\ref{TheoFastGlobal}. Furthermore, the Lasso penalty yields a
unique global optimum $\betahat$, since the
program~\eqref{EqnLogisticEst} is convex, as we observe in panel
(a). In contrast, the nonconvex program based on the SCAD penalty
produces multiple local optima, whereas the MCP yields a relatively
large number of local optima. Note that empirically, all local optima
appear to lie within the small ball around $\betastar$ defined in
Theorem~\ref{TheoEll12Meta}. However, if we use
$\lambda_{\min}(\nabla^2 \Loss_n(\betastar))$ as a surrogate for
$\alpha_1$, we see that $2\alpha_1 < \mu$ in the case of the SCAD or
MCP regularizers, which is not covered by our theory.

\begin{figure}[h!]
  \begin{center}
    \begin{tabular}{cc}
      \widgraph{.45\textwidth}{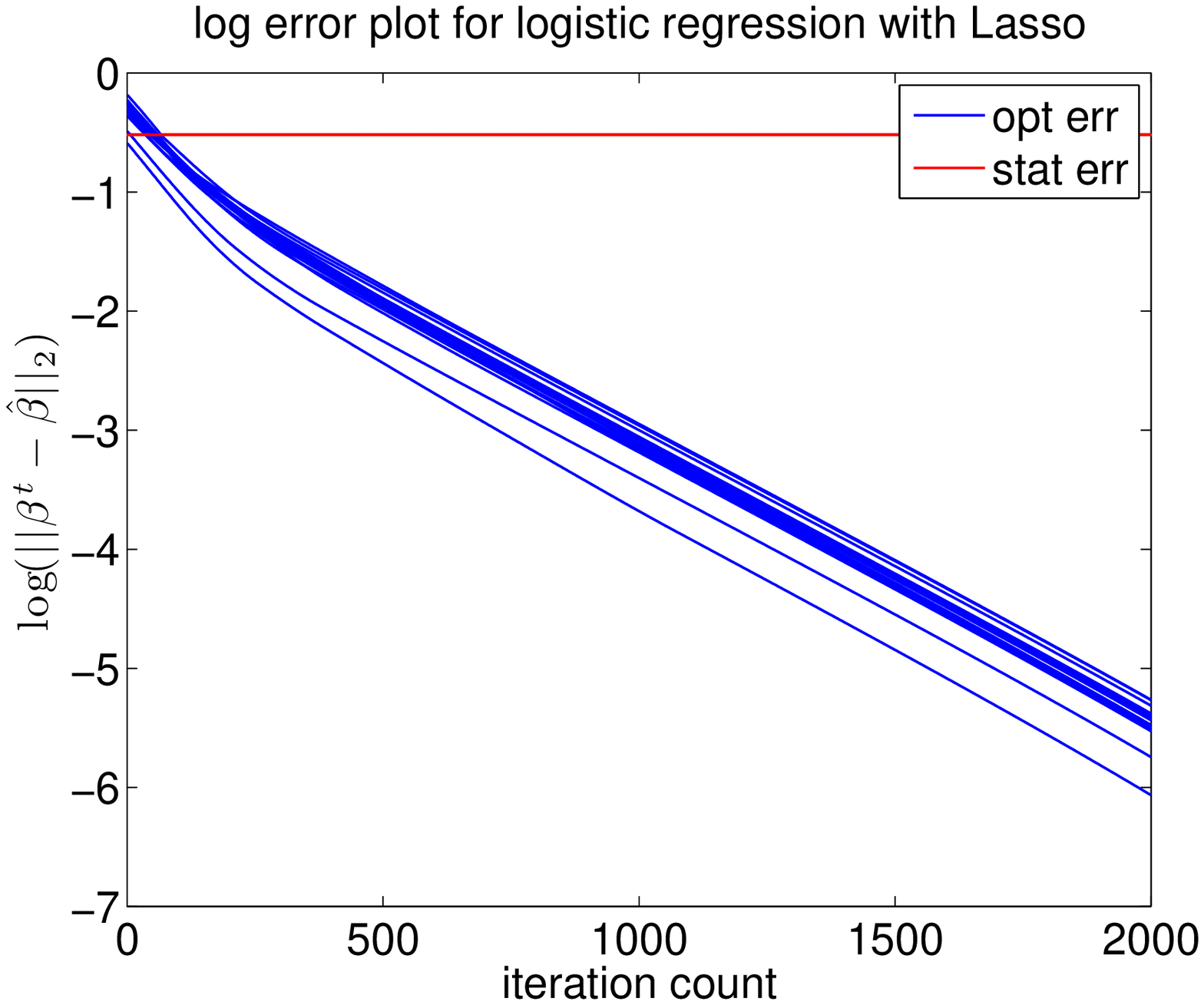} &
      \widgraph{.45\textwidth}{logerrs_logistic_scad_3} \\
      (a) & (b) \\
    \end{tabular}
    \begin{tabular}{c}
      \widgraph{.45\textwidth}{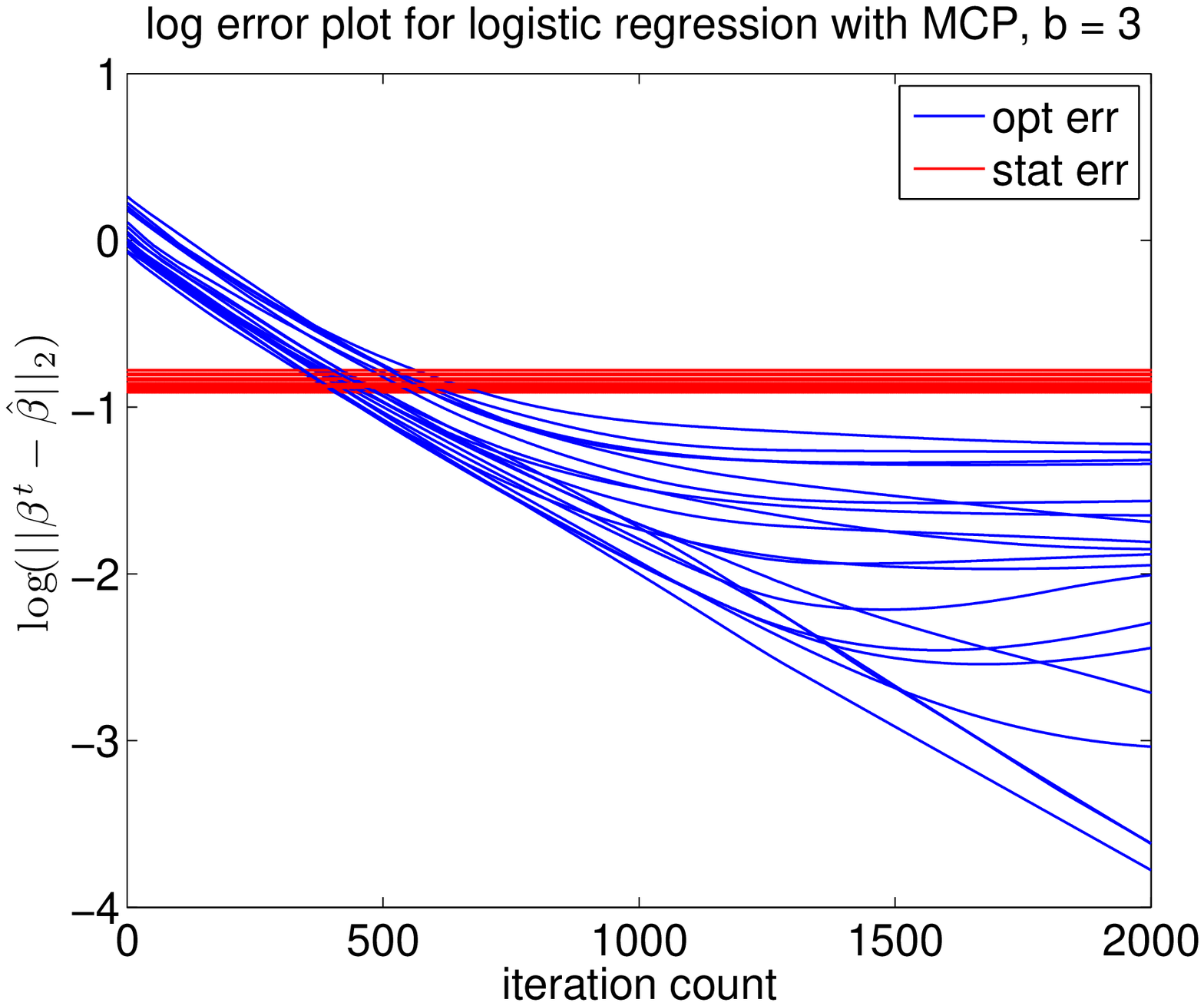} \\ (c)
    \end{tabular}
  \end{center}
  \caption{Plots that demonstrate linear rates of convergence on a log
    scale for logistic regression with $p = 64, k = \sqrt{p}$, and $n
    = \lfloor 20 k \log p \rfloor$. Red lines depict statistical error
    $\log\big(\|\betahat - \betastar\|_2\big)$ and blue lines depict
    optimization error $\log\big(\|\beta^t - \betahat\|_2\big)$. (a)
    Lasso penalty.  (b) SCAD penalty.  (c) MCP.  As predicted by
    Theorem~\ref{TheoFastGlobal}, the optimization error decreases
    linearly when plotted against the iteration number on a log scale,
    up to statistical accuracy. Each plot shows the solution
    trajectory for 20 different initializations of the composite
    gradient descent algorithm. Multiple local optima emerge in panels
    (b) and (c), due to nonconvex regularizers.}
\label{FigLogisticLogerrs}
\end{figure}
			
Finally, Figure~\ref{FigBreakdown} explores the behavior of our
algorithm when the condition $\mu < 2\alpha_1$ from
Theorem~\ref{TheoEll12Meta} is significantly violated. We generated
i.i.d.\ samples $x_i \sim N(0, \Sigma)$, with $\Sigma$ taken to be a
Toeplitz matrix with entries $\Sigma_{ij} = \zeta^{|i-j|}$, for some
parameter $\zeta \in [0,1)$, so that $\lambda_{\min}(\Sigma) \geq (1 -
  \zeta)^2$. We chose $\zeta \in \{0.5, 0.9\}$, resulting in $\alpha_1
  \approx \{ 0.25, 0.01 \}$. The problem parameters were chosen to be
  $p = 512, k = \lfloor \sqrt{p} \rfloor$, and $n = \lfloor 10 k \log
  p\rfloor$. Panel (a) shows the expected good behavior of
  $\ell_1$-regularization, even for $\alpha_1 = 0.01$; although
  convergence is slow and the overall statistical error is greater
  than for $\Sigma = I$ (cf.\ Figure~\ref{FigLinLogerrs}(a)),
  composite gradient descent still converges at a linear rate. Panel
  (b) shows that for SCAD parameter $a = 2.5$ (corresponding to $\mu
  \approx 0.67$), local optima still seem to be well-behaved even for
  $2\alpha_1 = 0.5 < \mu$. However, for much smaller values of
  $\alpha_1$, the good behavior breaks down, as seen in panels (c) and
  (d). Note that in the latter two panels, the composite gradient
  descent algorithm does not appear to be converging, even as the
  iteration number increases. Comparing (c) and (d) also illustrates
  the interplay between the curvature parameter $\alpha_1$ of
  $\Loss_n$ and the nonconvexity parameter $\mu$ of
  $\rho_\lambda$. Indeed, the plot in panel (d) is slightly ``better''
  than the plot in panel (c), in the sense that initial iterates at
  least demonstrate some pattern of convergence. This could be
  attributed to the fact that the SCAD parameter is larger,
  corresponding to a smaller value of $\mu$.

\begin{figure}[h!]
  \begin{center}
    \begin{tabular}{cc}
      \widgraph{.45\textwidth}{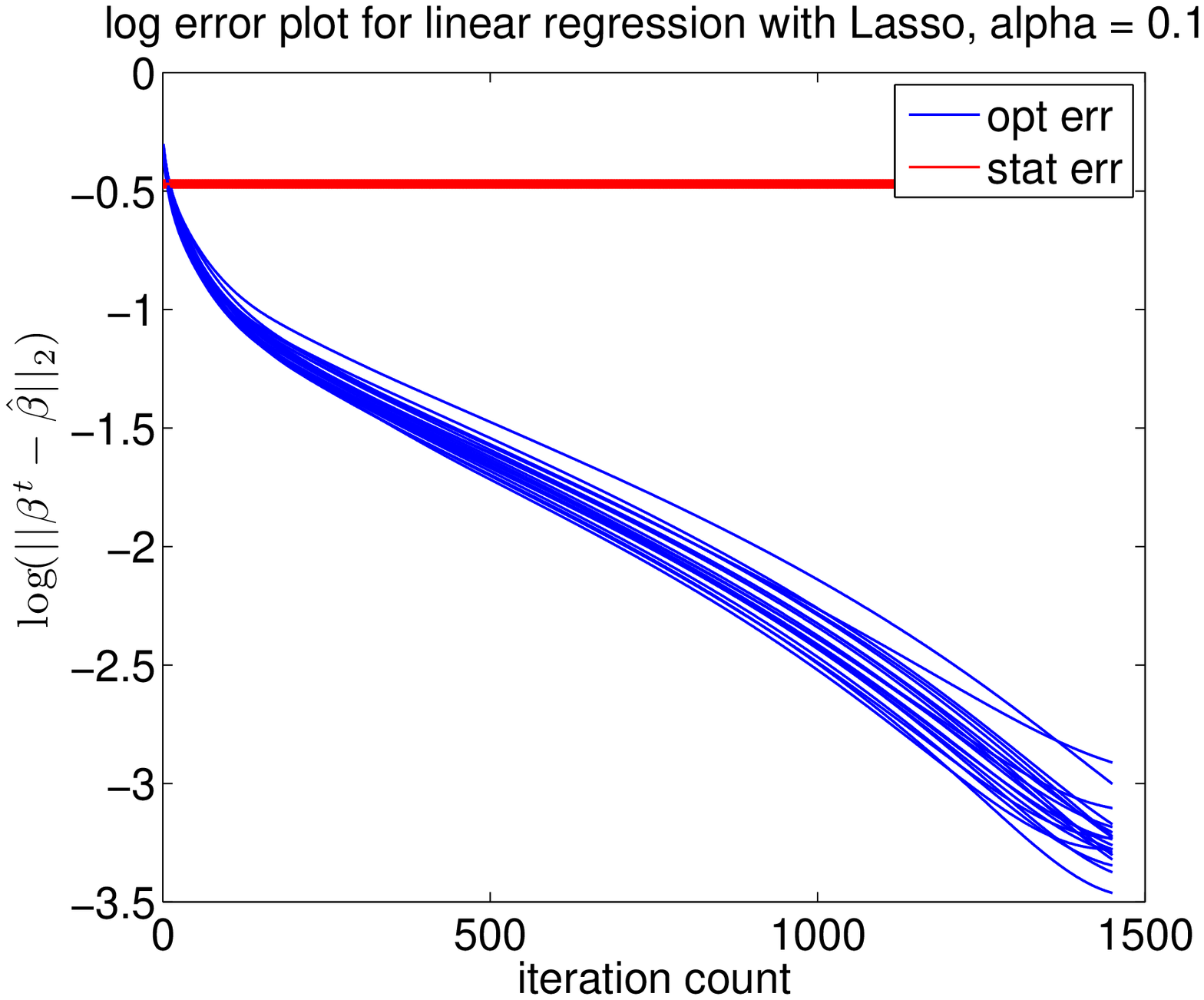} &
      \widgraph{.45\textwidth}{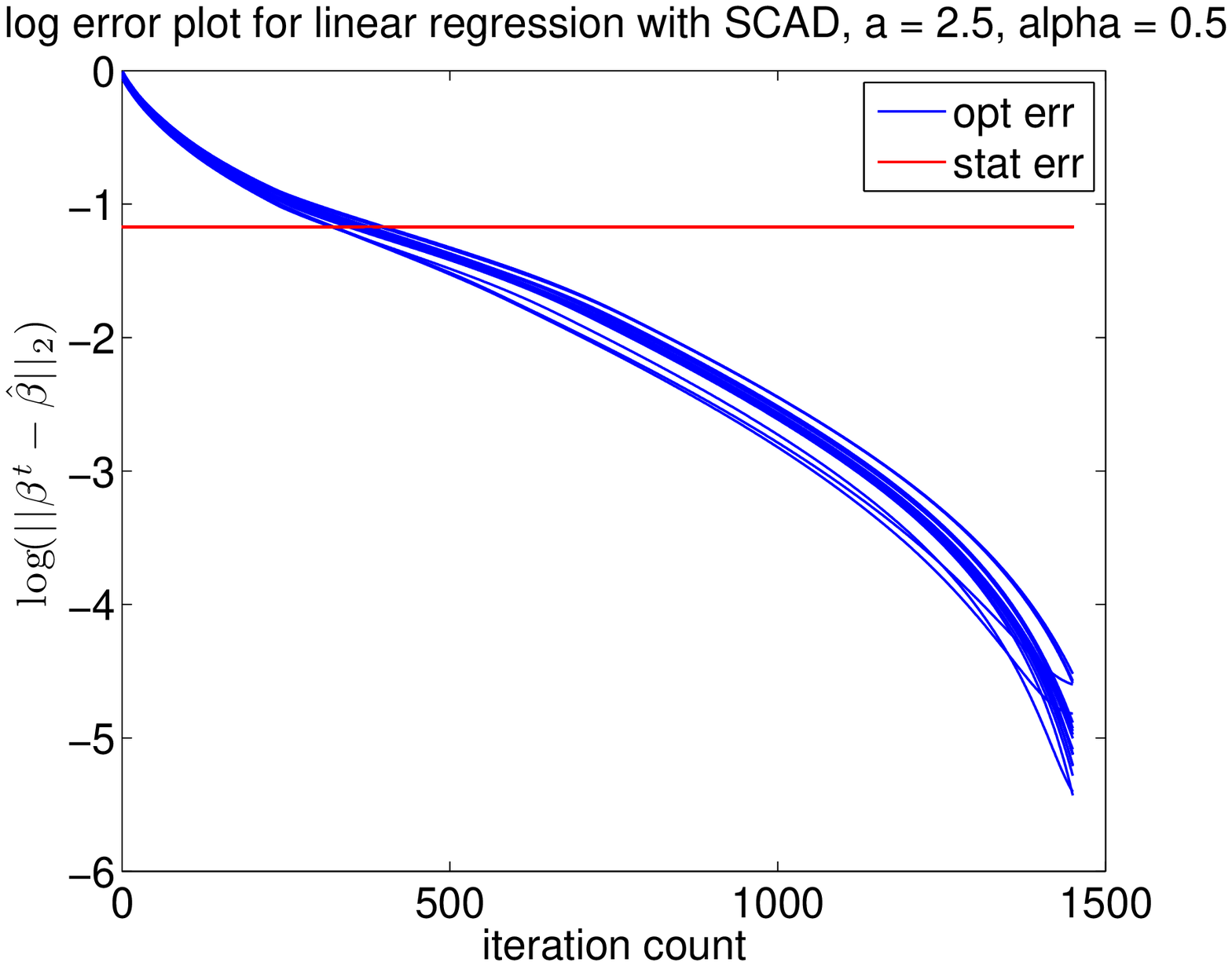} \\
      (a) & (b) \\ \widgraph{.45\textwidth}{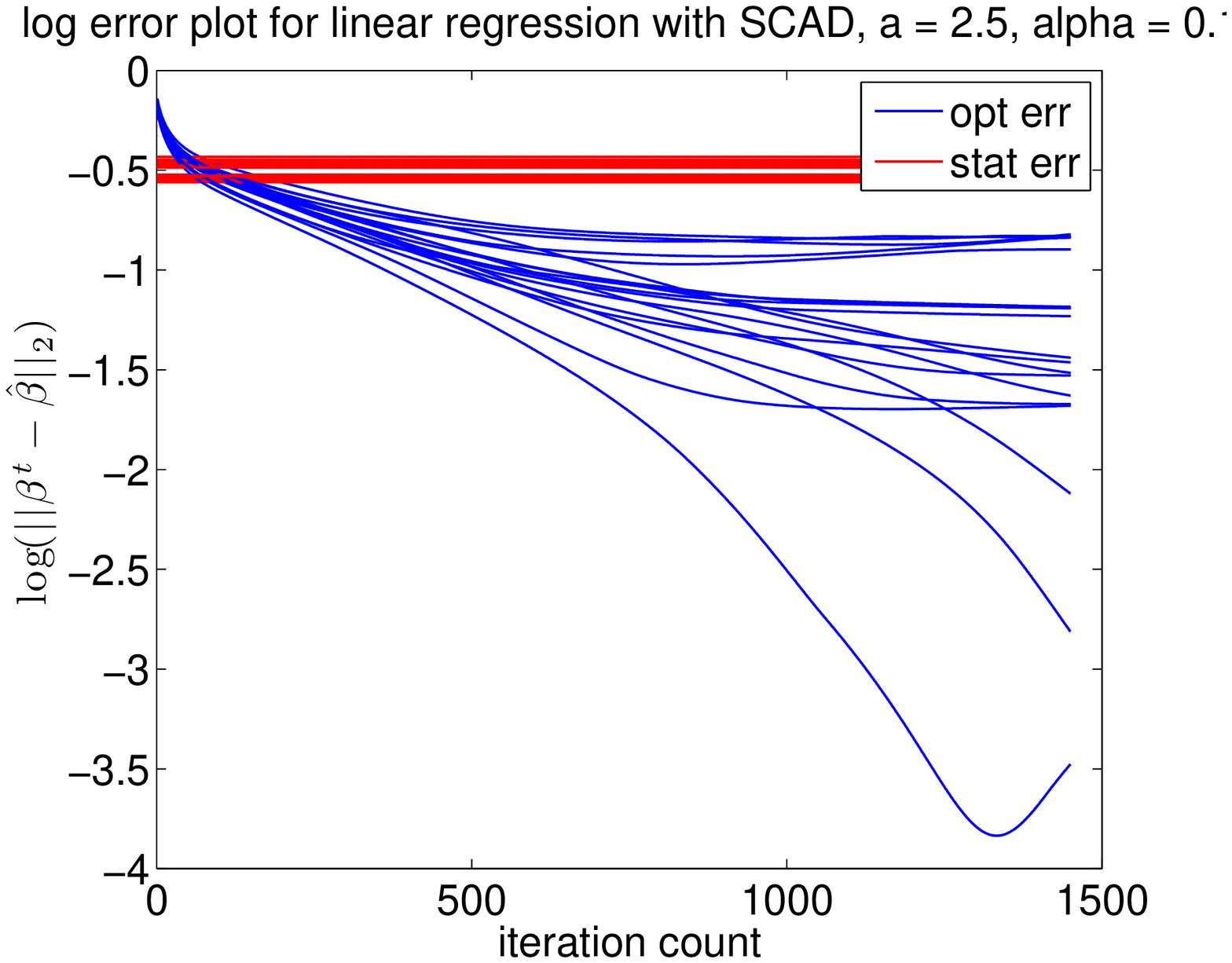} &
      \widgraph{.45\textwidth}{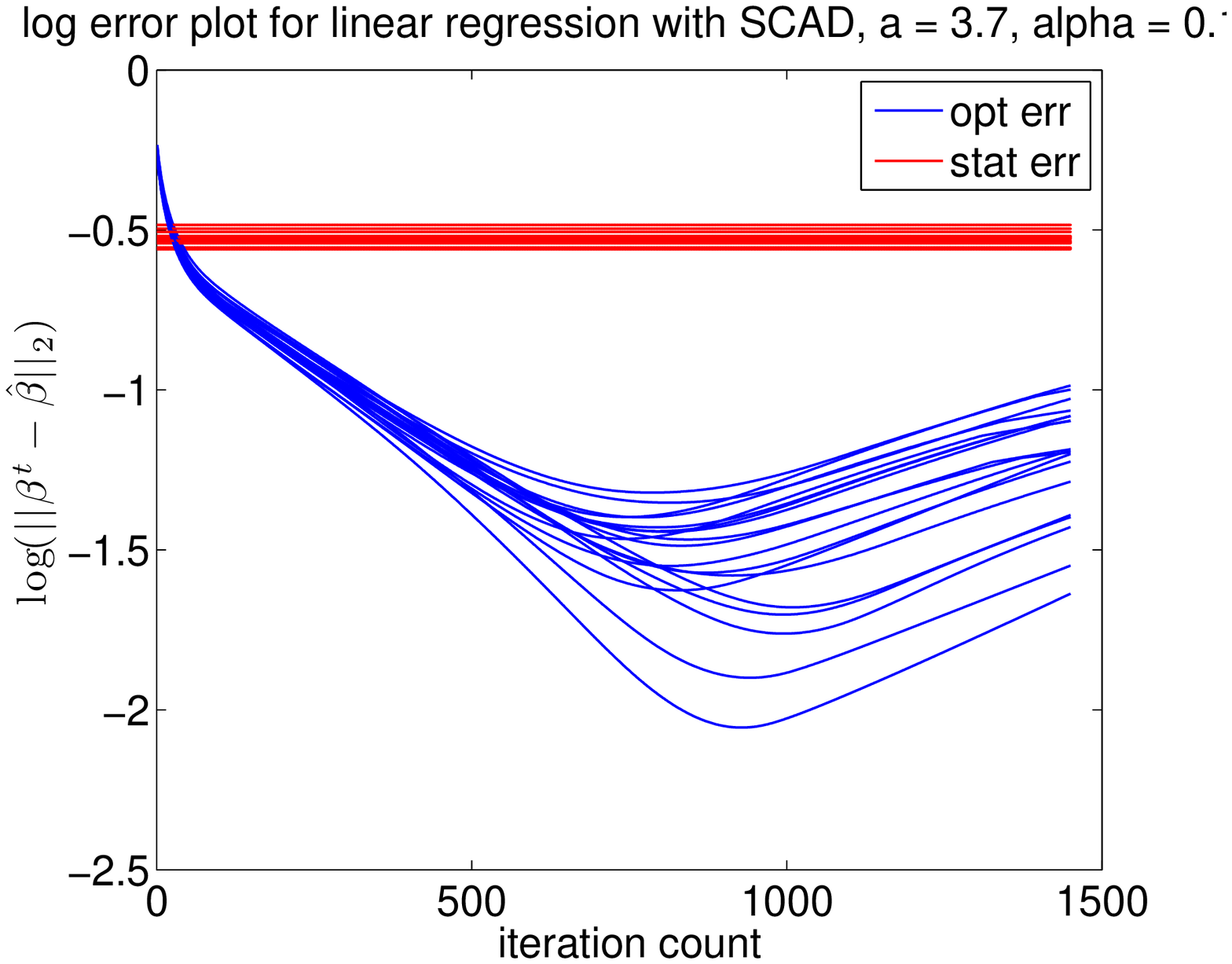} \\ (c) & (d)
    \end{tabular}
  \end{center}
  \caption{Plots showing breakdown points as a function of the
    curvature parameter $\alpha_1$ of the loss function and the
    nonconvexity parameter $\mu$ of the penalty function. The loss
    comes from ordinary least squares linear regression, where
    covariates are fully-observed and sampled from a Gaussian
    distribution with covariance equal to a Toeplitz matrix. Panel (a)
    depicts the good behavior of Lasso-based linear regression. Panel
    (b) shows that local optima may still be well-behaved even when
    $2\alpha_1 < \mu$, although this situation is not covered by our
    theory. Panels (c) and (d) show that the good behavior nonetheless
    disintegrates for very small values of $\alpha_1$ when the
    regularizer is nonconvex.}
\label{FigBreakdown}
\end{figure}
			
\section{Discussion}
We have analyzed theoretical properties of local optima of regularized
$M$-estimators, where both the loss and penalty function are allowed
to be nonconvex. Our results are the first to establish that \emph{all
stationary points} of such nonconvex problems are close to the truth,
implying that any optimization method guaranteed to converge to a
stationary point will provide statistically consistent solutions. We show
concretely that a variant of composite gradient descent may be used to
obtain near-global optima in linear time, and verify our theoretical
results with simulations.

Future directions of research include further generalizing our
statistical consistency results to other nonconvex regularizers not
covered by our present theory, such as bridge penalties or
regularizers that do not decompose across coordinates. In addition, it
would be interesting to expand our theory to nonsmooth loss functions
such as the hinge loss. For both nonsmooth losses and nonsmooth
penalties (including capped-$\ell_1$), it remains an open question
whether a modified version of composite gradient descent may be used
to obtain near-global optima in polynomial time. Finally, it would be
useful to develop a general method for establishing RSC and RSM
conditions, beyond the specialized methods used for studying GLMs in
this paper.

\acks{The work of PL was supported from a Hertz Foundation Fellowship
  and an NSF Graduate Research Fellowship. In addition, PL
  acknowledges support from a PD Award, and MJW from a DY Award. MJW
  and PL were also partially supported by NSF grant
  CIF-31712-23800. The authors thank the associate editors and
  anonymous reviewers for their helpful suggestions that improved the
  manuscript.}

%%%%%%%%%%%%%%%%%%

\appendix

\section{Properties of Regularizers}
\label{AppRegularizers}

In this section, we establish properties of some nonconvex
regularizers covered by our theory (Appendix~\ref{AppGeneralProps}) and
verify that specific regularizers satisfy Assumption~\ref{AsRho}
(Appendix~\ref{AppVerify}).  The properties given in
Appendix~\ref{AppGeneralProps} are used in the proof of
Theorem~\ref{TheoEll12Meta}.

\subsection{General Properties}
\label{AppGeneralProps}

We begin with some general properties of regularizers that satisfy
Assumption~\ref{AsRho}.

\begin{mylemma}
\label{LemRhoCond} \mbox{} \\
\begin{enumerate}
\item[(a)] Under conditions (i)--(ii) of Assumption~\ref{AsRho}, conditions (iii)
and (iv) together imply that $\myrho$ is $\lambda L$-Lipschitz as a
function of $t$. In particular, all subgradients and derivatives of
$\myrho$ are bounded in magnitude by $\lambda L$.
\item[(b)] Under the conditions of Assumption~\ref{AsRho}, we have
\begin{equation}
\label{EqnXmas}
\lambda L\|\beta\|_1 \le \rho_\lambda(\beta) + \frac{\mu}{2} \|\beta\|_2^2, \qquad \forall \beta \in \real^p.
\end{equation}
\end{enumerate}
\end{mylemma}
	
\begin{proof}
\textbf{(a):} Suppose $0 \le t_1 \le t_2$. Then
 \begin{equation*}
  \frac{\myrho(t_2) - \myrho(t_1)}{t_2 - t_1} \le
  \frac{\myrho(t_1)}{t_1},
\end{equation*}
by condition (iii). Applying (iii) once more, we have
\begin{equation*}
\frac{\myrho(t_1)}{t_1} \le \lim_{t \rightarrow 0^+}
\frac{\myrho(t)}{t} = \lambda L,
\end{equation*}
where the last equality comes from condition (iv). Hence,
\begin{equation*}
  0 \le \myrho(t_2) - \myrho(t_1) \le \lambda L(t_2 - t_1).
\end{equation*}
A similar argument applies to the cases when one (or both) of $t_1$
and $t_2$ are negative.

\textbf{(b):} Clearly, it suffices to verify the inequality for the scalar case:
\begin{equation*}
\lambda L t \le \rho_\lambda(t) + \frac{\mu t^2}{2}, \qquad \forall t \in \real.
\end{equation*}
The inequality is trivial for $t = 0$. For $t > 0$, the convexity of the right-hand expression implies that for any $s \in (0, t)$, we have
\begin{equation*}
\left(\rho_\lambda(t) + \frac{\mu t^2}{2}\right) - \left(\rho_\lambda(0) + \frac{\mu \cdot 0^2}{2}\right) \ge (t-0) \cdot \left(\rho_\lambda'(s) + \mu s\right).
\end{equation*}
Taking a limit as $s \rightarrow 0^+$ then yields the desired inequality. The case $t < 0$ follows by symmetry.
\end{proof}

\begin{mylemma}
\label{LemEll1Reg}
Suppose $\rho_\lambda$ satisfies the conditions of Assumption~\ref{AsRho}. Let $v \in \real^\pdim$, and let $A$ denote the index set of
the $k$ largest elements of $v$ in magnitude. Suppose $\xi > 0$ is such that $\xi \rho_\lambda(v_A) - \rho_\lambda(v_{A^c}) \ge 0$. Then
\begin{align}
\label{EqnRhoDif}
\xi \myrho(v_A) - \myrho(v_{A^c}) \le \lambda L (\xi \|v_A\|_1 -
\|v_{A^c}\|_1).
\end{align}
Moreover, if $\betastar \in \real^p$ is $k$-sparse, then for an vector $\beta \in \real^p$ such that $\xi \rho_\lambda(\betastar) - \rho_\lambda(\beta) > 0$ and $\xi \ge 1$, we have
\begin{align}
\label{EqnKeyConsequence}
\xi \myrho(\betastar) - \myrho(\beta) & \leq \lambda L \big(\xi \|\nu_A\|_1 -
\|\nu_{A^c}\|_1\big),
\end{align}
where $\nu \defn \beta - \betastar$ and $A$ is the index set of the $k$ largest elements of $\nu$ in magnitude.
\end{mylemma}

%%%%%%%%%%%%%%%%%%%%%%%%%%%%%%%%%%%%%%%%%%%%%%%%%%%%%%%%%%%%%%%%%%%%%%%%%%%%

\begin{proof}
We first establish~\eqref{EqnRhoDif}. Define $f(t) \defn
\frac{t}{\myrho(t)}$ for $t > 0$.  By our assumptions on $\myrho$, the
function $f$ is nondecreasing in $|t|$, so
\begin{equation}
\label{EqnF1}
\|v_{A^c}\|_1 = \sum_{j \in A^c} \myrho(v_j) \cdot f(|v_j|) \le
\sum_{j \in A^c} \myrho(v_j) \cdot f(\|v_{A^c}\|_\infty) =
\myrho(v_{A^c}) \cdot f\left(\|v_{A^c}\|_\infty\right).
\end{equation}
Again using the nondecreasing property of $f$, we have
\begin{align}
\label{EqnF2}
\myrho(v_A) \cdot f(\|v_{A^c}\|_\infty) = \sum_{j \in A} \myrho(v_j)
\cdot f(\|v_{A^c}\|_\infty) \le \sum_{j \in A} \myrho(v_j) \cdot
f(|v_j|) = \|v_A\|_1.
\end{align}
Note that for $t > 0$, we have
\begin{equation*}
f(t) \ge \lim_{s \rightarrow 0^+} f(s) = \lim_{s \rightarrow 0^+}
\frac{s - 0}{\myrho(s)- \myrho(0)} = \frac{1}{\lambda L},
\end{equation*}
where the last equality follows from condition (iv) of Assumption~\ref{AsRho}. Combining this result with~\eqref{EqnF1} and~\eqref{EqnF2} yields
\begin{equation*}
  0 \le \xi \myrho(v_A) - \myrho(v_{A^c}) \le \frac{1}{f(\|v_{A^c}\|_\infty)}
  \cdot \big( \xi \|v_A\|_1 - \|v_{A^c}\|_1\big) \le \lambda L
  \big(\xi \|v_A\|_1 - \|v_{A^c}\|_1\big),
\end{equation*}
as claimed.

We now turn to the proof of the bound~\eqref{EqnKeyConsequence}.
Letting $S \defn \supp(\betastar)$ denote the support of $\betastar$,
the triangle inequality and subadditivity of $\rho$ (see the remark
following Assumption~\ref{AsRho}; cf.\ Lemma 1 of~\citealp{CheGu13})
imply that
\begin{align*}
0 \le \xi \myrho(\betastar) - \myrho(\beta) & = \xi \myrho(\betastar_S) -
\myrho(\beta_S) - \myrho(\beta_{S^c}) \\
& \le \xi\myrho(\nu_S) - \myrho(\beta_{S^c}) \\
& = \xi\myrho(\nu_S) - \myrho(\nu_{S^c}) \\
& \le \xi\myrho(\nu_A) - \myrho(\nu_{A^c}) \\
& \le \lambda L \big(\xi\|\nu_A\|_1 - \|\nu_{A^c}\|_1\big),
\end{align*}
thereby completing the proof.
\end{proof}

%%%%%%%%%%%%%%%%%%%%%%%%%%%%%%%%%%%%%%%%%%%%%%%%%%%%%%%%%%%%%%%%%%%%%%%%

\subsection{Verification for Specific Regularizers}
\label{AppVerify}

We now verify that Assumption~\ref{AsRho} is satisfied by the SCAD and
MCP regularizers. (The properties are trivial to verify for the Lasso
penalty.)

\begin{mylemma}
\label{LemSCAD}
The SCAD regularizer~\eqref{EqnSCADdefn} with parameter $a$ satisfies
the conditions of Assumption~\ref{AsRho} with $L = 1$ and $\mu =
\frac{1}{a-1}$.
\end{mylemma}
\begin{proof}
Conditions (i)--(iii) were already verified
in~\cite{ZhaZha12}. Furthermore, we may easily compute the derivative
of the SCAD regularizer to be
\begin{equation}
  \label{EqnSCADDeriv}
  \frac{\partial}{\partial t} \myrho(t) = \sign(t) \cdot \left(\lambda
  \cdot \Ind\{|t| \le \lambda\} + \frac{(a\lambda - |t|)_+}{a-1} \cdot
  \Ind\{|t| > \lambda\}\right), \qquad t \neq 0,
\end{equation}
and any point in the interval $[-\lambda , \lambda]$ is a valid
subgradient at $t = 0$, so condition (iv) is satisfied for any $L \ge
1$.  Furthermore, we have $\frac{\partial^2}{\partial t^2} \myrho(t)
\ge \frac{-1}{a-1}$, so $\doubrho$ is convex whenever $\mu \ge
\frac{1}{a-1}$, giving condition (v).
\end{proof}

\begin{mylemma}
 \label{LemMCP}
 The MCP regularizer~\eqref{EqnMCPdefn} with parameter $b$ satisfies
 the conditions of Assumption~\ref{AsRho} with1 $L = 1$ and $\mu =
 \frac{1}{b}$.
\end{mylemma}

\begin{proof}
Again, the conditions (i)--(iii) are already verified
in~\cite{ZhaZha12}. We may compute the derivative of the MCP
regularizer to be
\begin{equation}
\label{EqnMCPDeriv}
\frac{\partial}{\partial t} \myrho(t) = \lambda \cdot \sign(t) \cdot
\left(1 - \frac{|t|}{\lambda b}\right)_+, \qquad t \neq 0,
\end{equation}
with subgradient $\lambda [-1, +1]$ at $t = 0$, so condition (iv) is
again satisfied for any $L \ge 1$.  Taking another derivative, we have
$\frac{\partial^2}{\partial t^2} \myrho(t) \ge \frac{-1}{b}$, so
condition (v) of Assumption~\ref{AsRho} holds with $\mu =
\frac{1}{b}$.
\end{proof}

%%%%%%%%%%%%%%%%%%%%%%%%%%%%%%%%%%%%%%%%%%%%%%%%%%%%%%%%%%%%%%%%%%%%%%%%%%

\section{Proofs of Corollaries in Section 3}
\label{AppCorollary}

In this section, we provide proofs of the corollaries to
Theorem~\ref{TheoEll12Meta} stated in Section~\ref{SecMain}.
Throughout this section, we use the convenient shorthand notation
\begin{align}
\label{EqnE}
\scriptE(\Delta) & \defn \inprod{\nabla \EmpLoss(\betastar + \Delta) -
  \nabla \EmpLoss(\betastar)}{\Delta}.
\end{align}

\subsection{General Results for Verifying RSC}
\label{AppGenRSC}

We begin with two lemmas that will be useful for establishing the RSC
conditions~\eqref{EqnRSC} in the special case where $\Loss_n$ is
convex. We assume throughout that $\|\Delta\|_1 \le 2R$, since
$\betastar$ and $\betastar + \Delta$ lie in the feasible set.

\begin{mylemma}
  \label{LemLocalRSC}
  Suppose $\Loss_n$ is convex. If condition~\eqref{EqnLocalRSC} holds
  and $n \ge 4R^2 \tau_1^2 \log p$, then
  \begin{equation}
    \label{EqnL2BdSpecific}
    \scriptE(\Delta) \ge \alpha_1 \|\Delta\|_2 - \sqrt{\frac{\log
        p}{n}} \|\Delta\|_1, \qquad \mbox{for all $\|\Delta\|_2 \ge
      1$.}
  \end{equation}
\end{mylemma}
\begin{proof}
  Fix an arbitrary $\Delta \in \real^\pdim$ with $\|\Delta\|_2 \ge
  1$. Since $\Loss_n$ is convex, the function $f:[0,1] \rightarrow
  \real$ given by $f(t) \defn \EmpLoss(\betastar + t \Delta)$ is also
  convex, so $f'(1) - f'(0) \geq f'(t) - f'(0)$ for all $t \in [0,1]$.
  Computing the derivatives of $f$ yields the inequality
  \begin{equation*}
    \scriptE(\Delta) \; = \; \inprod{\nabla \EmpLoss(\betastar +
      \Delta) - \nabla \EmpLoss(\betastar)}{\Delta} \ge \frac{1}{t} \,
    \inprod{\nabla \EmpLoss(\betastar + t\Delta) - \nabla
      \EmpLoss(\betastar)}{t\Delta}.
  \end{equation*}
  Taking $t = \frac{1}{\|\Delta\|_2} \in (0,1]$ and applying
condition~\eqref{EqnLocalRSC} to the rescaled vector
$\frac{\Delta}{\|\Delta\|_2}$ then yields
\begin{align*}
\scriptE(\Delta) & \geq \|\Delta\|_2 \left(\alpha_1 - \tau_1
\frac{\log p}{n} \frac{\|\Delta\|_1^2}{\|\Delta\|_2^2} \right) \\
& \ge \|\Delta\|_2 \left(\alpha_1 - \frac{2R \tau_1 \log p}{n}
\frac{\|\Delta\|_1}{\|\Delta\|_2^2}\right) \\
& \geq \|\Delta\|_2 \left ( \alpha_1 - \sqrt{\frac{\log p}{n}}
\frac{\|\Delta\|_1}{\|\Delta\|_2} \right) \\
& = \alpha_1 \|\Delta\|_2 - \sqrt{\frac{\log p}{n}} \|\Delta\|_1,
\end{align*}
where the third inequality uses the assumption on the relative scaling
of $(n,p)$ and the fact that $\|\Delta\|_2 \ge 1$.
\end{proof}

On the other hand, if~\eqref{EqnLocalRSC} holds globally
over $\Delta \in \real^p$, we obtain~\eqref{EqnL2Bd} for
free:

\begin{mylemma}
\label{LemGlobalRSC}
If inequality~\eqref{EqnLocalRSC} holds for all $\Delta \in \real^p$
and $\numobs \ge 4R^2 \tau_1^2 \log p$, then~\eqref{EqnL2Bd} holds, as well.
\end{mylemma}

\begin{proof}
Suppose $\|\Delta\|_2 \ge 1$. Then
\begin{equation*}
\alpha_1 \|\Delta\|_2^2 - \tau_1 \frac{\log p}{n} \|\Delta\|_1^2 \ge
\alpha_1 \|\Delta\|_2 - 2R \tau_1 \frac{\log p}{n} \|\Delta\|_1 \ge
\alpha_1 \|\Delta\|_2 - \sqrt{\frac{\log p}{n}} \|\Delta\|_1,
\end{equation*}
again using the assumption on the scaling of $(n,p)$.
\end{proof}

\subsection{Proof of Corollary~\ref{CorLinear}}
\label{AppCorLinear}

Note that $\scriptE(\Delta) = \Delta^T \GamHat \Delta$, so in
particular,
\begin{align*}
\scriptE(\Delta) & \ge \Delta^T \Sigma_x \Delta - |\Delta^T(\Sigma_x -
\GamHat) \Delta|.
\end{align*}
Applying Lemma 12 in~\cite{LohWai11a} with $s =
\frac{n}{\log p}$ to bound the second term, we have
\begin{align*}
\scriptE(\Delta) & \ge \lambda_{\min}(\Sigma_x) \|\Delta\|_2^2 -
\left(\frac{\lambda_{\min}(\Sigma_x)}{2} \|\Delta\|_2^2 + \frac{c \log
  p}{n} \|\Delta\|_1^2\right) \\
& = \frac{\lambda_{\min}(\Sigma_x)}{2} \|\Delta\|_2^2 - \frac{c \log
  p}{n} \|\Delta\|_1^2,
\end{align*}
a bound which holds for all $\Delta \in \real^\pdim$ with probability
at least $1 - c_1 \exp(-c_2 \numobs)$ whenever $\numobs \succsim \kdim
\log \pdim$. Then Lemma~\ref{LemGlobalRSC} in Appendix~\ref{AppGenRSC}
implies that the RSC condition~\eqref{EqnL2Bd} holds.  It remains
to verify the validity of the specified choice of $\lambda$.  We have
\begin{align*}
\|\nabla \EmpLoss(\betastar)\|_\infty = \|\GamHat \betastar -
\gamhat\|_\infty & = \|(\gamhat - \Sigma_x \betastar) + (\Sigma_x -
\GamHat)\betastar\|_\infty \\
& \leq \|(\gamhat - \Sigma_x \betastar)\|_\infty + \|(\Sigma_x -
\GamHat)\betastar\|_\infty.
\end{align*}
As shown in previous work~\citep{LohWai11a}, both of these terms are
upper-bounded by $c' \, \varphi \sqrt{\frac{\log \pdim}{\numobs}}$ with high
probability.  Consequently, the claim in the corollary follows by
applying Theorem~\ref{TheoEll12Meta}.

%%%%%%%%%%%%%%%%%%%%%%%%%%%%%%%%%%%%%%%%%%%%%%%%%%%%%%%%%%%%%%%%%%%%%%%%

\subsection{Proof of Corollary~\ref{CorGLM}}
\label{AppCorGLM}

In the case of GLMs, we have
\begin{align*}
\scriptE(\Delta) & = \frac{1}{\numobs} \sum_{i=1}^\numobs
(\psi'(\inprod{x_i} {\betastar + \Delta}) -
\psi'(\inprod{x_i}{\betastar})) \, x_i^T\Delta.
\end{align*}
Applying the mean value theorem, we find that
\begin{align*}
\scriptE(\Delta) & = \frac{1}{\numobs} \sum_{i=1}^\numobs \psi''(
\inprod{x_i}{\betastar} + t_i \, \inprod{x_i}{\Delta}) \,
\big(\inprod{x_i}{ \Delta} \big)^2,
\end{align*}
where $t_i \in [0,1]$. From (the proof of) Proposition 2 in~\cite{NegRavWaiYu12}, we then have
\begin{equation}
  \label{EqnLocalRSC2}
  \scriptE(\Delta) \ge \alpha_1 \|\Delta\|_2^2 - \tau_1
  \sqrt{\frac{\log p}{n}} \|\Delta\|_1 \|\Delta\|_2, \qquad \forall
  \|\Delta\|_2 \le 1,
\end{equation}
with probability at least $1 - c_1 \exp(-c_2 n)$, for an appropriate choice of $\alpha_1$. Note that by the arithmetic
mean-geometric mean inequality,
\begin{equation*}
  \tau_1 \sqrt{\frac{\log p}{n}} \|\Delta\|_1 \|\Delta\|_2 \le
  \frac{\alpha_1}{2} \|\Delta\|_2^2 + \frac{\tau_1^2}{2\alpha_1}
\frac{\log p}{n} \|\Delta\|_1^2,
\end{equation*}
and consequently,
\begin{equation*}
  \scriptE(\Delta) \ge \frac{\alpha_1}{2} \|\Delta\|_2^2 -
  \frac{\tau_1^2}{2\alpha_1} \frac{\log p}{n} \|\Delta\|_1^2,
\end{equation*}
which establishes~\eqref{EqnLocalRSC}. Inequality~\eqref{EqnL2Bd} then
follows via Lemma~\ref{LemLocalRSC} in Appendix~\ref{AppGenRSC}.

It remains to show that there are universal constants $(c, c_1, c_2)$
such that
\begin{align}
  \label{EqnLambdaBeta}
\mprob \left(\|\nabla \EmpLoss(\betastar)\|_\infty \geq c
\sqrt{\frac{\log \pdim}{\numobs}}\right) & \leq c_1 \exp(-c_2 \log
\pdim).
\end{align}
For each $1 \le i \le n$ and $1 \le j \le p$,
define the random variable $V_{ij} \defn (\psi'(x_i^T \betastar) -
y_i) x_{ij}$.  Our goal is to bound $\max_{j=1,
  \ldots, \pdim } |\frac{1}{\numobs} \sum_{i=1}^\numobs V_{ij}|$.
Note that
\begin{align}
  \label{EqnBikeRide}
  \mprob \left[ \max_{j=1, \ldots, \pdim} \big|\frac{1} {\numobs}
    \sum_{i=1}^\numobs V_{ij} \big| \geq \delta \right] & \leq
  \mprob[\Aevent^c] + \mprob \left[ \max_{j=1, \ldots, \pdim} \big|
    \frac{1}{\numobs} \sum_{i=1}^\numobs V_{ij} \big |\ge \delta \; \mid
    \; \Aevent \right],
\end{align}
where
\begin{equation*}
	\Aevent \defn \left\{\max \limits_{j=1, \ldots, \pdim}
  \left\{ \frac{1}{\numobs} \sum_{i=1}^\numobs x_{ij}^2 \right\} \le
  2\E[x_{ij}^2] \right\}.
\end{equation*}
Since the $x_{ij}$'s are sub-Gaussian and
\mbox{$\numobs \succsim \log \pdim$,} there exist universal constants
$(c_1, c_2)$ such that $\mprob[\Aevent^c] \leq c_1 \exp(-c_2
\numobs)$.  The last step is to bound the second term on the right
side of~\eqref{EqnBikeRide}.  For any $t \in \real$, we
have
\begin{align*}
\log \E[\exp(t V_{ij}) \mid x_i] & = \log\left[ \exp(t x_{ij}
  \psi'(x_i^T \betastar)\right] \cdot \E[\exp(-t x_{ij} y_i)] \notag
\\
 & = t x_{ij} \psi'(x_i^T \betastar) + \left(\psi(-t x_{ij} + x_i^T
\betastar) - \psi(x_i^T \betastar) \right),
\end{align*}
using the fact that $\psi$ is the cumulant generating function for the
underlying exponential family. Thus, by a Taylor series expansion,
there is some $v_i \in [0,1]$ such that
\begin{align}
  \log \E[\exp(t V_{ij}) \mid x_i] & = \frac{t^2 x_{ij}^2}{2} \,
  \psi''(x_i^T \betastar - v_i \, t x_{ij}) \; \leq \; \frac{\alpha_u
    t^2 x_{ij}^2}{2},
\end{align}
where the inequality uses the boundedness of $\psi''$.  Consequently,
conditioned on the event $\Aevent$, the variable $\frac{1}{\numobs}
\sum_{i=1}^\numobs V_{ij}$ is sub-Gaussian with parameter at most
$\kappa = \alpha_u \cdot \max_{j = 1, \ldots, \pdim} \Exs[x_{ij}^2]$,
for each $j = 1, \ldots, \pdim$. By a union bound, we then have
\begin{align*}
\mprob \left[ \max_{j=1, \ldots, \pdim} \big| \frac{1}{\numobs}
  \sum_{i=1}^\numobs V_{ij} \big |\ge \delta \; \mid \; \Aevent
  \right] & \leq \pdim \, \exp\left(- \frac{\numobs \delta^2}{2
  \kappa^2} \right).
\end{align*}
The claimed $\ell_1$- and $\ell_2$-bounds then follow directly from
Theorem~\ref{TheoEll12Meta}.

%%%%%%%%%%%%%%%%%%%%%%%%%%%%%%%%%%%%%%%%%%%%%%%%%%%%%%%%%%%%%%%%%%%%%%%%%%

\subsection{Proof of Corollary~\ref{CorGLasso}}
\label{AppCorGLasso}

We first verify condition~\eqref{EqnLocalRSC} in the case where
$\frobnorm{\Delta} \leq 1$.  A straightforward calculation yields
\begin{equation*}
\nabla^2 \Loss_n(\Theta) = \Theta^{-1} \otimes \Theta^{-1} =
\left(\Theta \otimes \Theta\right)^{-1}.
\end{equation*}
Moreover, letting $\myvec(\Delta) \in \real^{\pdim^2}$ denote the
vectorized form of the matrix $\Delta$, applying the mean value
theorem yields
\begin{align}
\label{EqnMVT}
\scriptE(\Delta) = \myvec(\Delta)^T \left(\nabla^2 \EmpLoss(\Thetastar
+ t \Delta) \right) \myvec(\Delta) & \geq \lambda_{\min}(\nabla^2
\EmpLoss(\Thetastar + t \Delta)) \; \frobnorm{\Theta}^2,
\end{align}
for some $t \in [0,1]$.  By standard properties of the Kronecker
product~\citep{HorJoh90}, we have
\begin{align*}
\lambda_{\min}(\nabla^2 \EmpLoss(\Thetastar + t\Delta)) & =
\opnorm{\Thetastar + t\Delta}_2^{-2} \ge \left(\opnorm{\Thetastar}_2
+ t\opnorm{\Delta}_2\right)^{-2} \\
& \geq \left(\opnorm{\Thetastar}_2 + 1\right)^{-2},
\end{align*}
using the fact that $\opnorm{\Delta}_2 \le \opnorm{\Delta}_F \le
1$. Plugging back into~\eqref{EqnMVT} yields
\begin{equation*}
\scriptE(\Delta) \ge \left(\opnorm{\Thetastar}_2 + 1 \right)^{-2} \,
\frobnorm{\Theta}^2,
\end{equation*}
so~\eqref{EqnLocalRSC} holds with
$\alpha_1 = \left(\opnorm{\Thetastar}_2 + 1\right)^{-2}$ and $\tau_1 =
0$.  Lemma~\ref{LemGlobalRSC} then implies~\eqref{EqnL2Bd}
with $\alpha_2 = \left(\opnorm{\Thetastar}_2 + 1\right)^{-2}$.
Finally, we need to establish that the given choice of $\lambda$
satisfies the requirement~\eqref{EqnLambdaChoice} of
Theorem~\ref{TheoEll12Meta}.  By the assumed deviation
condition~\eqref{EqnSigConc}, we have
\begin{equation*}
\opnorm{\nabla \Loss_n(\Thetastar)}_{\max} = \opnorm{\SigHat -
  (\Theta^{*})^{-1}}_{\max} = \opnorm{\SigHat - \Sigma}_{\max} \;
\leq \; c_0 \sqrt{\frac{\log p}{n}}.
\end{equation*}
Applying Theorem~\ref{TheoEll12Meta} then implies the desired result.

%%%%%%%%%%%%%%%%%%%%%%%%%%%%%%%%%%%%%%%%%%%%%%%%%%%%%%%%%%%%%%%%%%%%%%%%%%%%%%

\section{Auxiliary Optimization-Theoretic Results}
\label{AppFastGlobal}

In this section, we provide proofs of the supporting lemmas used in
Section~\ref{SecOpt}.

\subsection{Derivation of Three-Step Procedure}
\label{AppCompIterate}

We begin by deriving the correctness of the three-step procedure given
in Section~\ref{SecUpdates}. Let $\betahat$ be the unconstrained
optimum of the program~\eqref{EqnCompUnconstr}.  If $g_{\lambda,
  \mu}(\betahat) \le R$, we clearly have the update given in step
(2). Suppose instead that $g_{\lambda, \mu}(\betahat) > R$. Then since
the program~\eqref{EqnCompGradUpdate} is convex, the iterate
$\beta^{t+1}$ must lie on the boundary of the feasible set; i.e.,
\begin{equation}
\label{EqnGBound}
g_{\lambda, \mu}(\beta^{t+1}) = R.
\end{equation}

By Lagrangian duality, the program~\eqref{EqnCompGradUpdate} is also
equivalent to
\begin{equation*}
\beta^{t+1} \in \arg\min_{g_{\lambda, \mu}(\beta) \le R'}
\left\{\frac{1}{2} \left\|\beta - \left(\beta^t - \frac{\nabla
  \Lossbar_n(\beta^t)}{\eta}\right)\right\|_2^2\right\},
\end{equation*}
for some choice of constraint parameter $R'$.  Note that this is
projection of $\beta^t - \frac{\nabla \Lossbar_n(\beta^t)}{\eta}$ onto
the set $\{ \beta \in \real^\pdim \, \mid \, \SIDESPEC(\beta) \leq R'
\}$. Since projection decreases the value of $g_{\lambda, \mu}$,
equation~\eqref{EqnGBound} implies that
\begin{equation*}
g_{\lambda, \mu}\left(\beta^t - \frac{\nabla
  \Lossbar_n(\beta^t)}{\eta}\right) \ge R.
\end{equation*}
In fact, since the projection will shrink the vector to the boundary
of the constraint set, \eqref{EqnGBound} forces $R' =
R$. This yields the update~\eqref{EqnRhobarProj} appearing in step
(3).

\subsection{Derivation of Updates for SCAD and MCP}
\label{AppUpdates}

We now derive the explicit form of the updates~\eqref{EqnXSCAD}
and~\eqref{EqnXMCP} for the SCAD and MCP regularizers, respectively. We
may rewrite the unconstrained program~\eqref{EqnCompUnconstr} as
\begin{align}
\label{EqnCompUnconstr2}
\beta^{t+1} & \in \arg \min_{\beta \in \real^\pdim} \left\{\frac{1}{2}
\left\|\beta - \left(\beta^t - \frac{\nabla
  \Lossbar_n(\beta^t)}{\eta}\right)\right\|_2^2 + \frac{1}{\eta} \cdot
\rho_\lambda(\beta) + \frac{\mu}{2\eta} \|\beta\|_2^2\right\} \notag \\
& = \arg \min_{\beta \in \real^\pdim} \left\{\left(\frac{1}{2} +
\frac{\mu}{2\eta} \right) \|\beta\|_2^2 - \beta^T \left(\beta^t -
\frac{\nabla \Lossbar_n(\beta^t)}{\eta}\right) + \frac{1}{\eta} \cdot
\rho_\lambda(\beta) \right\} \notag \\
& = \arg \min_{\beta \in \real^\pdim} \left\{\frac{1}{2} \left\|\beta
- \frac{1}{1 + \mu/\eta} \left(\beta^t - \frac{\nabla
  \Lossbar_n(\beta^t)}{\eta}\right)\right\|_2^2 + \frac{1/\eta}{1 +
  \mu/\eta} \cdot \rho_\lambda(\beta)\right\}.
\end{align}
Since the program in the last line of
equation~\eqref{EqnCompUnconstr2} decomposes by coordinate, it
suffices to solve the scalar optimization problem
\begin{equation}
\label{EqnScalarPenalty}
\xhat \in \arg\min_x \left\{\frac{1}{2} (x-z)^2 + \nu \rho(x;
\lambda)\right\},
\end{equation}
for general $z \in \real$ and $\nu > 0$. \\

We first consider the case when $\rho$ is the SCAD penalty. The
solution $\xhat$ of the program~\eqref{EqnScalarPenalty} in the case
when $\nu = 1$ is given in~\cite{FanLi01}; the
expression~\eqref{EqnXSCAD} for the more general case comes from
writing out the subgradient of the objective as
\begin{equation*}
  (x-z) + \nu \rho'(x; \lambda) = \begin{cases} (x-z) + \nu \lambda
    [-1, +1] & \mbox{if } x = 0, \\
(x-z) + \nu \lambda & \mbox{if } 0 < x \le \lambda, \\
(x-z) + \frac{\nu(a\lambda - x)}{a-1} & \mbox{if } \lambda \le x \le
    a \lambda, \\ 
x-z & \mbox{if } x \ge a\lambda,
  \end{cases}
\end{equation*}
using the equation for the SCAD derivative~\eqref{EqnSCADDeriv}, and
setting the subgradient equal to zero.

Similarly, when $\rho$ is the MCP parametrized by $(b, \lambda)$, the
subgradient of the objective takes the form
\begin{equation*}
  (x-z) + \nu \rho'(x; \lambda) = \begin{cases} (x-z) + \nu \lambda
    [-1, +1] & \mbox{if } x = 0, \\
(x-z) + \nu \lambda \left(1 - \frac{x}{b \lambda}\right) & \mbox{if }
    0 < x \le b \lambda, \\
x-z & \mbox{if } x \ge b\lambda,
\end{cases}
\end{equation*}
using the expression for the MCP derivative~\eqref{EqnMCPDeriv},
leading to the closed-form solution given in~\eqref{EqnXMCP}. This agrees with the expression provided in~\cite{BreHua11} for the special case when $\nu = 1$.

\subsection{Proof of Lemma~\ref{LemICB}}

We first show that if $\lambda \ge \frac{8}{L} \cdot \|\nabla
\Loss_n(\betastar)\|_\infty$, then for any feasible $\beta$ such that
\begin{equation}
  \label{EqnTol}
  \phi(\beta) \le \phi(\betastar) + \etabar,
\end{equation}
we have
\begin{equation}
  \label{EqnConeInter}
  \|\beta - \betastar\|_1 \le 8 \sqrt{k}\|\beta - \betastar\|_2 + 2\cdot \min\left(\frac{2\etabar}{\lambda L}, R\right).
\end{equation}

Defining the error vector $\Delta \defn \beta - \betastar$, \eqref{EqnTol} implies
\begin{equation*}
  \Loss_n(\betastar + \Delta) + \rho_\lambda(\betastar + \Delta) \le
\Loss_n(\betastar) + \rho_\lambda(\betastar) + \etabar,
\end{equation*}
so subtracting $\inprod{\nabla \Loss_n(\betastar)}{\Delta}$ from both sides gives
\begin{equation}
  \label{EqnOde}
  \scriptT(\betastar + \Delta, \betastar) + \rho_\lambda(\betastar +
  \Delta) - \rho_\lambda(\betastar) \le -\inprod{\nabla
    \Loss_n(\betastar)}{\Delta} + \etabar.
\end{equation}
We divide the argument into two cases. First suppose $\|\Delta\|_2 \le 3$. Note that if $\etabar \ge \frac{\lambda L}{4} \|\Delta\|_1$, the claim~\eqref{EqnConeInter} is trivially true; so assume $\etabar \le \frac{\lambda L}{4} \|\Delta\|_1$. Then the RSC condition~\eqref{EqnLocalRSCFirst}, together with~\eqref{EqnOde}, implies that
\begin{align}
  \label{EqnConeInter2}
  \alpha_1 \|\Delta\|_2^2 - \tau_1 \frac{\log p}{n} \|\Delta\|_1^2 +
  \rho_\lambda(\betastar + \Delta) - \rho_\lambda(\betastar) & \le \|\nabla
  \Loss_n(\betastar)\|_\infty \cdot \|\Delta\|_1 + \etabar \notag \\
& \le  \frac{\lambda L}{8} \|\Delta\|_1 + \frac{\lambda L}{4} \|\Delta\|_1.
\end{align}
Rearranging and using the assumption $\lambda L \ge 16 R\tau_1 \frac{\log p}{n}$, along with Lemma~\ref{LemRhoCond} in Appendix~\ref{AppGeneralProps}, we then have
\begin{align*}
\alpha_1 \|\Delta\|_2^2 & \le \rho_\lambda(\betastar) - \rho_\lambda(\betastar + \Delta) + \frac{\lambda L}{2} \|\Delta\|_1 \\
& \le \rho_\lambda(\betastar) - \rho_\lambda(\betastar + \Delta) + \frac{\rho_\lambda(\betastar) + \rho_\lambda(\betastar + \Delta)}{2} + \frac{\mu}{4} \|\Delta\|_2^2,
\end{align*}
implying that
\begin{equation*}
0 \le \left(\alpha_1 - \frac{\mu}{4}\right) \|\Delta\|_2^2 \le \frac{3}{2} \rho_\lambda(\betastar) - \frac{1}{2} \rho_\lambda(\betastar + \Delta),
\end{equation*}
so
\begin{equation}
\label{EqnSleigh}
\rho_\lambda(\betastar) - \rho_\lambda(\betastar + \Delta) \le 3 \rho_\lambda(\betastar) - \rho_\lambda(\betastar + \Delta) \le 3\lambda L\|\Delta_A\|_1 - \lambda L\|\Delta_{A^c}\|_1,
\end{equation}
by Lemma~\ref{LemEll1Reg} in Appendix~\ref{AppGeneralProps}. Furthermore, note that  the bound~\eqref{EqnConeInter2} also implies that
\begin{equation}
  \label{EqnEtabarBd}
  \rho_\lambda(\betastar + \Delta) - \rho_\lambda(\betastar) \le \frac{\lambda L}{2} \|\Delta\|_1 + \etabar.
\end{equation}
Combining~\eqref{EqnSleigh} and~\eqref{EqnEtabarBd} then gives
\begin{equation*}
\|\Delta_{A^c}\|_1  - 3 \|\Delta_A\|_1 \le \frac{1}{2} \|\Delta\|_1 + \frac{\etabar}{\lambda L} \le \frac{1}{2} \|\Delta_A\|_1 + \frac{1}{2} \|\Delta_{A^c}\|_1 + \frac{\etabar}{\lambda L},
\end{equation*}
so
\begin{equation*}
\|\Delta_{A^c}\|_1 \le 7 \|\Delta_A\|_1 + \frac{2\etabar}{\lambda L},
\end{equation*}
implying that
\begin{equation*}
\|\Delta\|_1 \le 8 \|\Delta_A\|_1 + \frac{2 \etabar}{\lambda L} \le 8 \sqrt{k} \|\Delta\|_2 + \frac{2\etabar}{\lambda L}.
\end{equation*}
In the case when $\|\Delta\|_2 \ge 3$, the RSC condition~\eqref{EqnLocalRSCSecond}
gives
\begin{align}
\label{EqnBells}
\alpha_2 \|\Delta\|_2 - \tau_2 \sqrt{\frac{\log p}{n}} \|\Delta\|_1
+ \rho_\lambda(\betastar + \Delta) - \rho_\lambda(\betastar) & \le \|\nabla \Loss_n(\betastar)\|_\infty \cdot \|\Delta\|_1 + \etabar \notag \\
& \le \frac{\lambda L}{8} \|\Delta\|_1 + \frac{\lambda L}{4} \|\Delta\|_1,
\end{align}
so
\begin{equation*}
\alpha_2 \|\Delta\|_2 \le \rho_\lambda(\betastar) - \rho_\lambda(\betastar + \Delta) + \left(\frac{3\lambda L}{8} + \tau_2 \sqrt{\frac{\log p}{n}}\right) \|\Delta\|_1.
\end{equation*}
In particular, if $\rho_\lambda(\betastar) - \rho_\lambda(\betastar + \Delta) \le 0$, we have
\begin{equation*}
\|\Delta\|_2 \le \frac{2R}{\alpha_2} \left(\frac{3 \lambda L}{8} + \tau_2 \sqrt{\frac{\log p}{n}}\right) < 3,
\end{equation*}
a contradiction. Hence, using Lemma~\ref{LemEll1Reg} in Appendix~\ref{AppGeneralProps}, we have
\begin{equation}
\label{EqnCarol}
0 \le \rho_\lambda(\betastar) - \rho_\lambda(\betastar + \Delta) \le \lambda L \|\Delta_A\|_1 - \lambda L \|\Delta_{A^c}\|_1.
\end{equation}
Note that under the scaling $\lambda L \ge 4\tau_2 \sqrt{\frac{\log p}{n}}$, the bound~\eqref{EqnBells} also implies~\eqref{EqnEtabarBd}. Combining~\eqref{EqnEtabarBd} and~\eqref{EqnCarol}, we then have
\begin{equation*}
  \|\Delta_{A^c}\|_1 - \|\Delta_A\|_1 \le \frac{1}{2} \|\Delta\|_1 +
  \frac{\etabar}{\lambda L} = \frac{1}{2} \|\Delta_{A^c}\|_1 +
  \frac{1}{2} \|\Delta_A\|_1 + \frac{\etabar}{\lambda L},
\end{equation*}
and consequently,
\begin{equation*}
\|\Delta_{A^c}\|_1 \le 3 \|\Delta_A\|_1 + \frac{2\etabar}{\lambda L},
\end{equation*}
so
\begin{equation*}
\|\Delta\|_1 \le 4 \|\Delta_A\|_1 + \frac{2\etabar}{\lambda L} \le
4\sqrt{k} \|\Delta\|_2 + \frac{2\etabar}{\lambda L}.
\end{equation*}
Using the trivial bound $\|\Delta\|_1 \le 2R$, we obtain the claim~\eqref{EqnConeInter}. \\

We now apply the implication~\eqref{EqnTol} to the vectors $\betahat$
and $\beta^t$. Note that by optimality of $\betahat$, we have
\begin{equation*}
	\phi(\betahat) \le \phi(\betastar),
\end{equation*}
and by the assumption~\eqref{EqnObjTol}, we also have
\begin{equation*}
\phi(\beta^t) \le \phi(\betahat) + \etabar \le \phi(\betastar) +
\etabar.
\end{equation*}
Hence,
\begin{align*}
\|\betahat - \betastar\|_1 & \le 8 \sqrt{k} \|\betahat - \betastar\|_2, \qquad
\text{and} \\
\|\beta^t - \betastar\|_1 & \le 8 \sqrt{k}
\|\beta^t - \betastar\|_2 + 2 \cdot \min\left(\frac{2\etabar}{\lambda
  L}, R\right).
\end{align*}
By the triangle inequality, we then have
\begin{align*}
 \|\beta^t - \betahat\|_1 & \le \|\betahat - \betastar\|_1 + \|\beta^t
 - \betastar\|_1 \\
& \le 8\sqrt{k} \cdot \Big(\|\betahat - \betastar\|_2 + \|\beta^t -
 \betastar\|_2\Big) + 2 \cdot \min\left(\frac{2\etabar}{\lambda L},
 R\right) \\
& \le 8\sqrt{k} \cdot \Big(2\|\betahat - \betastar\|_2 + \|\beta^t -
 \betahat\|_2 \Big) + 2 \cdot \min\left(\frac{2\etabar}{\lambda L},
 R\right),
\end{align*}
as claimed.

\subsection{Proof of Lemma~\ref{LemL2Assump}}
\label{AppL2Assump}

Our proof proceeds via induction on the iteration number $t$.  Note
that the base case $t = 0$ holds by assumption. Hence, it remains to
show that if $\|\betait{\iter} - \betahat\|_2 \leq 3$ for some integer
$t \geq 1$, then $\|\betait{\iter+1} - \betahat\|_2 \leq 3$, as well.
	
We assume for the sake of a contradiction that $\|\beta^{t+1} -
\betahat\|_2 > 3$. By the RSC condition~\eqref{EqnLocalRSCSecond} and
the relation~\eqref{EqnTCompare}, we have
\begin{equation}
  \label{EqnIterateZeroPrime}
  \scriptTBar(\beta^{t+1}, \betahat) \ge \alpha \|\betahat -
  \beta^{t+1}\|_2 - \tau \sqrt{\frac{\log p}{n}} \|\betahat -
  \beta^{t+1}\|_1 - \frac{\mu}{2} \|\betahat - \beta^{t+1}\|_2^2.
	\end{equation}
Furthermore, by convexity of $g \defn g_{\lambda, \mu}$, we have
\begin{equation}
  \label{EqnRhobarBd}
  g(\beta^{t+1}) - g(\betahat) - \inprod{\nabla g(\betahat)}{\beta^{t+1} - \betahat} \ge 0.
\end{equation}
Multiplying by $\lambda$ and summing with~\eqref{EqnIterateZeroPrime} then yields
\begin{align*}
& \phi(\beta^{t+1}) - \phi(\betahat) - \inprod{\nabla
    \phi(\betahat)}{\beta^{t+1} - \betahat} \\
& \qquad \ge \alpha
  \|\betahat - \beta^{t+1}\|_2 - \tau \sqrt{\frac{\log p}{n}}
  \|\betahat - \beta^{t+1}\|_1 - \frac{\mu}{2} \|\betahat -
  \beta^{t+1}\|_2^2.
\end{align*}
Together with the first-order optimality condition $\inprod{\nabla
  \phi(\betahat)}{\beta^{t+1} - \betahat} \ge 0$, we then have
\begin{align}
  \label{EqnL2Upper}
  \phi(\beta^{t+1}) - \phi(\betahat) & \ge \alpha \|\betahat - \beta^{t+1}\|_2 - \tau \sqrt{\frac{\log p}{n}} \|\betahat - \beta^{t+1}\|_1 - \frac{\mu}{2} \|\betahat - \beta^{t+1}\|_2^2.
\end{align}
	
Since $\|\betahat - \beta^t\|_2 \le 3$ by the induction
hypothesis, applying the RSC
condition~\eqref{EqnLocalRSCFirst} to the pair $(\betahat,
\beta^t)$ also gives
\begin{equation*}
  \Lossbar_n(\betahat) \ge \Lossbar_n(\beta^t) + \inprod{\nabla \Lossbar_n(\beta^t)}{\betahat - \beta^t} + \left(\alpha - \frac{\mu}{2}\right) \cdot \|\beta^t - \betahat\|_2^2 - \tau \frac{\log p}{n} \|\beta^t - \betahat\|_1^2.
\end{equation*}
Combining with the inequality
\begin{equation*}
  g(\betahat) \ge g(\beta^{t+1}) + \inprod{\nabla g(\beta^{t+1})}{\betahat - \beta^{t+1}},
\end{equation*}
we then have
\begin{align}
\label{EqnPhiBd}
\phi(\betahat) & \ge \Lossbar_n(\beta^t) + \inprod{\nabla
  \Lossbar_n(\beta^t)}{\betahat - \beta^t} + \lambda g(\beta^{t+1}) +
\lambda \inprod{\nabla g(\beta^{t+1})}{\betahat - \beta^{t+1}} \notag \\
& \qquad \qquad + \left(\alpha - \frac{\mu}{2}\right) \cdot \|\beta^t - \betahat\|_2^2 -
\tau \frac{\log p}{n} \|\beta^t - \betahat\|_1^2 \notag \\
& \ge \Lossbar_n(\beta^t) + \inprod{\nabla \Lossbar_n(\beta^t)}{\betahat - \beta^t} + \lambda g(\beta^{t+1}) \notag \\
& \qquad \qquad + \lambda \inprod{\nabla g(\beta^{t+1})}{\betahat - \beta^{t+1}} - \tau \frac{\log p}{n} \|\beta^t - \betahat\|_1^2.
\end{align}
Finally, the RSM condition~\eqref{EqnRSM} on the pair $(\beta^{t+1},
\beta^t)$ gives
\begin{align}
\label{EqnRSMExpand}
\phi(\beta^{t+1}) & \le \Lossbar_n(\beta^t) + \inprod{\nabla
  \Lossbar_n(\beta^t)}{\beta^{t+1} - \beta^t} + \lambda g(\beta^{t+1})
\\
& \qquad \qquad \qquad \qquad + \left(\alpha_3 - \frac{\mu}{2}\right) \|\beta^{t+1} -
\beta^t\|_2^2 + \tau \frac{\log p}{n} \|\beta^{t+1} - \beta^t\|_1^2
\notag \\
& \le \Lossbar_n(\beta^t) + \inprod{\nabla \Lossbar_n(\beta^t)}{\beta^{t+1} - \beta^t} + \lambda g(\beta^{t+1}) \notag \\
& \qquad \qquad \qquad \qquad + \frac{\eta}{2} \|\beta^{t+1} - \beta^t\|_2^2 + \frac{4R^2 \tau \log p}{n},
\end{align}
since $\frac{\eta}{2} \ge \alpha_3 - \frac{\mu}{2}$ by assumption, and
$\|\beta^{t+1} - \beta^t\|_1 \le 2R$. It is easy to check that the
update~\eqref{EqnCompGradUpdate} may be written equivalently as
\begin{equation*}
  \beta^{t+1} \in \arg \min_{g(\beta) \le R, \; \beta \in \Omega}
  \left\{\Lossbar_n(\beta^t) + \inprod{\nabla
    \Lossbar_n(\beta^t)}{\beta - \beta^t} + \frac{\eta}{2} \|\beta -
  \beta^t\|_2^2 + \lambda g(\beta)\right\},
\end{equation*}
and the optimality of $\beta^{t+1}$ then yields
\begin{equation}
  \label{EqnBetaT}
  \inprod{\nabla \Lossbar_n(\beta^t) + \eta (\beta^{t+1} - \beta^t) +
    \lambda \nabla g(\beta^{t+1})}{\beta^{t+1} - \betahat} \le 0.
\end{equation}
Summing up~\eqref{EqnPhiBd}, \eqref{EqnRSMExpand},
and~\eqref{EqnBetaT}, we then have
\begin{align*}
\phi(\beta^{t+1}) \!\! - \!\! \phi(\betahat) & \le \frac{\eta}{2} \|\beta^{t+1}
- \beta^t\|_2^2 + \eta \inprod{\beta^t - \beta^{t+1}}{\beta^{t+1} -
  \betahat} + \tau \frac{\log p}{n} \|\beta^t - \betahat\|_1^2 \notag \\
& \qquad \qquad \qquad \qquad \qquad \qquad \qquad \qquad \qquad +
\frac{4 R^2 \tau \log p}{n} \\
& = \frac{\eta}{2} \|\beta^t - \betahat\|_2^2 - \frac{\eta}{2}
\|\beta^{t+1} - \betahat\|_2^2 + \tau \frac{\log p}{n} \|\beta^t -
\betahat\|_1^2 + \!\! \frac{4 R^2 \tau \log p}{n}.
\end{align*}
Combining this last inequality with~\eqref{EqnL2Upper}, we have
\begin{align*}
\alpha \|\betahat - \beta^{t+1}\|_2 - & \tau \sqrt{\frac{\log p}{n}}
\|\betahat - \beta^{t+1}\|_1 \\
& \le \frac{\eta}{2} \|\beta^t -
\betahat\|_2^2 - \frac{\eta-\mu}{2} \|\beta^{t+1} -
\betahat\|_2^2 + \frac{8 R^2 \tau \log p}{n} \\
		& \le \frac{9\eta}{2} - \frac{3(\eta-\mu)}{2} \|\beta^{t+1} - \betahat\|_2 + \frac{8R^2 \tau \log p}{n},
\end{align*}
since $\|\beta^t - \betahat\|_2 \le 3$ by the induction hypothesis and
$\|\beta^{t+1} - \betahat\|_2 > 3$ by assumption, and using the fact
that $\eta \ge \mu$. It follows that
\begin{align*}
\left(\alpha + \frac{3(\eta - \mu)}{2}\right) \cdot \|\betahat -
\beta^{t+1}\|_2 & \le \frac{9\eta}{2} + \tau \sqrt{\frac{\log p}{n}}
\|\betahat - \beta^{t+1}\|_1 + \frac{8 R^2 \tau \log p}{n} \\
& \le \frac{9\eta}{2} + 2R\tau \sqrt{\frac{\log p}{n}} + \frac{8 R^2
  \tau \log p}{n} \\
& \le 3\left(\alpha + \frac{3(\eta - \mu)}{2}\right),
\end{align*}
where the final inequality holds whenever $2R\tau \sqrt{\frac{\log
    p}{n}} + \frac{8 R^2 \tau \log p}{n} \le 3\left(\alpha -
\frac{3\mu}{2}\right)$. Rearranging gives $\|\beta^{t+1} - \betahat\|_2 \le 3$,
providing the desired contradiction.

%%%%%%%%%%%%%%%%%%%%%%

\subsection{Proof of Lemma~\ref{LemIterate}}
\label{AppLemIterate}

We begin with an auxiliary lemma:

\begin{mylemma}
\label{LemSunflowerSeed}
Under the conditions of Lemma~\ref{LemIterate}, we have
\begin{subequations}
\begin{align}
\label{EqnIterateTwo}
\scriptTBar(\betait{t}, \betahat) & \geq -2\tau \frac{\log
  p}{n}(\epsilon + \epsbar)^2, \quad \mbox{and} \\
\label{EqnIterateOne}
\phi(\beta^t) - \phi(\betahat) & \geq \frac{2\alpha - \mu}{4}
\|\betahat - \beta^t\|_2^2 - \frac{2 \tau \log p}{n} (\epsilon +
\epsbar)^2.
\end{align}
\end{subequations}
\end{mylemma}

\noindent We prove this result later, taking it as given for the
moment. \\

Define
\begin{align*}
\phi_t(\beta) & \defn \EmpLossBar(\beta^t) + \inprod{\nabla
  \EmpLossBar(\beta^t)}{\beta - \beta^t} + \frac{\eta}{2} \|\beta -
\beta^t\|_2^2 + \lambda \SIDE(\beta),
\end{align*}
the objective function minimized over the
constraint set $\{\SIDE(\beta) \leq R\}$ at iteration $t$. For any
$\gamma \in [0,1]$, the vector $\beta_\gamma \defn \gamma \betahat +
(1-\gamma) \beta^t$ belongs to the constraint set, as well.
Consequently, by the optimality of $\betait{\iter+1}$ and feasibility
of $\beta_\gamma$, we have
\begin{align*}
\phi_t(\beta^{t+1}) \! \leq \phi_t(\beta_\gamma) & =
\EmpLossBar(\beta^t) \! + \! \inprod{\nabla \EmpLossBar(\beta^t)}{\gamma
  \betahat - \gamma \beta^t} \! + \! \frac{\eta \gamma^2}{2} \|\betahat -
\beta^t\|_2^2 + \lambda \SIDE(\beta_\gamma).
\end{align*}
Appealing to~\eqref{EqnIterateTwo}, we then have
\begin{align}
 \label{EqnIterateThree}
\phi_t(\beta^{t+1}) & \le (1-\gamma) \EmpLossBar(\beta^t) + \gamma
\EmpLossBar(\betahat) + 2\gamma \tau \frac{\log p}{n}(\epsilon +
\epsbar)^2  \notag \\
& \qquad \qquad \qquad \qquad + \frac{\eta \gamma^2}{2} \|\betahat - \beta^t\|_2^2 +
\lambda g(\beta_\gamma) \notag \\
& \stackrel{(i)}{\le} \phi(\beta^t) - \gamma(\phi(\beta^t) -
\phi(\betahat)) + 2\gamma \tau \frac{\log p}{n}(\epsilon + \epsbar)^2
+ \frac{\eta \gamma^2}{2} \|\betahat - \beta^t\|_2^2 \notag \\
& \le \phi(\beta^t) - \gamma (\phi(\beta^t) - \phi(\betahat)) + 2\tau
\frac{\log p}{n}(\epsilon + \epsbar)^2 + \frac{\eta \gamma^2}{2}
\|\betahat - \beta^t\|_2^2,
\end{align}
where inequality (i) incorporates the fact that
\begin{equation*}
	g(\beta_\gamma) \le \gamma g(\betahat) + (1-\gamma) g(\beta^t),
\end{equation*}
by the convexity of $g$.

By the RSM condition~\eqref{EqnRSM}, we also have
\begin{equation*}
\scriptTBar(\betait{\iter+1}, \betait{\iter}) \leq \frac{\eta}{2}
\|\beta^{t+1} - \beta^t\|_2^2 + \tau \frac{\log p}{n} \|\beta^{t+1} -
\beta^t\|_1^2,
\end{equation*}
since $\alpha_3 - \mu \le \frac{\eta}{2}$ by assumption, and adding
$\lambda g(\beta^{t+1})$ to both sides gives
\begin{align*}
\phi(\beta^{t+1}) & \le \EmpLossBar(\beta^t) + \inprod{\nabla
  \EmpLossBar(\beta^t)}{\beta^{t+1} - \beta^t} + \frac{\eta}{2}
\|\beta^{t+1} - \beta^t\|_2^2 \\
& \qquad \qquad + \tau \frac{\log p}{n} \|\beta^{t+1} -
\beta^t\|_1^2 + \lambda \SIDE(\beta^{t+1}) \\
& = \phi_t(\beta^{t+1}) + \tau \frac{\log p}{n} \|\beta^{t+1}
- \beta^t\|_1^2.
\end{align*}
Combining with~\eqref{EqnIterateThree} then yields
\begin{align}
\label{EqnOatmeal}
\phi(\beta^{t+1}) & \le \phi(\beta^t) - \gamma (\phi(\beta^t) -
\phi(\betahat)) + \frac{\eta \gamma^2}{2} \|\betahat - \beta^t\|_2^2 \notag \\
& \qquad +
\tau \frac{\log p}{n} \|\beta^{t+1} - \beta^t\|_1^2 + 2\tau \frac{\log
  p}{n} (\epsilon + \epsbar)^2.
\end{align}

By the triangle inequality, we have
\begin{align*}
\|\betait{\iter+1} - \betait{\iter}\|_1^2 & \leq
\big(\|\Delta^{\iter+1}\|_1 + \|\Delta^{\iter}\|_1\big)^2 \; \leq \; 2
\|\Delta^{\iter+1}\|^2_1 + 2 \|\Delta^{\iter}\|^2_1,
\end{align*}
where we have defined $\Delta^t \defn \beta^t - \betahat$. Combined
with~\eqref{EqnOatmeal}, we therefore have
\begin{multline*}
\phi(\beta^{t+1}) \leq \phi(\beta^t) - \gamma (\phi(\beta^t) -
\phi(\betahat)) + \frac{\eta \gamma^2}{2} \|\Delta^t\|_2^2 \\ + 2 \tau
\frac{\log p}{n}(\|\Delta^{t+1}\|_1^2 + \|\Delta^t\|_1^2) + 2 \SQUIRREL,
\end{multline*}
where $\SQUIRREL \defn \tau \frac{\log \pdim}{\numobs} (\epsilon +
\epsbar)^2$. Then applying Lemma~\ref{LemICB} to bound the
$\ell_1$-norms, we have
\begin{align}
\phi(\beta^{t+1}) & \leq \phi(\beta^t) - \gamma (\phi(\beta^t) -
\phi(\betahat)) + \frac{\eta \gamma^2}{2} \|\Delta^t\|_2^2 \notag \\
& \qquad \qquad \qquad + ck \tau
\frac{\log p}{n}(\|\Delta^{t+1}\|_2^2 + \|\Delta^t\|_2^2) + c'
\SQUIRREL \nonumber \\
& = \phi(\beta^t) - \gamma(\phi(\beta^t) - \phi(\betahat)) +
\left(\frac{\eta \gamma^2}{2} + ck\tau \frac{\log p}{n}\right)
\|\Delta^t\|_2^2 \nonumber \\
\label{EqnChickaDee}
& \qquad \qquad \qquad
\qquad \qquad \qquad + ck\tau
\frac{\log p}{n} \|\Delta^{t+1}\|_2^2 + c'\SQUIRREL.
\end{align}

Now introduce the shorthand $\delta_t \defn \phi(\beta^t) -
\phi(\betahat)$ and $\CHICKA = k \tau \frac{\log \pdim}{\numobs}$. By
applying~\eqref{EqnIterateOne} and subtracting
$\phi(\betahat)$ from both sides of~\eqref{EqnChickaDee},
we have
\begin{multline*}
\delta_{t+1} \leq \big (1 - \gamma \big) \delta_t + \frac{\eta
  \gamma^2 + c \CHICKA}{\alpha - \mupar/2} \left(\delta_t +
2\SQUIRREL\right) \\ + \frac{c \CHICKA }{\alpha - \mupar/2}
\left(\delta_{t+1} + 2\SQUIRREL\right) + c' \SQUIRREL.
\end{multline*}
Choosing $\gamma = \frac{2\alpha - \mupar}{4\eta} \in (0,1)$ yields
\begin{multline*}
\left(1 - \frac{c \CHICKA}{\alpha - \mupar/2} \right) \delta_{t+1}
\leq \left(1 - \frac{2\alpha - \mupar}{8\eta} + \frac{c
  \CHICKA}{\alpha - \mupar/2} \right) \delta_t \\
+ 2 \, \left(\frac{2\alpha - \mupar} {8\eta} + \frac{2c \CHICKA
}{\alpha - \mupar/2} + c' \right) \SQUIRREL,
\end{multline*}
or $\delta_{t+1} \le \kappa \delta_t + \xi (\epsilon + \epsbar)^2$,
where $\kappa$ and $\xi$ were previously defined in~\eqref{EqnKappa} and~\eqref{EqnXi}, respectively.  Finally,
iterating the procedure yields
\begin{align}
  \delta_t \le \kappa^{t-T} \delta_T + \xi (\epsilon + \epsbar)^2 (1 +
  \kappa + \kappa^2 + \cdots + \kappa^{t-T-1}) \; \le \; \kappa^{t-T}
  \delta_T + \frac{\xi (\epsilon + \epsbar)^2}{1-\kappa},
\end{align}
as claimed. \\

The only remaining step is to prove the auxiliary lemma. \\

\textbf{Proof of Lemma~\ref{LemSunflowerSeed}:} By the RSC condition~\eqref{EqnLocalRSCFirst} and the
assumption~\eqref{EqnL2Assump}, we have
\begin{align}
\label{EqnIterateZero}
\scriptTBar(\beta^t, \betahat) & \geq \left(\alpha - \frac{\mupar}{2} \right) \,
\|\betahat - \beta^t\|_2^2 - \tau \frac{\log p}{n} \|\betahat -
\beta^t\|_1^2.
\end{align}
Furthermore, by convexity of $g$, we have
\begin{align}
\label{EqnRhobarBd2}
\lambda\Big(\SIDE(\beta^t) - \SIDE(\betahat) - \inprod{\nabla
  \SIDE(\betahat)}{\beta^t - \betahat}\Big) & \geq 0,
\end{align}
and the first-order optimality condition for $\betahat$ gives
\begin{equation}
\label{EqnPhiOpt2}
\inprod{\nabla \phi(\betahat)}{\beta^t - \betahat} \ge 0.
\end{equation}
Summing~\eqref{EqnIterateZero}, \eqref{EqnRhobarBd2},
and~\eqref{EqnPhiOpt2} then yields
\begin{align*}
\phi(\beta^t) - \phi(\betahat) & \geq \left(\alpha - \frac{\mupar}{2} \right) \,
\|\betahat - \beta^t\|_2^2 - \tau \frac{\log p}{n}\|\betahat -
\beta^t\|_1^2.
\end{align*}
Applying Lemma~\ref{LemICB} to bound the term $\|\betahat -
\beta^t\|_1^2$ and using the assumption $\frac{ck \tau \log p}{n}
\le \frac{2\alpha - \mu}{4}$ yields the bound~\eqref{EqnIterateOne}.
On the other hand, applying Lemma~\ref{LemICB} directly to~\eqref{EqnIterateZero} with $\beta^t$ and $\betahat$
switched gives
\begin{align*}
\scriptTBar(\betahat, \beta^t) & \ge \left(\alpha - \frac{\mu}{2}\right) \|\betahat -
\beta^t\|_2^2 - \tau \frac{\log p}{n}\left(c k \|\betahat -
\beta^t\|_2^2 + 2(\epsilon + \epsbar)^2\right) \notag \\
& \ge -2\tau \frac{\log p}{n}(\epsilon + \epsbar)^2.
\end{align*}
This establishes~\eqref{EqnIterateTwo}.

%%%%%%%%%%%%%%%%%%%%%%%%%%%%%%%%%%%%%%%%%%%%%%%%%%%%%%%%%%%%%%%%%%%%%%%%%%%%%%

\section{Verifying RSC/RSM Conditions}
\label{AppGLM}

In this Appendix, we provide a proof of Proposition~\ref{PropGLM}, which verifies the RSC~\eqref{EqnRSC2} and RSM~\eqref{EqnRSM} conditions for GLMs.

\subsection{Main Argument}
\label{SecVerify}

Using the notation for GLMs in Section~\ref{SecGLM}, we introduce
the shorthand \mbox{$\Delta \defn \beta_1 - \beta_2$} and observe
that, by the mean value theorem, we have
\begin{equation}
\label{EqnTexpand}
\scriptT(\beta_1, \beta_2) = \frac{1}{\numobs} \sum_{i=1}^\numobs
\psi''\big(\inprod{\beta_1}{x_i}) + t_i \inprod{\Delta}{x_i} \big) \,
(\inprod{\Delta}{x_i})^2,
\end{equation}
for some $t_i \in [0,1]$. The $t_i$'s are i.i.d. random variables,
with each $t_i$ depending only on the random vector $x_i$. \\

\textbf{Proof of bound~\eqref{EqnRSMglm}:} The proof of this upper bound is relatively straightforward given
earlier results~\citep{LohWai12}.  From the Taylor series
expansion~\eqref{EqnTexpand} and the boundedness assumption
$\|\psi''\|_\infty \leq \alpha_u$, we have
\begin{align*}
\scriptT(\beta_1, \beta_2) & \leq \alpha_u \cdot \frac{1}{\numobs}
\sum_{i=1}^\numobs \big(\inprod{\Delta}{x_i} \big)^2.
\end{align*}
By known results on restricted eigenvalues for ordinary linear
regression (cf.\ Lemma 13 in~\cite{LohWai11a}), we also have
\begin{equation*}
\frac{1}{\numobs} \sum_{i=1}^n (\inprod{\Delta}{x_i})^2 \le
\lambda_{\max}(\CovX) \left( \frac{3}{2} \|\Delta\|_2^2 + \frac{\log
  p}{n}\|\Delta\|_1^2\right),
\end{equation*}
with probability at least $1 - c_1 \exp(-c_2 n)$. Combining the two
inequalities yields the desired result. \\

\textbf{Proof of bounds~\eqref{EqnRSCglm}:}
The proof of the RSC bound is much more involved, and we provide only
high-level details here, deferring the bulk of the technical analysis
to later in the appendix.  We define
\begin{align*}
\alpha_\ell \defn \left(\inf_{|t| \le 2T} \psi''(t)\right) \,
\frac{\lambda_{\min}(\CovX)}{8},
\end{align*}
where $T$ is a suitably chosen constant depending only on
$\lambda_{\min}(\CovX)$ and the sub-Gaussian parameter $\sigma_x$. (In
particular, see~\eqref{EqnChooseTau} below, and take $T =
3\tau$.) The core of the proof is based on the following lemma,
proved in Section~\ref{AppTaylorGLM}:

\begin{mylemma}
\label{LemTaylorGLM}
With probability at least $1 - c_1 \exp(-c_2 n)$, we have
\begin{align*}
  \scriptT(\beta_1, \beta_2) & \geq \alpha_\ell \|\Delta\|_2^2 - c
  \sigma_x \|\Delta\|_1 \|\Delta\|_2 \sqrt{\frac{\log p}{n}},
\end{align*}
uniformly over all pairs $(\beta_1, \beta_2)$ such that $\beta_2 \in
\ball_2(3) \cap \ball_1(R)$, $\|\beta_1 - \beta_2\|_2 \le 3$, and
\begin{equation}
  \label{EqnDeltaBd}
  \frac{\|\Delta\|_1}{\|\Delta\|_2} \le \frac{\alpha_\ell}{c
    \sigma_x}\sqrt{\frac{n}{\log p}}.
\end{equation}
\end{mylemma}

Taking Lemma~\ref{LemTaylorGLM} as given, we now complete the proof of
the RSC condition~\eqref{EqnRSCglm}.  By the arithmetic mean-geometric
mean inequality, we have
\begin{align*}
  c \sigma_x \|\Delta\|_1 \|\Delta\|_2 \sqrt{\frac{\log p}{n}} & \le
  \frac{\alpha_\ell}{2} \|\Delta\|_2^2 + \frac{c^2
    \sigma_x^2}{2\alpha_\ell} \frac{\log p}{n} \|\Delta\|_1^2,
\end{align*}
so Lemma~\ref{LemTaylorGLM} implies that~\eqref{EqnGLMFirst} holds uniformly over all pairs
$(\beta_1, \beta_2)$ such that $\beta_2 \in \ball_2(3) \cap
\ball_1(R)$ and $\|\beta_1 - \beta_2\|_2 \le 3$, whenever the
bound~\eqref{EqnDeltaBd} holds.  On the other hand, if the
bound~\eqref{EqnDeltaBd} does not hold, then the lower bound in~\eqref{EqnGLMFirst} is negative. By convexity of
$\EmpLoss$, we have $\scriptT(\beta_1, \beta_2) \ge 0$, so~\eqref{EqnGLMFirst} holds trivially in that case.

We now show that~\eqref{EqnGLMSecond} holds: in particular,
consider a pair $(\beta_1, \beta_2)$ with $\beta_2 \in \ball_2(3)$ and
$\|\beta_1 - \beta_2\|_2 \ge 3$. For any $t \in [0,1]$, the convexity
of $\EmpLoss$ implies that
\begin{align*}
\EmpLoss(\beta_2 + t \Delta) & \leq t \EmpLoss(\beta_2 + \Delta) +
(1-t) \EmpLoss(\beta_2),
\end{align*}
where $\Delta \defn \beta_1 - \beta_2$.  Rearranging yields
\begin{equation*}
	\EmpLoss(\beta_2 + \Delta) - \EmpLoss(\beta_2) \ge
        \frac{\EmpLoss(\beta_2 + t\Delta) - \EmpLoss(\beta_2)}{t},
\end{equation*}
so
\begin{equation}
\label{EqnTaylorIneq}
\scriptT(\beta_2 + \Delta, \beta_2) \ge \frac{\scriptT(\beta_2 +
  t\Delta, \beta_2)}{t}.
\end{equation}
Now choose $t = \frac{3}{\|\Delta\|_2} \in [0,1]$ so that
$\|t\Delta\|_2 = 1$.  Introducing the shorthand $\alpha_1 \defn
\frac{\alpha_\ell}{2}$ and $\tau_1 \defn \frac{c^2
  \sigma_x^2}{2\alpha_\ell}$, we may apply~\eqref{EqnGLMFirst} to obtain
\begin{align}
\label{EqnPadawan}
\frac{\scriptT(\beta_2 + t\Delta, \beta_2)}{t} & \ge
\frac{\|\Delta\|_2}{3} \left(\alpha_1 \left(\frac{3
  \|\Delta\|_2}{\|\Delta\|_2}\right)^2 - \tau_1 \frac{\log p}{n}
\left(\frac{3 \|\Delta\|_1}{\|\Delta\|_2}\right)^2\right) \notag \\
& = 3 \alpha_1
\|\Delta\|_2 - 9\tau_1 \frac{\log p}{n}
\frac{\|\Delta\|_1^2}{\|\Delta\|_2}.
\end{align}
Note that~\eqref{EqnGLMSecond} holds trivially unless
$\frac{\|\Delta\|_1}{\|\Delta\|_2} \leq \frac{\alpha_\ell}{2 c
  \sigma_x} \sqrt{\frac{\numobs}{\log \pdim}}$, due to the convexity
of $\EmpLoss$. In that case, \eqref{EqnTaylorIneq}
and~\eqref{EqnPadawan} together imply
\begin{align*}
\scriptT(\beta_2 + \Delta, \beta_2) & \geq 3 \alpha_1 \|\Delta\|_2 -
\frac{9 \tau_1 \, \alpha_\ell}{2 c \sigma_x} \sqrt{\frac{\log p}{n}}
\|\Delta\|_1,
\end{align*}
which is exactly the bound~\eqref{EqnGLMSecond}.

\subsection{Proof of Lemma~\ref{LemTaylorGLM}}
\label{AppTaylorGLM}

\newcommand{\taubar}{\ensuremath{\tau'}}

For a truncation level $\taubar > 0$ to be chosen, define the functions
\begin{equation*}
  \varphi_{\taubar}(u) =
	\begin{cases}
	  u^2, & \text{if } |u| \le \frac{\taubar}{2}, \\
	  (\taubar - u)^2, & \text{if } \frac{\taubar}{2} \le |u| \le \taubar, \\
	  0, & \text{if } |u| \ge \taubar.
	\end{cases}
\end{equation*}
By construction, $\varphi_{\taubar}$ is $\taubar$-Lipschitz and
\begin{equation}
\label{EqnPhiProps}
	\varphi_{\taubar}(u) \le u^2 \cdot \Ind\{|u| \le \taubar\}, \quad
        \mbox{for all $u \in \real$.}
\end{equation}
In addition, we define the trapezoidal function
\begin{equation*}
\gamma_\taubar(u) =
\begin{cases}
    1, & \text{if } |u| \le \frac{\taubar}{2}, \\
    2 - \frac{2}{\taubar} |u|, & \text{if } \frac{\taubar}{2} \le |u| \le \taubar, \\
		0, & \text{if } |u| \ge \taubar,
\end{cases}
\end{equation*}
and note that $\gamma_\taubar$ is $\frac{2}{\taubar}$-Lipschitz and
$\gamma_\taubar(u) \le \Ind\{|u| \le \taubar\}$.

Taking $T \ge 3\taubar$ so that $T \ge \taubar\|\Delta\|_2$ (since $\|\Delta\|_2 \le 3$ by assumption), and defining
\begin{equation*}
L_\psi(T) \defn \inf_{|u| \le 2T} \psi''(u),
\end{equation*}
we have the following inequality:
\begin{align}
\label{EqnTaylorBd1}
\scriptT(\beta + \Delta, \beta) & = \frac{1}{n} \sum_{i=1}^n
\psi''(x_i^T \beta + t_i \cdot x_i^T \Delta) \cdot(x_i^T \Delta)^2
\notag \\
& \ge L_{\psi}(T) \cdot \sum_{i=1}^n (x_i^T \Delta)^2 \cdot \Ind\{|x_i^T \Delta| \le \taubar \|\Delta\|_2\} \cdot \Ind\{|x_i^T \beta| \le T\} \notag \\
& \ge L_\psi(T) \cdot \frac{1}{n} \sum_{i=1}^n \varphi_{\taubar \|\Delta\|_2} (x_i^T \Delta) \cdot \gamma_T(x_i^T \beta),
\end{align}
where the first equality is the expansion~\eqref{EqnTexpand} and the
second inequality uses the bound~\eqref{EqnPhiProps}.

Now define the subset of $\real^\pdim \times \real^\pdim$ via
\begin{equation*}
\Aregion_\delta \defn \left\{(\beta, \Delta): \beta \in \ball_2(3)
\cap \ball_1(R), \; \Delta \in \ball_2(3), \;
\frac{\|\Delta\|_1}{\|\Delta\|_2} \le \delta\right\},
\end{equation*}
as well as the random variable
\begin{align*}
Z(\delta) & \defn \sup_{(\beta, \Delta) \in \Aregion_\delta}
\frac{1}{\|\Delta\|_2^2}\left|\frac{1}{n} \sum_{i=1}^n \varphi_{\taubar
  \|\Delta\|_2} (x_i^T \Delta) \cdot \gamma_T(x_i^T \beta) -
\E\left[\varphi_{\taubar \|\Delta\|_2} (x_i^T \Delta) \, \gamma_T(x_i^T
  \beta)\right]\right|.
\end{align*}
For any pair $(\beta, \Delta) \in \Aregion_\delta$, we have
\begin{align*}
	\E\big[(x_i^T \Delta)^2 & - \varphi_{\taubar \|\Delta\|_2} (x_i^T \Delta) \cdot \gamma_T(x_i^T \beta)\big] \\
	& \le \E\left[(x_i^T \Delta)^2 \Ind\left\{|x_i^T \Delta| \ge \frac{\taubar \|\Delta\|_2}{2}\right\}\right] + \E\left[(x_i^T \Delta)^2 \Ind\left\{|x_i^T \beta| \ge \frac{T}{2}\right\}\right] \\
	& \le \sqrt{\E\left[(x_i^T \Delta)^4\right]} \cdot \left(\sqrt{\mprob\left(|x_i^T \Delta| \ge \frac{\taubar \|\Delta\|_2}{2}\right)} + \sqrt{\mprob\left(|x_i^T \beta| \ge \frac{T}{2}\right)}\right) \\
	& \le \sigma_x^2 \|\Delta\|_2^2 \cdot c \exp\left(-\frac{c' \tau^{'2}}{\sigma_x^2}\right),
\end{align*}
where we have used Cauchy-Schwarz and a tail bound for sub-Gaussians, assuming $\beta \in \ball_2(3)$. It follows that for $\taubar$ chosen such that
\begin{equation}
	\label{EqnChooseTau}
	c\sigma_x^2 \exp\left(-\frac{c'\tau^{'2}}{\sigma_x^2}\right) = \frac{\lambda_{\min}\left(\E[x_i x_i^T]\right)}{2},
\end{equation}
we have the lower bound
\begin{equation}
	\label{EqnMeanBd}
	\E\left[\varphi_{\taubar \|\Delta\|_2}(x_i^T \Delta) \cdot \gamma_T(x_i^T \beta)\right] \ge \frac{\lambda_{\min}\left(\E[x_i x_i^T]\right)}{2} \cdot \|\Delta\|_2^2.
\end{equation}

By construction of $\varphi$, each summand in the expression for $Z(\delta)$ is sandwiched as
\begin{equation*}
	0 \le \frac{1}{\|\Delta\|_2^2} \cdot \varphi_{\taubar \|\Delta\|_2} (x_i^T \Delta) \cdot \gamma_T(x_i^T \beta) \le \frac{\tau^{'2}}{4}.
	\end{equation*}
Consequently, applying the bounded differences inequality yields
\begin{equation}
	\label{EqnAzumaHoeff}
	\mprob\left(Z(\delta) \ge \E[Z(\delta)] + \frac{\lambda_{\min}\left(\E[x_i x_i^T]\right)}{4}\right) \le c_1 \exp(-c_2 n).
\end{equation}
	
Furthermore, by Lemmas~\ref{LemSymmetrization} and~\ref{LemRadGauss}
in Appendix~\ref{AppConcentrate}, we have
\begin{equation}
  \label{EqnComp0}
  \E[Z(\delta)] \le 2\sqrt{\frac{\pi}{2}} \cdot
  \E\left[\sup_{(\beta, \Delta) \in \Aregion_\delta}
    \frac{1}{\|\Delta\|_2^2}\left|\frac{1}{n} \sum_{i=1}^n g_i
    \Big(\varphi_{\taubar \|\Delta\|_2} (x_i^T \Delta) \cdot
    \gamma_T(x_i^T \beta)\Big)\right|\right],
\end{equation}
where the $g_i$'s are i.i.d.\ standard Gaussians. Conditioned on
$\{x_i\}_{i=1}^n$, define the Gaussian processes
\begin{equation*}
  Z_{\beta, \Delta} \defn \frac{1}{\|\Delta\|_2^2} \cdot
  \frac{1}{n} \sum_{i=1}^n g_i \Big(\varphi_{\taubar \|\Delta\|_2}
  (x_i^T \Delta) \cdot \gamma_T(x_i^T \beta)\Big),
\end{equation*}
and note that for pairs $(\beta, \Delta)$ and $(\betatil, \Deltatil)$, we have
\begin{align*}
  \var\left(Z_{\beta, \Delta} - Z_{\betatil, \Deltatil}\right) &
  \le 2 \var\left(Z_{\beta, \Delta} - Z_{\betatil,
    \Delta}\right) + 2 \var\left(Z_{\betatil, \Delta} -
  Z_{\betatil, \Deltatil}\right),
\end{align*}
with
\begin{align*}
\var\left(Z_{\beta, \Delta} - Z_{\betatil, \Delta}\right) & =
\frac{1}{\|\Delta\|_2^4} \cdot \frac{1}{n^2} \sum_{i=1}^n
\varphi_{\taubar \|\Delta\|_2}^2(x_i^T \Delta) \cdot \left(\gamma_T(x_i^T
\beta) - \gamma_T(x_i^T \betatil)\right)^2 \\
& \le \frac{1}{n^2} \sum_{i=1}^n \frac{\tau^{'4}}{16} \cdot \frac{4}{T^2} \left(x_i^T(\beta - \betatil)\right)^2,
\end{align*}
since $\varphi_{\taubar \|\Delta\|_2} \le \frac{\tau^{'2} \|\Delta\|_2^2}{4}$ and $\gamma_T$ is $\frac{2}{T}$-Lipschitz. Similarly, using the homogeneity property
\begin{equation*}
\frac{1}{c^2} \cdot \varphi_{ct} (cu) = \varphi_t(u), \qquad \forall c > 0,
\end{equation*}
and the fact that $\varphi_{\taubar \|\Delta\|_2}$ is $\taubar \|\Delta\|_2$-Lipschitz, we have
\begin{align*}
\var \left(Z_{\betatil, \Delta} \! - \! Z_{\betatil, \Deltatil}\right) & \le
\frac{1}{n^2} \sum_{i=1}^n
\gamma_T^2(x_i^T \betatil) \left(\frac{\varphi_{\taubar \|\Delta\|_2}
(x_i^T \Delta)}{\|\Delta\|_2^2} - \frac{\varphi_{\taubar \|\Deltatil\|_2} (x_i^T
\Deltatil)}{\|\Deltatil\|_2^2}\right)^2 \\
& = \frac{1}{n^2} \sum_{i=1}^n \frac{\gamma_T^2 (x_i^T \betatil)}{\|\Delta\|_2^4} \left(\varphi_{\taubar \|\Delta\|_2}(x_i^T \Delta) - \varphi_{\taubar \|\Delta\|_2} \left(x_i^T \Deltatil \cdot \frac{\|\Delta\|_2}{\|\Deltatil\|_2}\right) \right)^2 \\
& \le \frac{1}{n^2} \sum_{i=1}^n \frac{\tau^{'2}}{\|\Delta\|_2^2}
\left(x_i^T \Delta - x_i^T \Deltatil \cdot \frac{\|\Delta\|_2}{\|\Deltatil\|_2}\right)^2 \\
& = \frac{1}{n^2} \sum_{i=1}^n \tau^{'2} \left(\frac{x_i^T \Delta}{\|\Delta\|_2} - \frac{x_i^T \Deltatil}{\|\Deltatil\|_2}\right)^2.
\end{align*}
Defining the centered Gaussian process
\begin{equation*}
Y_{\beta, \Delta} \defn \frac{\tau^{'2}}{\sqrt{2}T} \cdot \frac{1}{n}
\sum_{i=1}^n \ghat_i \cdot x_i^T \beta + \frac{\sqrt{2}\taubar}{\|\Delta\|_2}
\cdot \frac{1}{n} \sum_{i=1}^n \gtil_i \cdot x_i^T \Delta,
\end{equation*}
where the $\ghat_i$'s and $\gtil_i$'s are independent standard
Gaussians, it follows that
\begin{equation*}
 \var \left(Z_{\beta, \Delta} - Z_{\betatil, \Deltatil}\right) \le
 \var \left(Y_{\beta, \Delta} - Y_{\betatil, \Deltatil}\right).
\end{equation*}
Applying Lemma~\ref{LemGaussComp} in Appendix~\ref{AppConcentrate}, we
then have
\begin{equation}
 \label{EqnComp1}
 \E\left[\sup_{(\beta, \Delta) \in \Aregion_\delta} Z_{\beta,
     \Delta}\right] \le 2 \cdot \E\left[\sup_{(\beta, \Delta) \in
     \Aregion_\delta} Y_{\beta, \Delta}\right].
\end{equation}
Note further (cf.\ p.77 of~\cite{LedTal91}) that
\begin{equation}
\label{EqnComp2}
 \E \left[\sup_{(\beta, \Delta) \in \Aregion_\delta} |Z_{\beta,
     \Delta}|\right] \le \E\left[|Z_{\beta_0, \Delta_0}|\right] + 2
 \E\left[\sup_{(\beta, \Delta) \in \Aregion_\delta} Z_{\beta,
     \Delta}\right],
\end{equation}
for any $(\beta_0, \Delta_0) \in \Aregion_\delta$, and furthermore,
\begin{equation}
\label{EqnComp3}
\E\left[|Z_{\beta_0, \Delta_0}|\right] \le
\sqrt{\frac{2}{\pi}} \cdot \sqrt{\var\left(Z_{\beta_0,
    \Delta_0}\right)} \le c_0 \cdot
\sqrt{\frac{2}{\pi}} \cdot \sqrt{\frac{\tau^{'2}}{4n}}.
\end{equation}
Finally,
\begin{align}
\label{EqnComp4}
\E \left[\sup_{(\beta, \Delta) \in \Aregion_\delta} Y_{\beta,
    \Delta}\right] & \le \frac{\tau^{'2} R}{\sqrt{2}T} \cdot \E
\left[\left\|\frac{1}{n} \sum_{i=1}^n \ghat_i
  x_i\right\|_\infty\right] + \sqrt{2} \taubar \delta \cdot
\E\left[\left\|\frac{1}{n} \sum_{i=1}^n \gtil_i
  x_i\right\|_\infty\right] \notag \\
& \le \frac{c\tau^{'2} R \sigma_x}{T} \sqrt{\frac{\log p}{n}} + c\taubar
\delta \sigma_x \cdot \sqrt{\frac{\log p}{n}},
\end{align}
by Lemma~\ref{LemMaxSubGProd} in
Appendix~\ref{AppConcentrate}. Combining~\eqref{EqnComp0}, \eqref{EqnComp1}, \eqref{EqnComp2},
\eqref{EqnComp3}, and~\eqref{EqnComp4}, we then obtain
\begin{equation}
  \label{EqnCompFinal}
\E[Z(\delta)] \le \frac{c'\tau^{'2} R \sigma_x}{T} \sqrt{\frac{\log
    p}{n}} + c' \taubar \delta \sigma_x \cdot \sqrt{\frac{\log p}{n}}.
\end{equation}
Finally, combining~\eqref{EqnMeanBd},
\eqref{EqnAzumaHoeff}, and~\eqref{EqnCompFinal}, we see that
under the scaling $R \sqrt{\frac{\log p}{n}} \precsim 1$, we have
\begin{align}
\label{EqnCompFinal2}
\frac{1}{\|\Delta\|_2^2} \cdot \frac{1}{n} \sum_{i=1}^n &
\varphi_{\taubar \|\Delta\|_2} (x_i^T \Delta) \cdot
\gamma_T(x_i^T \beta) \notag \\
& \ge \frac{\lambda_{\min}\left(\E[x_i
    x_i^T]\right)}{4} - \left(\frac{c'\tau^{'2} R\sigma_x}{T}
\sqrt{\frac{\log p}{n}} + c' \taubar \delta \sigma_x
\sqrt{\frac{\log p}{n}}\right) \notag \\
& \ge \frac{\lambda_{\min}\left(\E[x_i x_i^T]\right)}{8} -
c'\taubar \delta \sigma_x\sqrt{\frac{\log p}{n}},
\end{align}
uniformly over all $(\beta, \Delta) \in \Aregion_\delta$, with
probability at least $1 - c_1 \exp(-c_2 n)$.

It remains to extend this bound to one that is uniform in the ratio
$\frac{\|\Delta\|_1}{\|\Delta\|_2}$, which we do via a peeling
argument~\citep{Alex87, vandeGeer}. Consider the inequality
\begin{align}
\label{EqnHummingBird}
\frac{1}{\|\Delta\|_2^2} \cdot \frac{1}{n} \sum_{i=1}^n \varphi_{\taubar
  \|\Delta\|_2} (x_i^T \Delta) \cdot \gamma_T(x_i^T \beta) & \ge
\frac{\lambda_{\min}\left(\E[x_i x_i^T]\right)}{8} - 2c'\taubar \sigma_x
\frac{\|\Delta\|_1}{\|\Delta\|_2} \sqrt{\frac{\log p}{n}},
\end{align}
as well as the event
\begin{align*}
\Annoying \defn \Bigg\{ \mbox{Inequality~\eqref{EqnHummingBird} holds
  $\forall \|\beta\|_2 \le 3$ and $\frac{\|\Delta\|_1}{\|\Delta\|_2}
  \le \frac{\lambda_{\min}\left(\E[x_i
      x_i^T]\right)}{16c'\tau\sigma_x}\sqrt{\frac{n}{\log p}}$}
\Bigg\}.
\end{align*}
Define the function
\begin{equation}
 \label{EqnFdef}
f(\beta, \Delta; X) \defn \frac{\lambda_{\min}\left(\E[x_i
    x_i^T]\right)}{8} - \frac{1}{\|\Delta\|_2^2} \cdot \frac{1}{n}
\sum_{i=1}^n \varphi_{\taubar \|\Delta\|_2}(x_i^T \Delta) \cdot \gamma_T(x_i^T \beta),
\end{equation}
along with
\begin{equation*}
  g(\delta) \defn c' \taubar \sigma_x \delta \sqrt{\frac{\log p}{n}},
  \qquad \text{and} \qquad h(\beta, \Delta) \defn
  \frac{\|\Delta\|_1}{\|\Delta\|_2}.
\end{equation*}
Note that~\eqref{EqnCompFinal2} implies
\begin{equation}
  \label{EqnPeel}
  \mprob\left(\sup_{h(\beta, \Delta) \le \delta} f(\beta, \Delta; X) \ge g(\delta)\right) \le c_1 \exp(-c_2n), \quad \text{for any } \delta > 0,
\end{equation}
where the $\sup$ is also restricted to $\{(\beta, \Delta): \beta \in \ball_2(3) \cap \ball_1(R), \; \Delta \in \ball_2(3)\}$.

Since $\frac{\|\Delta\|_1}{\|\Delta\|_2} \ge 1$, we have
\begin{equation}
\label{EqnMartinIncompetentOutlaw}
  1 \leq \; h(\beta, \Delta) \; \leq \frac{\lambda_{\min}\left(\E [x_i
      x_i^T] \right)}{16 c' \taubar \sigma_x}\sqrt{\frac{\numobs}{\log
      \pdim}},
\end{equation}
over the region of interest. For each integer $m \geq 1$, define the
set
\begin{align*}
  \MartinPingPong_m \defn \left \{ (\beta, \Delta) \, \mid \, 2^{m-1} \mu \le
  g(h(\beta, \Delta)) \le 2^m \mu \right\},
\end{align*}
where $\mu = c' \taubar \sigma_x \sqrt{\frac{\log p}{n}}$. By a union bound, we then have
\begin{align*}
  \mprob(\Annoying^c) & \le \sum_{m=1}^M \mprob\left(\exists (\beta,
  \Delta) \in \MartinPingPong_m \mbox{ s.t. } f(\beta, \Delta; X) \ge
  2g(h(\beta, \Delta))\right),
\end{align*}
where the index $m$ ranges up to \mbox{$M \defn
  \Big\lceil \log\left(c \sqrt{\frac{\numobs}{\log \pdim }}
  \right)\Big \rceil$} over the relevant
region~\eqref{EqnMartinIncompetentOutlaw}.  By the definition~\eqref{EqnFdef}
of $f$, we have
\begin{align*}
  \mprob(\Annoying^c) & \le \sum_{m=1}^M \mprob\left(\sup_{h(\beta,
    \Delta) \le g^{-1}(2^m \mu)} f(\beta, \Delta; X) \ge 2^m \mu\right)
  \; \stackrel{(i)}{\leq} \; M \cdot c_1 \exp(-c_2 n),
\end{align*}
where inequality (i) applies the tail bound~\eqref{EqnPeel}. It
follows that
\begin{equation*}
  \mprob(\Annoying^c) \le c_1 \exp\left(-c_2n +
  \log\log\left(\frac{n}{\log p}\right)\right) \; \leq c_1' \exp \left(
  - c_2' \numobs \right).
\end{equation*}
Multiplying through by $\|\Delta\|_2^2$ then yields the desired
result.

%%%%%%%%%%%%%%%%%%%%%%%%%%%%%%%%%%%%%%%%%%%%%%%%%%%%%%%%%%%%%%%%%%%%%%%%%%%

\section{Auxiliary Results}
\label{AppConcentrate}
 
In this section, we provide some auxiliary results that are useful for
our proofs.  The first lemma concerns symmetrization and
desymmetrization of empirical processes via Rademacher random
variables:
\begin{mylemma} [Lemma 2.3.6 in~\cite{vanWel96}]
	\label{LemSymmetrization}
	Let $\{Z_i\}_{i=1}^n$ be independent zero-mean stochastic processes. Then
	\begin{equation*}
		\frac{1}{2} \E\left[\sup_{t \in T}\left|\sum_{i=1}^n \epsilon_i Z_i(t_i)\right|\right] \! \le \! \E\left[\sup_{t \in T} \left|\sum_{i=1}^n Z_i(t_i)\right|\right] \! \le \! 2 \E\left[\sup_{t \in T}\left|\sum_{i=1}^n \epsilon_i (Z_i(t_i) - \mu_i)\right|\right],
	\end{equation*}
	where the $\epsilon_i$'s are independent Rademacher variables and the functions $\mu_i: \scriptF \rightarrow \real$ are arbitrary.
\end{mylemma}

We also have a useful lemma that bounds the Gaussian
complexity in terms of the Rademacher complexity:

\begin{mylemma} [Lemma 4.5 in~\cite{LedTal91}]
	\label{LemRadGauss}
	Let $Z_1, \dots, Z_n$ be independent stochastic processes. Then
	\begin{equation*}
		\E\left[\sup_{t \in T} \left|\sum_{i=1}^n \epsilon_i
                  Z_i(t_i) \right|\right] \le \sqrt{\frac{\pi}{2}}
                \cdot \E\left[\sup_{t \in T} \left|\sum_{i=1}^n g_i
                  Z_i(t_i)\right|\right],
	\end{equation*}
	where the $\epsilon_i$'s are Rademacher variables and the
        $g_i$'s are standard normal.
\end{mylemma}

\noindent We next state a version of the Sudakov-Fernique comparison
inequality:

\begin{mylemma} [Corollary 3.14 in~\cite{LedTal91}]
\label{LemGaussComp}
Given a countable index set $T$, let $\{ X(t), t \in T \}$ and $\{
Y(t), t \in T \}$ be centered Gaussian processes such that
\begin{equation*}
  \var\left(Y(s) - Y(t)\right) \le \var\left(X(s) -
  X(t)\right), \qquad \forall (s,t) \in T \times T.
\end{equation*}
Then
\begin{equation*}
  \E\left[\sup_{t \in T} Y(t)\right] \le 2 \cdot
  \E\left[\sup_{t \in T} X(t)\right].
\end{equation*}
\end{mylemma}

A zero-mean random variable $Z$ is sub-Gaussian with parameter
$\sigma$ if $\mprob(Z > t) \leq \exp(- \frac{t^2}{2 \sigma^2})$ for
all $t \geq 0$.  The next lemma provides a standard bound on the
expected maximum of $N$ such variables (cf.\ equation (3.6) in~\cite{LedTal91}):

\begin{mylemma}
  \label{LemSubGMax}
Suppose $X_1, \dots, X_N$ are zero-mean sub-Gaussian random variables
such that $\max \limits_{j=1, \ldots, N} \|X_j\|_{\psi_2} \le
\sigma$. Then $\E\left[\max \limits_{j=1, \ldots, \pdim} |X_j| \right]
\le c_0 \, \sigma \sqrt{\log N}$, where $c_0 > 0$ is a universal
constant.
\end{mylemma}

We also have a lemma about maxima of products of sub-Gaussian
variables:

\begin{mylemma}
\label{LemMaxSubGProd}
Suppose $\{g_i\}_{i=1}^n$ are i.i.d.\ standard Gaussians and
$\{X_i\}_{i=1}^n \subseteq \real^p$ are i.i.d.\ sub-Gaussian vectors
with parameter bounded by $\sigma_x$. Then as long as $\numobs \geq c
\sqrt{\log \pdim}$ for some constant $c > 0$, we have
\begin{equation*}
\E\left[\left\|\frac{1}{n} \sum_{i=1}^\numobs g_i X_i\right\|_\infty
  \right] \le c' \sigma_x \sqrt{\frac{\log \pdim}{\numobs}}.
\end{equation*}
\end{mylemma}

\begin{proof}
Conditioned on $\{X_i\}_{i=1}^n$, for each $j = 1, \ldots, \pdim$, the
variable $\left|\frac{1}{n} \sum_{i=1}^n g_i X_{ij}\right|$ is
zero-mean and sub-Gaussian with parameter bounded by
$\frac{\sigma_x}{\numobs} \, \sqrt{\sum_{i=1}^\numobs
  X_{ij}^2}$. Hence, by Lemma~\ref{LemSubGMax}, we have
\begin{equation*}
\E\left[\left\|\frac{1}{n} \sum_{i=1}^\numobs g_i X_i\right\|_\infty
  \Bigg | X\right] \le \frac{c_0 \sigma_x}{\numobs} \cdot \max_{j=1,
  \ldots, \pdim} \sqrt{\sum_{i=1}^\numobs X_{ij}^2} \cdot \sqrt{\log
  \pdim},
 \end{equation*} 
implying that
\begin{equation}
  \label{EqnCookieMonster}
\E\left[\left\|\frac{1}{n} \sum_{i=1}^n g_i X_i\right\|_\infty \right]
\le c_0 \sigma_x \sqrt{\frac{\log p}{n}} \cdot \E\left[\max_j
  \sqrt{\frac{\sum_{i=1}^n X_{ij}^2}{n}}\right].
\end{equation}
Furthermore, $Z_j \defn \frac{\sum_{i=1}^n X_{ij}^2}{n}$ is an i.i.d.\ average of subexponential variables, each with parameter
bounded by $c \sigma_x$. Since $\Exs[Z_j] \leq 2
\sigma_x^2$, we have
\begin{align}
\label{EqnZbound}
\mprob \left (Z_j - \Exs[Z_j] \ge u + 2 \sigma_x^2 \right) \le c_1
\exp \left(- \frac{c_2 \numobs u}{\sigma_x} \right), \qquad \mbox{$\forall u \geq 0$ and $1 \le j \le p$.}
\end{align}
Now fix some $t \geq \sqrt{2 \sigma_x^2}$.  Since the
$\{Z_j\}_{j=1}^\pdim$ are all nonnegative, we have
\begin{align*}
\E \left[\max_{j=1, \ldots, \pdim} \sqrt{Z_j} \right] & \leq t +
\int_t^\infty \mprob\left(\max_{j=1, \ldots, \pdim} \sqrt{Z_j} > s
\right) ds \\
& \le t + \sum_{j=1}^p \int_t^\infty \mprob \left(\sqrt{Z_j} > s
\right) ds \\
& \leq t + c_1 \pdim \int_t^\infty \exp\left(-\frac{c_2 n (s^2 - 2
  \sigma_x^2)}{\sigma_x}\right) ds
\end{align*}
where the final inequality follows from the bound~\eqref{EqnZbound}
with $u = s^2 - 2 \sigma_x^2$, valid as long as $s^2 \geq t^2 \geq 2
\sigma_x^2$. Integrating, we have the bound
\begin{align*}
\E \left[\max_{j=1, \ldots, \pdim} \sqrt{Z_j} \right] & \leq t + c_1'
p\sigma_x \exp\left(-\frac{c_2' \numobs (t^2 - 2
  \sigma_x^2)}{\sigma_x^2}\right).
\end{align*}
Since $\numobs \succsim \sqrt{\log p}$ by assumption, setting $t$ equal to
a constant implies $\E \left[ \max_j \sqrt{Z_j} \right] = \order(1)$,
which combined with~\eqref{EqnCookieMonster} gives the
desired result.
\end{proof}

%%%%%%%%%%%%%%%%%%%%%%%%%%%%%%%%%%%%%%%%%%%%%%%%%%%%%%%%%%%%%%%%%%%%%%%%%%%%%%%%

\section{Capped-$\ell_1$ Penalty}
\label{AppCapped}

In this section, we show how our results on nonconvex but
subdifferentiable regularizers may be extended to include certain
types of more complicated regularizers that do not possess
(sub)gradients everywhere, such as the capped-$\ell_1$ penalty.

In order to handle the case when $\rho_\lambda$ has points where
neither a gradient nor subderivative exists, we assume the existence
of a function $\rhotil_\lambda$ (possibly defined according to the
particular local optimum $\betatil$ of interest), such that the
following conditions hold:

\begin{assumption}
	\label{AsRhotil}
	\mbox{}
\begin{enumerate}
	\item[(i)] The function $\rhotil_\lambda$ is
          differentiable/subdifferentiable everywhere, and $\|\nabla
          \rhotil_\lambda(\betatil)\|_\infty \le \lambda L$.
	\item[(ii)] For all $\beta \in \real^p$, we have
          $\rhotil_\lambda(\beta) \ge \rho_\lambda(\beta)$.
	\item[(iii)] The equality $\rhotil_\lambda(\betatil) = \rho_\lambda(\betatil)$ holds.
	\item[(iv)] There exists $\mu_1 \ge 0$ such that $\rhotil_\lambda(\beta) + \frac{\mu_1}{2} \|\beta\|_2^2$ is convex.
	\item[(v)] For some index set $A$ with $|A| \le k$ and some parameter $\mu_2 \ge 0$, we have
	\begin{equation*}
		\rhotil_\lambda(\betastar) - \rhotil_\lambda(\betatil)
                \le \lambda L \|\betatil_A - \betastar_A\|_1 - \lambda
                L \|\betatil_{A^c} - \betastar_{A^c}\|_1 + \frac{\mu_2}{2}
                \|\betatil - \betastar\|_2^2.
	\end{equation*}
\end{enumerate}
\end{assumption}

\noindent In addition, we assume conditions (i)--(iii) of
Assumption~\ref{AsRho} in Section~\ref{SecNonconvexRegExas} above.

When $\rho_\lambda(\beta) + \frac{\mu_1}{2} \|\beta\|_2^2$ is convex
for some $\mu_1 \ge 0$ (as in the case of SCAD or MCP), we may take
$\rhotil_\lambda = \rho_\lambda$ and $\mu_2 = 0$ (cf.\
Lemma~\ref{LemEll1Reg} in Appendix~\ref{AppGeneralProps}). When no
such convexification of $\rho_\lambda$ exists (as in the case of the
capped-$\ell_1$ penalty), we instead construct a separate convex
function $\rhotil_\lambda$ to upper-bound $\rho_\lambda$ and take
$\mu_1 = 0$.

Under the conditions of Assumption~\ref{AsRhotil}, we have the
following variant of Theorems~\ref{TheoEll12Meta}
and~\ref{TheoPredMeta}:

\begin{mytheorem}
\label{TheoEll12Meta2}
Suppose $\Loss_n$ satisfies the RSC conditions~\eqref{EqnRSC}, and the
functions $\myrho$ and $\rhotil_\lambda$ satisfy
Assumption~\ref{AsRho} and Assumption~\ref{AsRhotil},
respectively. Suppose $\lambda$ is chosen according to the
bound~\eqref{EqnLambdaChoice} and $n \ge \frac{16R^2
  \max(\tau_1^2, \tau_2^2)}{\alpha_2^2} \log p$. Then for any stationary point $\betatil$ of the program~\eqref{EqnNonconvexReg}, we have
\begin{equation*}
 \|\betatil - \betastar\|_2 \le \frac{7 \lambda L \sqrt{k}}{4\alpha_1 -
   2\mu_1 - 2\mu_2}, \qquad \mbox{and} \qquad \|\betatil -
 \betastar\|_1 \le \frac{28 \lambda L k}{2\alpha_1 - \mu_1 - \mu_2},
	\end{equation*}
	along with the prediction error bound
	\begin{equation*}
		\inprod{\nabla \Loss_n(\betatil) - \nabla \Loss_n(\betastar)}{\nutil} \le \lambda^2 L^2 k \left(\frac{21}{8\alpha_1 - 4\mu_1 - 4\mu_2)} + \frac{49(\mu_1 + \mu_2)}{8(2\alpha_1 - \mu_1 - \mu_2)^2}\right).
	\end{equation*}
\end{mytheorem}
\begin{proof}
	
The proof is essentially the same as the proofs of
Theorems~\ref{TheoEll12Meta} and~\ref{TheoPredMeta}, so we only
mention a few key modifications here.  First note that any local
minimum $\betatil$ of the program~\eqref{EqnNonconvexReg} is a local
minimum of $\Loss_n + \rhotil_\lambda$, since
\begin{equation*}
\EmpLoss(\betatil) + \rhotil_\lambda(\betatil) = \EmpLoss(\betatil) +
\rho_\lambda(\betatil) \leq \EmpLoss(\beta) + \rho_\lambda(\beta) \le
\EmpLoss(\beta) + \rhotil_\lambda(\beta),
\end{equation*}
locally for all $\beta$ in the constraint set, where the first
inequality comes from the fact that $\betatil$ is a local minimum of
$\Loss_n + \rho_\lambda$, and the second inequality holds because
$\rhotil_\lambda$ upper-bounds $\rho_\lambda$. Hence, the first-order
condition~\eqref{EqnFirstOrder} still holds with $\rho_\lambda$
replaced by $\rhotil_\lambda$. Consequently, \eqref{EqnPlug} holds, as well.

Next, note that~\eqref{EqnRhotilConvex} holds as before,
with $\rho_\lambda$ replaced by $\rhotil_\lambda$ and $\mu$ replaced
by $\mu_1$. By condition (v) on $\rhotil_\lambda$, we then have~\eqref{EqnPrePreCone} with $\mu$ replaced by $\mu_1 +
\mu_2$. The remainder of the proof is exactly as before.
\end{proof}

Specializing now to the case of the capped-$\ell_1$ penalty, we have
the following lemma. For a fixed parameter $c \ge 1$, the
capped-$\ell_1$ penalty~\citep{ZhaZha12} is given by
\begin{align}
\label{EqnCapped}
\myrho(t) & \defn \min\left\{\frac{\lambda^2 c}{2}, \; \lambda
|t|\right\}.
\end{align}

\begin{mylemma}
\label{LemCapped}
The capped-$\ell_1$ regularizer~\eqref{EqnCapped} with parameter $c$
satisfies the conditions of Assumption~\ref{AsRhotil}, with $\mu_1 = 0$, $\mu_2 = \frac{1}{c}$, and $L = 1$.
\end{mylemma}
	
\begin{proof}
We will show how to construct an appropriate choice of
$\rhotil_\lambda$. Note that $\myrho$ is piecewise linear and locally
equal to $|t|$ in the range $\left[-\frac{\lambda c}{2}, \frac{\lambda
    c}{2}\right]$, and takes on a constant value outside that
region. However, $\myrho$ does not have either a gradient or
subgradient at $t = \pm \frac{\lambda c}{2}$, hence is not
``convexifiable'' by adding a squared-$\ell_2$ term.

We begin by defining the function $\rhotil: \real \rightarrow \real$
via
\begin{equation*}
\rhotil_\lambda(t) =
\begin{cases}
\lambda |t|, & \mbox{if} \quad |t| \le \frac{\lambda c}{2}, \\
\frac{\lambda^2 c}{2}, & \mbox{if} \quad |t| > \frac{\lambda c}{2}.
\end{cases}
\end{equation*}
For a fixed local optimum $\betatil$, note that we have
$\rhotil_\lambda(\beta) = \sum_{j \in T} \lambda |\betatil_j| +
\sum_{j \in T^c} \frac{\lambda^2 c}{2}$, where \mbox{$T \defn \left\{j
  \, \mid |\betatil_j| \le \frac{\lambda c}{2}\right\}$.} Clearly, $\rhotil_\lambda$ is a convex upper bound on
$\rho_\lambda$, with \mbox{$\rhotil_\lambda(\betatil) =
  \rho_\lambda(\betatil)$.} Furthermore, by the convexity of
$\rhotil_\lambda$, we have
\begin{align}
\label{EqnPolingBicyling}
\inprod{\nabla \rhotil_\lambda(\betatil)}{\betastar - \betatil} & \le
\rhotil_\lambda(\betastar) - \rhotil_\lambda(\betatil) \; = \; \sum_{j
  \in S} \left(\rhotil_\lambda(\betastar_j) -
\rhotil_\lambda(\betatil_j)\right) - \sum_{j \notin S}
\rhotil_\lambda(\betatil_j),
\end{align}
using decomposability of $\rhotil$.  For $j \in T$, we have
\begin{equation*}
	\rhotil_\lambda(\betastar_j) - \rhotil_\lambda(\betatil_j) \le \lambda
|\betastar_j| - \lambda |\betatil_j| \le \lambda |\nutil_j|,
\end{equation*}
whereas
for $j \notin T$, we have $\rhotil_\lambda(\betastar_j) -
\rhotil_\lambda(\betatil_j) = 0 \le \lambda |\nutil_j|$.  Combined
with the bound~\eqref{EqnPolingBicyling}, we obtain
\begin{align}
\label{EqnRhotilGrad}
\inprod{\nabla \rhotil_\lambda(\betatil)}{\betastar - \betatil} & \le
\sum_{j \in S} \lambda |\nutil_j| - \sum_{j \notin S}
\rhotil_\lambda(\betatil_j) \nonumber \\
& = \lambda \|\nutil_S\|_1 - \sum_{j
  \notin S} \myrho(\betatil_j) \nonumber \\
& = \lambda \|\nutil_S\|_1 - \lambda \|\nutil_{S^c}\|_1 + \sum_{j
  \notin S} \left(\lambda |\betatil_j| - \myrho(\betatil_j)\right).
\end{align}
Now observe that
\begin{equation*}
\lambda |t| - \myrho(t) = \begin{cases} 0, & \quad \text{if} \quad |t|
  \le \frac{\lambda c}{2}, \\
	\lambda |t| - \frac{\lambda^2 c}{2}, & \quad \text{if} \quad |t| > \frac{\lambda c}{2},
\end{cases}
\end{equation*}
and moreover, the derivative of $\frac{t^2}{c}$ always exceeds
$\lambda$ for $|t| > \frac{\lambda c}{2}$.  Consequently, we have
$\lambda |t| - \myrho(t) \le \frac{t^2}{c}$ for all $t \in \real$.
Substituting this bound into~\eqref{EqnRhotilGrad} yields
\begin{equation*}
\inprod{\nabla \rhotil_\lambda(\betatil)}{\betastar - \betatil} \le
\lambda \|\nutil_S\|_1 - \lambda \|\nutil_{S^c}\|_1 + \frac{1}{c}
\|\nutil_{S^c}\|_2^2,
\end{equation*}
which is condition (v) of Assumption~\ref{AsRhotil} on
$\rhotil_\lambda$ with $L = 1$, $A = S$, and $\mu_2 =
\frac{1}{c}$. The remaining conditions are easy to verify (see also~\cite{ZhaZha12}).
\end{proof}

%%%%%%%%%%%%%%%%%%%%%%%%%%%

\bibliography{refs.bib}

\begin{thebibliography}{39}
\providecommand{\natexlab}[1]{#1}
\providecommand{\url}[1]{\texttt{#1}}
\expandafter\ifx\csname urlstyle\endcsname\relax
  \providecommand{\doi}[1]{doi: #1}\else
  \providecommand{\doi}{doi: \begingroup \urlstyle{rm}\Url}\fi

\bibitem[Agarwal et~al.(2012)Agarwal, Negahban, and Wainwright]{AgaEtal12}
A.~Agarwal, S.~Negahban, and M.~J. Wainwright.
\newblock Fast global convergence of gradient methods for high-dimensional
  statistical recovery.
\newblock \emph{Annals of Statistics}, 40\penalty0 (5):\penalty0 2452--2482,
  2012.

\bibitem[Alexander(1987)]{Alex87}
K.~S. Alexander.
\newblock Rates of growth and sample moduli for weighted empirical processes
  indexed by sets.
\newblock \emph{Probability {T}heory and {R}elated {F}ields}, 75:\penalty0
  379--423, 1987.

\bibitem[Banerjee et~al.(2008)Banerjee, Ghaoui, and d'Aspremont]{BanEtal08}
O.~Banerjee, L.~El Ghaoui, and A.~d'Aspremont.
\newblock Model selection through sparse maximum likelihood estimation for
  multivariate {G}aussian or binary data.
\newblock \emph{Journal of Machine Learning Research}, 9:\penalty0 485--516,
  2008.

\bibitem[Bertsekas(1999)]{Ber99}
D.~P. Bertsekas.
\newblock \emph{Nonlinear Programming}.
\newblock Athena Scientific, Belmont, MA, 1999.

\bibitem[Bickel et~al.(2009)Bickel, Ritov, and Tsybakov]{BicEtal08}
P.~J. Bickel, Y.~Ritov, and A.~Tsybakov.
\newblock Simultaneous analysis of {L}asso and {D}antzig selector.
\newblock \emph{Annals of Statistics}, 37\penalty0 (4):\penalty0 1705--1732,
  2009.

\bibitem[Breheny and Huang(2011)]{BreHua11}
P.~Breheny and J.~Huang.
\newblock Coordinate descent algorithms for nonconvex penalized regression,
  with applications to biological feature selection.
\newblock \emph{Annals of Applied Statistics}, 5\penalty0 (1):\penalty0
  232--253, 2011.

\bibitem[Carroll et~al.(1995)Carroll, Ruppert, and Stefanski]{CarEtal95}
R.~J. Carroll, D.~Ruppert, and L.~A. Stefanski.
\newblock \emph{Measurement Error in Nonlinear Models}.
\newblock Chapman and Hall, 1995.

\bibitem[Chen and Gu(2014)]{CheGu13}
L.~Chen and Y.~Gu.
\newblock The convergence guarantees of a non-convex approach for sparse
  recovery.
\newblock \emph{{IEEE} Transactions on Signal Processing}, 62\penalty0
  (15):\penalty0 3754--3767, 2014.

\bibitem[Fan and Li(2001)]{FanLi01}
J.~Fan and R.~Li.
\newblock Variable selection via nonconcave penalized likelihood and its oracle
  properties.
\newblock \emph{Journal of the American Statistical Association}, 96:\penalty0
  1348--1360, 2001.

\bibitem[Fan et~al.(2009)Fan, Feng, and Wu]{FanEtal09}
J.~Fan, Y.~Feng, and Y.~Wu.
\newblock Network exploration via the adaptive {LASSO} and {SCAD} penalties.
\newblock \emph{Annals of Applied Statistics}, pages 521--541, 2009.

\bibitem[Fan et~al.(2014)Fan, Xue, and Zou]{FanEtal13}
J.~Fan, L.~Xue, and H.~Zou.
\newblock Strong oracle optimality of folded concave penalized estimation.
\newblock \emph{The Annals of Statistics}, 42\penalty0 (3):\penalty0 819--849,
  06 2014.

\bibitem[Friedman et~al.(2008)Friedman, Hastie, and Tibshirani]{FriEtal08}
J.~Friedman, T.~Hastie, and R.~Tibshirani.
\newblock Sparse inverse covariance estimation with the graphical {L}asso.
\newblock \emph{Biostatistics}, 9\penalty0 (3):\penalty0 432--441, July 2008.

\bibitem[Horn and Johnson(1990)]{HorJoh90}
R.~A. Horn and C.~R. Johnson.
\newblock \emph{Matrix Analysis}.
\newblock Cambridge University Press, 1990.

\bibitem[Hunter and Li(2005)]{HunLi05}
D.~R. Hunter and R.~Li.
\newblock Variable selection using {MM} algorithms.
\newblock \emph{Annals of Statistics}, 33\penalty0 (4):\penalty0 1617--1642,
  2005.

\bibitem[Ledoux and Talagrand(1991)]{LedTal91}
M.~Ledoux and M.~Talagrand.
\newblock \emph{Probability in {B}anach Spaces: {I}soperimetry and Processes}.
\newblock Springer-Verlag, New York, NY, 1991.

\bibitem[Lehmann and Casella(1998)]{LehCas98}
E.~L. Lehmann and G.~Casella.
\newblock \emph{{Theory of Point Estimation}}.
\newblock Springer Verlag, 1998.

\bibitem[Loh and Wainwright(2012)]{LohWai11a}
P.~Loh and M.~J. Wainwright.
\newblock High-dimensional regression with noisy and missing data: {P}rovable
  guarantees with non-convexity.
\newblock \emph{Annals of Statistics}, 40\penalty0 (3):\penalty0 1637--1664,
  2012.

\bibitem[Loh and Wainwright(2013{\natexlab{a}})]{LohWai12}
P.~Loh and M.~J. Wainwright.
\newblock Structure estimation for discrete graphical models: {G}eneralized
  covariance matrices and their inverses.
\newblock \emph{The Annals of Statistics}, 41\penalty0 (6):\penalty0
  3022--3049, 12 2013{\natexlab{a}}.

\bibitem[Loh and Wainwright(2013{\natexlab{b}})]{LohWai13}
P.~Loh and M.~J. Wainwright.
\newblock Regularized {$M$}-estimators with nonconvexity: Statistical and
  algorithmic theory for local optima.
\newblock \emph{arXiv e-prints}, May 2013{\natexlab{b}}.
\newblock Available at \url{http://arxiv.org/abs/1305.2436}.

\bibitem[Loh and Wainwright(2013{\natexlab{c}})]{LohWai13NIPS}
P.~Loh and M.~J. Wainwright.
\newblock Regularized {$M$}-estimators with nonconvexity: {S}tatistical and
  algorithmic theory for local optima.
\newblock In \emph{NIPS}, pages 476--484, 2013{\natexlab{c}}.

\bibitem[Mazumder et~al.(2011)Mazumder, Friedman, and Hastie]{MazEtal11}
R.~Mazumder, J.~H. Friedman, and T.~Hastie.
\newblock Sparsenet: {C}oordinate descent with nonconvex penalties.
\newblock \emph{Journal of the American Statistical Association}, 106\penalty0
  (495):\penalty0 1125--1138, 2011.

\bibitem[McCullagh and Nelder(1989)]{McCNel89}
P.~McCullagh and J.~A. Nelder.
\newblock \emph{Generalized Linear Models (Second Edition)}.
\newblock London: Chapman \& Hall, 1989.

\bibitem[Negahban et~al.(2012)Negahban, Ravikumar, Wainwright, and
  Yu]{NegRavWaiYu12}
S.~Negahban, P.~Ravikumar, M.~J. Wainwright, and B.~Yu.
\newblock A unified framework for high-dimensional analysis of {$M$}-estimators
  with decomposable regularizers.
\newblock \emph{Statistical Science}, 27\penalty0 (4):\penalty0 538--557,
  December 2012.
\newblock See arXiv version for lemma/propositions cited here.

\bibitem[Nesterov(2007)]{Nes07}
Y.~Nesterov.
\newblock Gradient methods for minimizing composite objective function.
\newblock CORE Discussion Papers 2007076, Université Catholique de Louvain,
  Center for Operations Research and Econometrics (CORE), 2007.
\newblock URL \url{http://EconPapers.repec.org/RePEc:cor:louvco:2007076}.

\bibitem[Nesterov and Nemirovskii(1987)]{NesNem87}
Y.~Nesterov and A.~Nemirovskii.
\newblock \emph{Interior Point Polynomial Algorithms in Convex Programming}.
\newblock SIAM studies in applied and numerical mathematics. Society for
  Industrial and Applied Mathematics, 1987.

\bibitem[Raskutti et~al.(2011)Raskutti, Wainwright, and Yu]{RasEtal11}
G.~Raskutti, M.~J. Wainwright, and B.~Yu.
\newblock Minimax rates of estimation for high-dimensional linear regression
  over $\ell_q$-balls.
\newblock \emph{IEEE Transactions on Information Theory}, 57\penalty0
  (10):\penalty0 6976--6994, 2011.

\bibitem[Ravikumar et~al.(2011)Ravikumar, Wainwright, Raskutti, and
  Yu]{RavEtal11}
P.~Ravikumar, M.~J. Wainwright, G.~Raskutti, and B.~Yu.
\newblock High-dimensional covariance estimation by minimizing
  $\ell_1$-penalized log-determinant divergence.
\newblock \emph{Electronic Journal of Statistics}, 4:\penalty0 935--980, 2011.

\bibitem[Rosenbaum and Tsybakov(2010)]{RosTsy10}
M.~Rosenbaum and A.~B. Tsybakov.
\newblock Sparse recovery under matrix uncertainty.
\newblock \emph{Annals of Statistics}, 38:\penalty0 2620--2651, 2010.

\bibitem[Rothman et~al.(2008)Rothman, Bickel, Levina, and Zhu]{Rot08}
A.~J. Rothman, P.~J. Bickel, E.~Levina, and J.~Zhu.
\newblock Sparse permutation invariant covariance estimation.
\newblock \emph{Electronic Journal of Statistics}, 2:\penalty0 494--515, 2008.

\bibitem[van~de Geer(2000)]{vandeGeer}
S.~van~de Geer.
\newblock \emph{Empirical Processes in M-Estimation}.
\newblock Cambridge University Press, 2000.

\bibitem[van~der Vaart and Wellner(1996)]{vanWel96}
A.~W. van~der Vaart and J.~A. Wellner.
\newblock \emph{Weak Convergence and Empirical Processes}.
\newblock Springer Series in Statistics. Springer-Verlag, New York, 1996.
\newblock With applications to statistics.

\bibitem[Vavasis(1995)]{Vav95}
S.~A. Vavasis.
\newblock Complexity issues in global optimization: A survey.
\newblock In \emph{Handbook of Global Optimization}, pages 27--41. Kluwer,
  1995.

\bibitem[Vial(1982)]{Via82}
J.-P. Vial.
\newblock Strong convexity of sets and functions.
\newblock \emph{Journal of Mathematical Economics}, 9\penalty0 (1-2):\penalty0
  187--205, January 1982.

\bibitem[Wainwright(2014)]{Wai14}
M.~J. Wainwright.
\newblock Structured regularizers for high-dimensional problems: {S}tatistical
  and computational issues.
\newblock \emph{Annual Review of Statistics and Its Application}, 1\penalty0
  (1):\penalty0 233--253, 2014.

\bibitem[Wang et~al.(2014)Wang, Liu, and Zhang]{WanEtal13}
Z.~Wang, H.~Liu, and T.~Zhang.
\newblock Optimal computational and statistical rates of convergence for sparse
  nonconvex learning problems.
\newblock \emph{The Annals of Statistics}, 42\penalty0 (6):\penalty0
  2164--2201, 12 2014.

\bibitem[Yuan and Lin(2007)]{YuaLin07}
M.~Yuan and Y.~Lin.
\newblock Model selection and estimation in the {G}aussian graphical model.
\newblock \emph{Biometrika}, 94\penalty0 (1):\penalty0 19--35, 2007.

\bibitem[Zhang(2010)]{Zha12}
C.-H. Zhang.
\newblock Nearly unbiased variable selection under minimax concave penalty.
\newblock \emph{Annals of Statistics}, 38\penalty0 (2):\penalty0 894--942,
  2010.

\bibitem[Zhang and Zhang(2012)]{ZhaZha12}
C.-H. Zhang and T.~Zhang.
\newblock A general theory of concave regularization for high-dimensional
  sparse estimation problems.
\newblock \emph{Statistical Science}, 27\penalty0 (4):\penalty0 576--593, 2012.

\bibitem[Zou and Li(2008)]{ZouLi08}
H.~Zou and R.~Li.
\newblock One-step sparse estimates in nonconcave penalized likelihood models.
\newblock \emph{Annals of Statistics}, 36\penalty0 (4):\penalty0 1509--1533,
  2008.

\end{thebibliography}

\end{document}